\documentclass[12pt, colorinlistoftodos]{amsart}
\usepackage{amsmath,amssymb, mathtools}
\usepackage{a4wide}
\usepackage{verbatim}
\usepackage[all]{xy}
\usepackage{hyperref}
\newcommand{\note}[1]{\marginpar{\small \sf#1}}

\usepackage{tikz}
\usepackage{xcolor}
\usepackage{dsfont}
\usepackage{todonotes}
\usetikzlibrary{calc}
\usetikzlibrary{decorations.pathreplacing,decorations.markings}
\usetikzlibrary{positioning,arrows,patterns}
\usetikzlibrary{cd}
\usetikzlibrary{intersections}
\usetikzlibrary{angles,quotes}

\numberwithin{equation}{section}

\theoremstyle{plain}
\newtheorem{theorem}{Theorem}[section]
\newtheorem{corollary}[theorem]{Corollary}
\newtheorem{lemma}[theorem]{Lemma}
\newtheorem{proposition}[theorem]{Proposition}

\theoremstyle{definition}

\theoremstyle{remark}
\newtheorem{remark}[theorem]{Remark}

\newcommand{\A}{\mathbb{A}}
\newcommand{\R}{\mathbb{R}}
\newcommand{\Q}{\mathbb{Q}}
\newcommand{\ebf}{\mathbf{e}}
\newcommand{\Z}{\mathbb{Z}}
\newcommand{\N}{\mathbb{N}}
\newcommand{\C}{\mathbb{C}}

\renewcommand{\H}{\mathbb{H}}

\newcommand{\zxz}[4]{\begin{pmatrix} #1 & #2 \\ #3 & #4 \end{pmatrix}}

\newcommand{\kzxz}[4]{\left(\begin{smallmatrix} #1 & #2 \\ #3 & #4\end{smallmatrix}\right) }
\newcommand{\kabcd}{\kzxz{a}{b}{c}{d}}

\newcommand{\calC}{\mathcal{C}}
\newcommand{\calD}{\mathcal{D}}
\newcommand{\calE}{\mathcal{E}}

\newcommand{\calO}{\mathcal{O}}

\newcommand{\calR}{\mathcal{R}}
\newcommand{\calS}{\mathcal{S}}
\newcommand{\calT}{\mathcal{T}}

\newcommand{\calW}{\mathcal{W}}

\newcommand{\frakM}{\mathfrak M}
\newcommand{\frako}{\mathfrak o}

\newcommand{\fraksp}{\mathfrak s \mathfrak p}

\newcommand{\eps}{\varepsilon}
\newcommand{\bs}{\backslash}
\newcommand{\norm}{\operatorname{N}}
\newcommand{\vol}{\operatorname{vol}}
\newcommand{\tr}{\operatorname{tr}}

\newcommand{\Sl}{\operatorname{SL}}
\newcommand{\Gl}{\operatorname{GL}}
\newcommand{\SL}{\operatorname{SL}}
\newcommand{\PSL}{\operatorname{PSL}}
\newcommand{\Sp}{\operatorname{Sp}}

\newcommand{\GSpin}{\operatorname{GSpin}}
\newcommand{\Mp}{\operatorname{Mp}}
\newcommand{\Orth}{\operatorname{O}}
\newcommand{\Uni}{\operatorname{U}}

\newcommand{\Spec}{\operatorname{Spec}}

\newcommand{\Sym}{\operatorname{Sym}}

\newcommand{\GL}{\operatorname{GL}}
\newcommand{\SO}{\operatorname{SO}}

\newcommand{\Lie}{\operatorname{Lie}}

\newcommand{\tQv}{\Upsilon}
\newcommand{\prM}{\operatorname{pr}}
\newcommand{\diag}{\operatorname{diag}}

\newcommand{\cha}{\operatorname{char}}

\newcommand{\wt}{\widetilde}
\renewcommand{\tilde}{\widetilde}
\newcommand{\wh}{\widehat}
\newcommand{\tc}{\tilde{c}}
\newcommand{\hv}{b}

\newcommand{\Pb}{\mathbb{P}}
\newcommand{\p}{p}
\newcommand{\lp}{\left (}
  \newcommand{\rp}{\right )}

\def\be{\begin{equation}}   \def\ee{\end{equation}}     \def\bes{\begin{equation*}}    \def\ees{\end{equation*}}
\def\ba{\be\begin{aligned}} \def\ea{\end{aligned}\ee}   \def\bas{\bes\begin{aligned}}  \def\eas{\end{aligned}\ees}
\def\={\;=\;}  \def\+{\,+\,} 

\newlength{\halfbls}

\usepackage[backend=bibtex,
            style=alphabetic,
            isbn=false,
            doi=false,
            url=false,
            maxnames=5,   
            maxalphanames=5,
            giveninits=true
           ]{biblatex}
\bibliography{my.bib}

\begin{document}

\title[Special cycles on Shimura curves and Siegel Maass forms]{Intersections of special cycles on Shimura curves \\and Siegel Maass forms}

\author[Jan H.~Bruinier]{Jan Hendrik Bruinier}
\thanks{Research of all three authors is supported  
  by the DFG Collaborative Research Centre TRR 326 ``Geometry and Arithmetic of Uniformized Structures'', project number 444845124.
Furthermore, the second author is supported by the Heisenberg Program, project number 539345613.
}
\address{Fachbereich Mathematik,
Technische Universit\"at Darmstadt, Schlossgartenstrasse 7, D--64289
Darmstadt, Germany}
\email{bruinier@mathematik.tu-darmstadt.de}

\author{Yingkun Li}
\address{Department of Mathematics,
  University of Wisconsin--Madison,
  Van Vleck Hall,
480 Lincoln Drive,
Madison, WI 53706, USA}
\email{lykpi@math.wisc.edu}

\author{Martin M\"oller}
\address{
Institut f\"ur Mathematik, Goethe--Universit\"at Frankfurt,
Robert-Mayer-Str. 6--8,
60325 Frankfurt am Main, Germany
}
\email{moeller@math.uni-frankfurt.de}


\date{\today}
\begin{abstract} 
  We show that the generating series of the number of pairs of geodesics on a compact Shimura curve with given discriminants and intersection angle are coefficients of a non-holomorphic Siegel modular form, a theta lift of the constant function. This retrieves and generalizes counting results of Rickards via the Siegel-Weil formula.
  \par
  More generally, we study the genus two theta lift of Maass forms on this Shimura curve and prove a Fourier-Taylor expansion in terms of some generalized Whittaker functions. We also provide a geometric interpretation of all Fourier coefficients of these theta lifts in terms of averages of geodesic Taylor coefficients over special cycles.
\end{abstract}

\maketitle
 \makeatletter
 \providecommand\@dotsep{5}
 \def\listtodoname{List of Todos}
 \def\listoftodos{\@starttoc{tdo}\listtodoname}
 \makeatother
\setcounter{tocdepth}{1}
 
\tableofcontents

\section{Introduction} \label{sec:intro}

Computing the number of closed geodesics and of Heegner points of given discriminant on a Shimura curve and understanding the distribution of these subsets has a long history. An incomplete account starts for the modular curve with Siegel's estimate \cite{Siegel}, continues with the equidistribution results by Linnik \cite{LinnikErgodic} and Duke \cite{DukeHypDist} and includes numerous quantitative versions exploring subconvexity bounds, see e.g.\ the survey \cite{MichelVenkICM}. A key step in the latter approaches is to interpret the sum of the evaluations of a Maass form at Heegner points of discriminant~$D$ as the $D$-th Fourier coefficient of a Maass form of weight~$1/2$, in fact a theta lift, following ideas of Maass and studied in detail in this non-holomorphic context by Katok-Sarnak \cite{KS}.
\par
\medskip
\paragraph{\textbf{Counting pairs of geodesics.}}
Similar questions for the number and distribution of intersection points of \emph{two} geodesic curves with given discriminant are much less explored. (For references to counting and equidistribution sorted by length see `Related results' below.) One of the few results in this direction is the work of Rickards (\cite{Ri}, see also \cite{RiSL} for the case of the modular curve), who computed the number of intersection points by studying pairs of optimal embeddings of quadratic number fields into quaternion algebras. More precisely, under some `admissibility' condition on the discriminants he showed that this number admits a nice expression (that we recall in Theorem~\ref{thm:rickardscount}) in terms of local factors. One of our motivations was to provide an expression in the general case together with an explanation for this factorization.
\par
Let $B$ be an indefinite quaternion algebra over $\Q$ which is not globally split.
Fix an Eichler order~$\calO_\frakM$ of level~$\frakM$ in $B$ and consider the compact Shimura curve $X = \Gamma_\frakM \backslash \H$ where~$\Gamma_\frakM$ are the units of norm one in~$\calO_\frakM$.
Our first main result is that the total number of intersection points of two geodesic curves on $X$ of discriminants $D_1$ and $D_2$ with intersection angle~$\theta$ are essentially the coefficients of a Siegel Maass form~$G$ of genus two and weight~$1/2$. More precisely, this number is encoded in the non-archimedian part of the Fourier coefficient~$G_T$  of index $T$ given a positive definite matrix $T = \kzxz{D_1}{b}{b}{D_2}$ where $b=\cos(\theta)\sqrt{D_1D_2}$. What about the other Fourier coefficients of~$G$? Two non-intersecting geodesics have a unique common perpendicular, whose length is a parameter that can be taken as replacement for the intersection angle. Similarly, given one geodesic and a Heegner point there is unique perpendicular geodesic through the point. Finally, two Heegner points have a unique geodesic segment connecting them, see Figure~\ref{cap:Introfourcases} for pictures of these cases. There is one Siegel Maass form to count them all!\

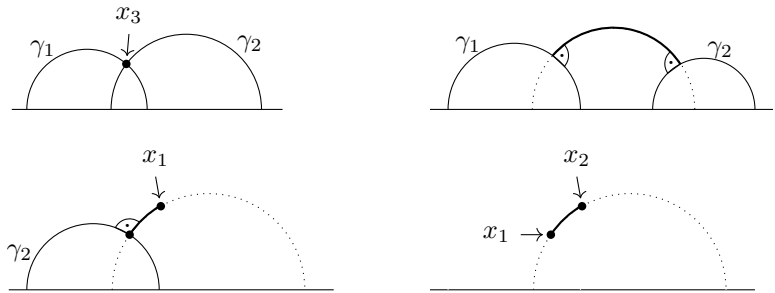
\begin{figure}[htb]
	\centering
\begin{tikzpicture} [scale=.8]
\begin{scope}
\path[draw] (-.25,0) -- (4.25,0);	
\draw [name path=circ 1] (0,0) arc (180:0:1) ++(150:2cm) node {\footnotesize{$\gamma_1$}};
\draw [name path=circ 2] (1.4,0) arc (180:0:1.25) ++(100:1.1cm) node {\footnotesize{$\gamma_2$}};
\fill [name intersections={of=circ 1 and circ 2, by={s}}] (s) circle (2pt);   
\draw[<-,shorten <=3pt] (s) -- ++(85:.5cm) node[above] {\footnotesize{$x_3$}};	
\end{scope}	
\begin{scope} [xshift=7cm]	
	\path[draw] (-.3,0) -- (5.4,0);	
	\draw [name path=circ 1] (0,0) arc (180:0:1.1) ++(150:2.15cm) node {\footnotesize{$\gamma_1$}};
	\draw [name path=circ 2, dotted] (1.4,0) arc (180:0:1.35) ++(127:1.9cm) node {}; 
	\draw [name path=circ 3] (3.4,0) arc (180:0:.85) ++(120:1.15cm) node {\footnotesize{$\gamma_2$}};		
	\path [draw=red,name intersections={of=circ 1 and circ 2, by={s1}}] (s1) circle (0pt);  
	\fill [name intersections={of=circ 2 and circ 3, by={s2}}] (s2) circle (0pt);  	
	\path[name path=s1_arc1] ([shift=(0:8pt)]s1) arc (0:90:8pt);
		\path [name intersections={of=circ 2 and s1_arc1, by={w1}}] (w1); 	
	\path[name path=s1_arc2] ([shift=(-90:8pt)]s1) arc (-90:20:8pt);	
		\path [name intersections={of=circ 1 and s1_arc2, by={w2}}] (w2);  
	\draw (w2) coordinate (A) -- (s1) coordinate (B) -- (w1) coordinate (C)
	 pic () [draw, angle radius=6pt, pic text=.] {angle = A--B--C};	 
	\path[name path=s2_arc1] ([shift=(90:8pt)]s2) arc (90:180:8pt);
		\path [name intersections={of=circ 2 and s2_arc1, by={w3}}] (w3);  	
	\path[name path=s2_arc2] ([shift=(180:8pt)]s2) arc (180:250:8pt);	
	\path [name intersections={of=circ 3 and s2_arc2, by={w4}}] (w4);
	\draw (w3) coordinate (A) -- (s2) coordinate (B) -- (w4) coordinate (C)
pic () [draw, angle radius=6pt, pic text=.] {angle = A--B--C};
\draw [name path=circ 3, thick]  (s1) arc (139.5:33:1.325) ++(100:.95cm) node {};
\end{scope}	
\begin{scope} [yshift = -3cm, xshift=0cm]	
	\path[draw] (-.3,0) -- (5.1,0);	
	\draw [name path=circ 1] (0,0) arc (180:0:1.1) ++(165:2.4cm) node {\footnotesize{$\gamma_2$}};
\draw [name path=circ 2, dotted] (1.42,0) arc (180:0:1.6) ++(90:.9cm) node {}; 
	\fill [name intersections={of=circ 1 and circ 2, by={s1}}] (s1) circle (2pt);  
	\path[name path=s1_arc1] ([shift=(0:6pt)]s1) arc (0:90:6pt);
	\fill [name intersections={of=circ 2 and s1_arc1, by={w1}}] (w1);  	
	\path[name path=s1_arc2] ([shift=(120:6pt)]s1) arc (90:200:6pt);	
	\fill [name intersections={of=circ 1 and s1_arc2, by={w2}}] (w2);
	\draw (w1) coordinate (A) -- (s1) coordinate (B) -- (w2) coordinate (C)
	pic () [draw, angle radius=6pt, pic text=.] {angle = A--B--C};	 
	%
		\path[name path=s3_arc1] ([shift=(20:20pt)]s1) arc (20:60:20pt); 
		\fill [name intersections={of=circ 2 and s3_arc1, by={w3}}] (w3) circle (2pt);  
\draw[<-,shorten <=3pt] (w3) -- ++(100:.5cm) node[above] {\footnotesize{$x_1$}};
\draw [name path=circ 3, thick]  (s1) arc (146:120:1.6) ++(90:.9cm) node {};
\end{scope}	
\begin{scope} [yshift = -3 cm, xshift=7cm]	
	\path[draw] (-.3,0) -- (5.1,0);	
	\draw [name path=circ 1,color = white] (0,0) arc (180:0:1.1) ++(165:2.4cm) node {};
	\draw [name path=circ 2,color = white] (1.42,0) arc (180:0:1.6) ++(90:.9cm) node {};	
	\fill [name intersections={of=circ 1 and circ 2, by={s1}}] (s1) circle (2pt);  
	\draw[<-,shorten <=3pt] (s1) -- ++(0:-.5cm) node[left] {\footnotesize{$x_1$}};  	
	\path[name path=s1_arc1] ([shift=(0:6pt)]s1) arc (0:90:6pt);
	\fill [name intersections={of=circ 2 and s1_arc1, by={w1}}] (w1);  	
	\path[name path=s1_arc2] ([shift=(120:6pt)]s1) arc (90:200:6pt);	
	\fill [name intersections={of=circ 1 and s1_arc2, by={w2}}] (w2);
	%
\path[name path=s3_arc1] ([shift=(20:20pt)]s1) arc (20:60:20pt); 
\fill [name intersections={of=circ 2 and s3_arc1, by={w3}}] (w3) circle (2pt);
\draw[<-,shorten <=3pt] (w3) -- ++(100:.5cm) node[above] {\footnotesize{$x_2$}};
\draw [name path=circ 3, thick]  (s1) arc (146:120:1.6) ++(90:.9cm) node {};
	\draw [name path=circ4, dotted] (1.42,0) arc (180:0:1.6) ++(90:.9cm) node {}; 
\end{scope}	
\end{tikzpicture}
\caption{Two geodesics intersecting (top left) or with a common perpendicular (top right). Geodesic through a point, perpendicular to a geodesic (bottom left) or through a second point (bottom right).}
 \label{cap:Introfourcases}
\end{figure}

\par
\begin{theorem} 
\label{thm:intro1}
Consider the Shimura curve $X$ as above, and let $\tau=u+iv$ be a variable in the Siegel upper half space $\H_2$ of genus $2$.
There is a Siegel Maass form $G(\tau)$ of genus $2$ and weight $1/2$ with the following properties.
The $T$-th Fourier coefficient $G_T(v)$ for~$T = \kzxz{D_1}{b}{b}{D_2}$ is the product $|R_T| \cdot W_{T,\infty}(v)$ where $W_{T,\infty}(v)$ is some explicit 
archimedian Whittaker function given in~\eqref{eq:Whittaker} and where $R_T$ is respectively
\begin{itemize}
\item[i)] if $T$ is positive definite: the set of intersection points of geodesics~$\gamma_i$ of discriminant $D_i$ with intersection angle~$\cos(\theta) = b/\sqrt{D_1D_2}$
\item[ii)] if $T$ is indefinite with $D_1,D_2 >0$: the set of pairs of geodesics~$\gamma_i$ of discriminant $D_i$ with common perpendicular~$\gamma$ with~$\cosh(\ell_{\mathrm{hyp}}(\gamma)) = b/\sqrt{D_1D_2}$
\item[iii)] if $T$ is indefinite with $D_1> 0 > D_2$: the set of pairs consisting of a geodesic~$\gamma_1$ of discriminant~$D_1$ and a Heegner point~$x_2$ of discriminant~$D_2$ such that the segment~$\gamma$ orthogonal to~$\gamma_1$ and through~$x_2$ satisfies $\sinh(\ell_{\mathrm{hyp}}(\gamma)) =  b/\sqrt{-D_1D_2}$
\item[iv)] if $T$ is indefinite with $0> D_1,D_2 $: the set of Heegner points $x_i$ of discriminant~$D_i$ such that the connecting geodesic segment $\gamma$ satisfies
$\cosh(\ell_{\mathrm{hyp}}(\gamma)) = b/\sqrt{D_1D_2}$
\end{itemize}
on the curve~$X$.
\par
The number~$|R_T|$ is a product of local factors, explicitly for $\gcd(D_1,D_2,D_1D_2-b^2)=1$ given in Theorem~\ref{thm:rickardscount} and in the general case in Theorem~\ref{thm:locwhitt3} together with Theorem~\ref{thm:locwhitt4}.
\end{theorem}
\par
Note that the Whittaker function  satisfies the transformation law
\begin{align}
\label{eq:Shimintro}
W_{T,\infty}(av\,{}^ta)&\=\det(a)^{3/2}  \cdot W_{{}^{t}aTa,\infty}(v)
\end{align}
for $a\in \Gl_2^+(\R)$, and for $v$ in the positive definite $(2\times 2)$-matrices,  see~\eqref{eq:i1}. Hence, the index $T$ can always be reduced to a diagonal matrix with diagonal entries $\pm 1$.  
\par
\medskip
\paragraph{\textbf{A theta lift.}}
To prove the theorem we consider the rational quadratic space $V$ given by the trace zero elements of $B$ equipped with the negative of the reduced norm as a quadratic form. The signature of $V$ is $(2,1)$ and its general spin group $\GSpin(V)$ is isomorphic to the group $B^\times$ of invertible elements of $B$. 
The intersection 
\[
L
\=V\cap \calO_{\mathfrak{M}}
\]
defines an even lattice.
The Siegel theta function $\theta_L^{(2)}(\tau,z)$ of genus $2$ associated with the lattice $L$ (and the standard Gaussian on $V(\R)$) is a non-holomorphic automorphic form which is invariant under $\Gamma_{\mathfrak{M}}$ in the `orthogonal' variable $z\in \H$ and a modular form of weight $1/2$ in the `metaplectic' variable $\tau\in \H_2$ for a congruence subgroup $\widetilde \Gamma$ of the metaplectic group of genus~$2$.

For a Maass form $f$  of weight $0$ for the group $\Gamma_{\mathfrak{M}}$ with Laplace eigenvalue~$\lambda$ we consider the theta integral
\begin{align}
\label{eq:thetaintintro}
I(\tau,f) \= \int_{\Gamma_{\mathfrak{M}}\bs \H} f(z) \theta_L^{(2)}(\tau,z)\, d\mu(z),
\end{align}
where $d\mu(z)$ denotes the usual invariant measure on
$\H$. It defines a Siegel Maass form of weight $1/2$ for the group $\widetilde \Gamma$, whose infinitesimal character we determine in Proposition \ref{prop:eigenvals}.

Theorem \ref{thm:intro1} will be proved by carefully analyzing the theta integral in the special case when $f$ is constant. In fact, the Maass form $G$ is given by the theta integral $I(\tau,1)$, which is just the average of the theta function over the Shimura curve $X$ and can be evaluated by means of the Siegel-Weil formula. In this way $G$ can be identified with a specific Siegel Eisenstein series of weight $1/2$ and genus $2$, see Theorem~\ref{thm:sw}. Its Fourier coefficients can be computed in terms of local Whittaker functions, and the quantity~$R_T$ arises as a products of local non-archimedian Whittaker functions in Theorem~\ref{thm:posdeg2}. Employing explicit formulas for these due to Kudla and Yang \cite{Ku1, Ya} we derive the result of Rickards. Moreover, this approach allows for a generalization to more general level structures.
\par
\medskip
\paragraph{\textbf{General input and generalized Whittaker functions.}}
However, this argument involving the Siegel-Weil formula is limited to the case when $f$ is constant. The main part of Section \ref{sec:Thetalift} is devoted to the case of a general  Maass form~$f$. This will also lead to a second proof of (large part of)  Theorem~\ref{thm:intro1}. We now turn to a description of our main result in the general case. It turns out that it requires a finer analysis of the coefficients $F_T(v)$ in the Fourier expansion 
\[
F(\tau) \= \sum_{T\in \Sym_2(\Q)} F_T(v) \,\ebf(\tr(Tu))
\]
of a Siegel Maass form $F$ of genus $2$.
For simplicity, in this exposition we only consider the case when the index matrix $T$ is positive definite, referring to Section~\ref{sect:SM} for general regular $T$.

The special orthogonal group $\SO(T)$ of the positive definite quadratic form $T$ acts on the space of smooth functions $h:\H_2\to \C$ satisfying $h(\tau+b)= \ebf(\tr (Tb))h(\tau)$ and having a fixed infinitesimal character. Consequently, the Fourier coefficient $F_T(v)$ has an expansion 
in terms of the characters of the compact abelian group $\SO(T)$. 
Fix a matrix $B\in \GL_2(\R)$ with $T={}^t BB$. This determines an isomorphism $ \SO_2(\R)\to \SO(T)$,  $a \mapsto {}^t\!B a \, {}^t\!B^{-1}$, and an isomorphism of the  character group of $\SO(T)$ with $\Z$. We obtain an expansion
\[
F_T(v) \= \sum_{n\in \Z} F_{T,n}(v),
\]
where $F_{T,n}(v)$ transforms according to the $n$-th character of $\SO(T)/\{\pm 1\}$, see~\eqref{eq:siegelfou2}. In particular, the $T$-th Fourier coefficient $I_T(v,f)$ of our theta integral~\eqref{eq:thetaintintro} possesses such an expansion, whose coefficients we denote by $I_{T,n}(v,f)$.
\par
In Section~\ref{subsec:posdef} we construct an explicit Whittaker function 
$\calW^{+,+}_{B,n}(v,s)$
for $v$ a positive definite $(2\times 2)$-matrix and $s\in \C$. Here the superscript refers to the case the $T$ has two {\em positive} eigenvalues. In the body of the text, where we consider general $T$, there will be further cases depending on the signature of $T$.
Under right translation by $a\in \Gl_2(\R)$ the function $ \calW^{+,+}_{B,n}(v,s)$ satisfies the transformation law 
\[
\calW^{+,+}_{Ba,n}(v,s) \= \det(a)^{-1}\calW^{+,+}_{B,n}(av\,{}^ta,s), 
\]
analogous to \eqref{eq:Shimintro}.
Under left translation of $B$ by $\SO_2(\R)$ it transforms with the character of index $n$, see Remark~\ref{rem:w2}. Moreover, it is holomorphic in $s$ and invariant under the transformation $s\mapsto 1-s$. The function $\calW^{+,+}_{B,n}(v,s) \ebf(\tr(Tu))$ is an eigenfunction of the invariant differential operators on Siegel Maass forms of weight $1/2$  induced by the Casimir elements of the symplectic group. Similar Whittaker functions were studied by Niwa in weight zero for positive definite $T$ and by Moriyama for integral weights and general $T$, see e.g.~\cite{Niwa}, \cite{Moriyama}.

We prove that the coefficient $I_{T,n}(v,f)$ of the theta integral is a constant multiple of 
$\calW^{+,+}_{B,n}(v,s)$, where we have written the Laplace eigenvalue of $f$ as $\lambda= s(1-s)$.
We show that the constant of proportionality can be described in terms of geodesic Taylor coefficients of $f$ at the intersection points of geodesics on $X$ with characteristic matrix as in Theorem~\ref{thm:intro1}. 
\medskip
\paragraph{\textbf{Geodesic Taylor expansion.}}

To describe this in more detail, recall that the Cartan decomposition of $\SL_2(\R)$ implies that any $z\in \H$ can be written in the form 
$z=k_\theta a_t i$, where 
\[
a_t \= \zxz{e^{t/2}}{0}{0}{e^{-t/2}}, \qquad  k_\theta\= \zxz{\cos\theta}{\sin\theta}{-\sin\theta}{\cos\theta }
\]
with $t\in \R_{\geq 0}$ and $\theta\in [0,\pi)$. The pair $(t,\theta)$ defines the standard geodesic polar coordinates of $z$ centered at $i$, where $t$ is the geodesic distance from $z$ to $i$, and $2\theta$ the angle determined by the geodesic ray $i\R_{\geq 1}$ and the geodesic ray from $i$ to $z$. 
The Maass form $f$ has a geodesic Taylor expansion at the point $z_0=i$ of the form 
\[
f(z) \= \sum_{n\in \Z} f_{n}(z_0) P^{-|n|}_{-s}(\cosh(t))\cdot  e^{2i n\theta},
\]
where $P^{-\mu}_\nu(x)$ denotes the classical Legendre function, and the $n$-th coefficient $f_n(z_0)$ at $z_0$ is given by Proposition \ref{prop:Fay0}, see also~\cite{Fay}.
More generally, if $z_1\in \H$ is any point, we choose a translation matrix $g\in \SL_2(\R)$ such that $z_1= gi$. We obtain geodesic polar coordinates centered at $z_1$ by writing $z= g k_\theta a_t i$. Then $t$ denotes again the hyperbolic distance from $z$ to $z_1$, and $2\theta$ the angle determined by the geodesic ray $g(i\R_{\geq 1})$ and the geodesic ray from $z_1$ to $z$. Note that the angle actually depends on the choice  of $g$, not only on $z_1$. The geodesic Taylor expansion of $f$ at the point $z_1$ (with respect to $g$) is defined as the geodesic Taylor expansion of $f\mid g$ at $i$. The corresponding Taylor coefficients are denoted by $f_n(z_1)$.
Our main result for positive definite $T$ is now as follows.

\begin{theorem}
\label{thm:evaleintro}
Let the notation be as above. 
The $n$-th coefficient of $I_T(v,f)$ is given by
\begin{align*}
I_{T,n}(v,f) &\=  \calW^{+,+}_{B,n}(v,s)\cdot 
\sum_{x \in R_T} f_{-2n}(z_x).
\end{align*}
Here, for every $x\in R_T$ we have written $z_x$ for the Heegner point determined by the intersection of the two geodesics defined by $x$. Moreover, the geodesic polar coordinates at $z_x$ are defined with respect to a specific translation matrix $g_x$ (depending on the choice of $B$), see Theorem \ref{thm:evale}
for details.
In particular, the $0$-th coefficient is given by the average value
\begin{align*}
I_{T,0}(v,f) &\=  \calW^{+,+}_{B,0}(v,s)\cdot \sum_{x \in R_T} f(z_x).
\end{align*}
\end{theorem}
\par
\begin{remark} \label{rem:introIT}
i) 
An analogous result also holds in the case when $T$ is indefinite, corresponding to the cases (ii)--(iv) of Theorem \ref{thm:intro1}. Then the appropriate coordinates are given by the distance of a point $z\in \H$ from the geodesic $\gamma$ appearing in these cases, and the intersection point of~$\gamma$ with the geodesic defining the distance. The quotient of~$\gamma$ by its stabilizer in $\Gamma_\frakM$ is isomorphic to $S^1$, and consequently there is again an expansion of $f$ in terms of the corresponding characters with coefficients depending on the distance. The analogue of Theorem \ref{thm:evaleintro} involves twisted cycle integrals of~$f$ over the image of~$\gamma$ in~$X$ and an archimedian Whittaker function $\calW^{+,-}_{B, n}(v, s)$, see Theorem~\ref{thm:W-indef} for details. 

ii) The coefficients $I_T(v,f)$ for singular $T$ turn out to be given by the average values of~$f$ over Heegner points or geodesic cycles of fixed discriminant (determined by the trace of $T$), 
 see Theorems~\ref{thm:negsemidef} and~\ref{thm:possemidef}. In this way we recover the modularity result for such averages due to Katok--Sarnak \cite{KS}.

iii) In the special case when $f=1$, we have $s=0$ and the Whittaker function $\calW^{+,+}_{B,0}(v,0)$ reduces to the Whittaker function $W_{T,\infty}(v)$ of Theorem~\ref{thm:intro1}. The geodesic Taylor coefficients $f_n(z)$ vanish for all $n\neq 0$, and $f_0(z)=1$ at all points $z\in \H$. Hence, we recover the modularity assertion of Theorem~\ref{thm:intro1}. However, this does not give the explicit formula for $R_T$ in terms of non-archimedian local Whittaker functions.
\end{remark} 
\par
To prove the theorem, we decompose the Fourier coefficient $I_T(v,f)$ into orbital integrals and expand these with respect to the characters of $\SO(T)$.
It turns out that the corresponding integrals also make sense if $f$ is replaced by a spherical function on $\H$. We employ this fact, together with properties of the Weil representation, to construct the Whittaker functions $\calW^{+,+}_{B,n}(v,s)$ in Theorem~\ref{thm:w1}. Next we use the fact that the group $\SO(T)$ is naturally embedded into both, the Levi factor of $\Sp_2(\R)$ and the orthogonal group $\SO(V)(\R)$, and that the corresponding actions on $V(\R)^2$ are compatible, see Lemma~\ref{lem:eact}.
With these preparations, the orbital integrals can be computed in terms of the geodesic Taylor expansions of $f$. 
\par
\medskip
\paragraph{\textbf{Hecke action.}}
Let $p$ be an odd prime which is coprime to the discriminant of~$B$ and the level of the Eichler order. If~$f$ is an eigenform of the Hecke operator~$T_p$, then the theta integral $I(\tau,f)$ will again be a Hecke eigenform, now for the two non-trivial Hecke operators $\wh{T}_{p,0}^{(2)}, \wh{T}_{p,1}^{(2)}$ that we have in genus two. In Proposition~\ref{prop:HeckeIandf} we give the explicit relation of these Hecke operators on $I(\tau,f)$ with the action of~$T_p$ on~$f$. The proof is an application of the method of Rallis, extended by Zhuravlev to include the metaplectic case. In fact, Zhuravlev gives maps that intertwine the action of Hecke operators on theta functions with the Hecke action on the Schwarz functions via the Weil representation. Moreover the Rallis-Zhuravlev map intertwines the symplectic and orthogonal Hecke operators inside the Weil representation. In Section~\ref{sec:Hecke} we combine these ideas to ready-to-use formulas and cross-check them with the genus one theta lift via the Siegel-$\Phi$-operator, see Proposition~\ref{prop:HeckePhi}.
\par    
\medskip
\paragraph{\textbf{Related results.}} 
We remark that all the above questions on counting of closed geodesics on \emph{any} compact Riemann surface according to \emph{length} have been intensely studied, too, starting with Sarnak's work \cite{SarnakClassNumber}. There is no analog of special points like Heegner points in this context, though. Analogous questions on counting and equidistribution of intersection points have been investigated in \cite{Lalley1}, \cite{Lalley2}, and also \cite{Torkaman}, however without taking any properties of the intersection angle into account. Most recently, these type of results have been upgraded to give explicit error terms in \cite{Torkaman2} and also in \cite{KatzIntGeo}. Since it is difficult to compare discriminant and length on Shimura curves, we make no attempt to directly relate the two kinds of results.
\par
The distribution of intersection \emph{angles} has been studied by Pollicott--Sharp \cite{PSangular} and Herrera \cite{Herrera} in the general hyperbolic context and by Jung--Sardari \cite{JungSardari} in the arithmetic context (in fact: the modular surface case), solving the distribution conjecture from \cite{RiSL}.
\par
\par
An old theorem of Birkhoff identifies the number of intersection points of two geodesics with the \emph{linking numbers} of the links formed by the geodesic and its time reverse. This fact is employed by Duke, Imamoglu and Toth in \cite{DIT} to compute linking numbers of modular knots in terms of generalized Dedekind symbols. The latter are related to geodesic cycle integrals of weakly holomorphic modular functions on the modular curve $\Sl_2(\Z)\bs \H$. Generating series of such cycle integrals were studied earlier in \cite{DIT-Annals}, where the main result can be viewed as an analogue of \cite{KS} for weakly holomorphic modular forms. It states that the simultaneous generating series of cycles integrals and CM values of a given weakly holomorphic cusp form of weight zero is a harmonic Maass form of weight $1/2$. See also \cite{BIF} for an approach via the regularized theta correspondence. 
It is an interesting question whether the results of the present paper admit a generalization to meromorphic (or, in the modular curve case, weakly holomorphic) modular forms.
\par
One motivation for the work of Rickards came from the work of Darmon and Vonk on a real quadratic analogue  of singular moduli.
Rickards stresses in \cite[Section~2.5]{Ri} the analogy of his formula for the (unsigned) total intersection number of geodesics cycles of positive discriminants $D_1$ and $D_2$ with the formula of Gross-Zagier for the norms of differences of singular moduli \cite{GrossZagier} of negative discriminants. In particular, both have in common a decomposition according to `$b$-linking'. In the Gross-Zagier case the contribution of the $b$-linked part is determined by the arithmetic degree of a special $0$-cycle $Z(T)$, where $T=\kzxz{D_1}{b}{b}{D_2}$, on the arithmetic surface given by the modular curve $X(1)$ over $\Spec(\Z)$, which in turn can be expressed as the $T$-th Fourier coefficient of the derivative of an incoherent genus $2$ Siegel Eisenstein series of weight $3/2$.
In our interpretation  of Rickards' result (see Theorem~\ref{thm:posdeg2}), the contribution of the $b$-linked part is given by the degree of a special $0$-cycle $Z(T)$ on the complex Shimura curve associated to $B$, which in turn can be expressed as the $T$-th Fourier coefficient of the value of a coherent genus $2$ Siegel Eisenstein series of weight $1/2$.
\par
Darmon-Vonk propose that a real quadratic analogue of the differences of singular moduli is given by the evaluation of rigid meromorphic cocycles for the group $\SL_2(\Z[1/p])$ at real multiplication points. They show that the valuations at primes $q$ are related to certain \emph{signed} intersection numbers of geodesic cycles on Shimura curves. It would be interesting to investigate whether there is a variant of our theta correspondence approach involving signed intersections that could lead to a connection between the conjectural arithmetic intersection numbers of \cite[Conjecture~3.27]{DarmonVonkRigidAn} and (derivatives of incoherent) Siegel Eisenstein series of genus $2$.

\par
\medskip
\paragraph{\textbf{Acknowledgments.}} 
This work was initiated while discussing with James Rickards at the CIRM conference ``Renormalization, computation and visualization in Geometry, Number Theory and Dynamics''. The authors moreover thank Valentin Blomer, Jens Funke, \"Ozlem Imamoglu, Paul Nelson, Rebekka Strathausen, and Tonghai Yang for helpful comments and suggestions.

\section{Geodesic cycles and Heegner points}

\subsection{The setup} \label{sec:setup}

Let $B$ be an indefinite quaternion algebra over $\Q$. Let $D(B)$ be the discriminant of $B$, that is, the product of all primes $p$ for which $B_p=B\otimes_\Q \Q_p$ is a division algebra.
For $x\in B$ we denote by $\tr(x)\in \Q$ the reduced trace and by $\norm(x)\in \Q$ the reduced norm. The vector space 
\begin{align}
\label{eq:V}
V \= \left\{ x\in B\mid \; \tr(x)=0\right\}
\end{align}
together with the quadratic form $Q(x)= -\norm (x)=x^2$ is a rational quadratic space of signature $(2,1)$.  It is anisotropic if and only if $D(B)\neq 1$. The corresponding bilinear form is $(x,y) = \tr(xy)$. 
We view the group $H= B^\times$ as an algebraic group over $\Q$. The action of $H$ on $V$ by conjugation induces an isomorphism $H\cong \operatorname{GSpin}(V,Q)$. Over $\R$, the space $V(\R)$ is given by the real $2\times 2$ matrices of trace $0$, and $Q(x)=-\det(x)$. Moreover, we  have 
$H(\R)= \operatorname{GSpin}(V,Q)(\R)\cong\GL_2(\R)$. 
For any subfield $F\subset \R$, we write 
\[
H(F)^+\= \{h\in H(F)\mid \; \det(h)>0\}.
\]
\par
The hermitian symmetric space associated with $H$ can be realized as the Grassmannian
\[
\calD \= \{ z\subset V(\R)\mid \; \text{$\dim(z)=1$ and $Q\mid_z <0$}\} 
\]
of negative lines in $V(\R)$.  It is isomorphic to the complex upper half
plane $\H$ via
\be \label{eq:HtoD}
\H \ni z = x+iy \,\mapsto\, \R\cdot \frac1y 
\left(\begin{matrix} -x & |z|^2 \\ -1 & x \end{matrix} \right) \in \calD\,.
\ee
Here the natural action of $H(\R)^+$ on~$\calD$ is identified with the action
on $\H$ by fractional linear transformations. To describe the action of the full group $H(\R)$ on $\H$, we consider the involution $\sigma:\H\to \H$, $z\mapsto -\bar z$. The subgroup of order $2$ in $H(\R)$ generated by the element $\kzxz{1}{0}{0}{-1}$ acts on $H(\R)^+$ by conjugation, giving rise to an isomorphism
\[
H(\R) \cong H(\R)^+\ltimes \left\langle \kzxz{1}{0}{0}{-1}\right\rangle.
\]
Hence we may define an action of $H(\R)$ on $\H$ by letting $H(\R)^+$ act via fractional linear transformations and $\kzxz{1}{0}{0}{-1}$ via $\sigma$.  This action corresponds to the natural action on~$\calD$.
\par
Let $\calO_B$ be a maximal order in $B$ or more generally, let~$\calO_\frakM$ be an Eichler order of level $\frakM \in \N$. Then $L_\frakM :=\calO_\frakM \cap V$ is an even
lattice in $V$.
Locally, we can describe $L_{\frakM, p} := L_\frakM \otimes \Z_p$ as
\begin{equation}
  \label{eq:Lp}
  L_{\frakM, p} \cong
  \begin{cases}
    (\calR_p \cap V) \oplus \calR_p j_p & p \mid D(B),\\
\{\kzxz{a}{b}{c}{-a} \in M_2(\Z_p): c \in \frakM \Z_p\} & p \nmid D(B),
  \end{cases}
\end{equation}
where $\calR_p$ is the ring of integers of the unique unramified quadratic extension of $\Q_p$ and $j_p^2 = p$, see \cite[Theorem~13.3.11(b)]{Voight}.
When $p = 2 \mid D(B)$, it is more convenient to represent $L_{\frakM, 2}$ as $\Z_2 i \oplus \Z_2 j \oplus \Z_2 k$ with $k = ij$ and $i^2 = j^2 = k^2 = -1$.
From this description, it is easy to check that
\begin{equation}
  \label{eq:latt-size}
  |L_{\frakM, p}^\vee/L_{\frakM, p}| \=
  \begin{cases}
    2^{1+2(1-D\bmod2)+2\nu_p(\frakM)} & p = 2,\\ 
    p^{2\nu_p(D(B)\frakM)} & p \neq 2.
  \end{cases}
\end{equation}
\par
The subgroup
\begin{equation}
  \label{eq:KL}
  \begin{split}
    \Gamma_\frakM
    &:= H(\Q)^+ \cap K_{\frakM} \subset H(\Q),\quad K_{\frakM} = \prod_{p < \infty} K_{\frakM, p}\subset H(\hat\Q),\\
    K_{\frakM, p}
    &:=
      \begin{cases}
        \calO_{B, p}^\times,& p \mid D(B),\\
        \{\kzxz{a}{b}{c}{d} \in \GL_2(\Z_p): c \in \frakM \Z_p\},& p \nmid D(B),
      \end{cases}
  \end{split}
\end{equation}
stabilizes $L_\frakM$. The quotient 
\begin{align}
\label{eq:decomp}
  X_{\frakM} \= \Gamma_\frakM\bs \calD \,\cong \, \Gamma_\frakM\bs \H
  \cong X_{K_\frakM} \= H(\Q)^+\bs (\calD \times H(\hat\Q)  /K_\frakM)
\end{align}
is a Shimura curve (compact if $D(B)\neq 1$). Note that the double quotient
above is connected, since the image of $H(\Q)^+\bs H(\hat\Q)/K_\frakM$ under the spinor norm is $\Q_{>0} \bs \hat\Q^\times/\hat\Z^\times$, which is trivial.
In particular, the group $\Gamma_\frakM = \calO^1_\frakM  \subset H(\Q)^+$ consists of all elements of reduced norm one.
\par
If we equip the upper half plane with the Haar measure normalized as $dxdy/y^2$,
then the volume of the Shimura curve is
(see e.g. \cite[Theorem~39.1.13]{Voight})
\be \label{eq:volumeformula}
\vol(X_\frakM) \= \frac{\pi}3 \,D(B) \prod_{p | D(B)} \Big(1 - \frac1p  \Big)
\cdot \frakM
\prod_{p | \frakM} \Big(1 + \frac1p  \Big),
\ee
where the second product is over all distinct primes dividing~$\frakM$.
\par

\subsection{Geodesic cycles}

A vector $x \in L_\frakM$ (or more generally a vector $x \in V$) whose length $Q(x) >0$
is positive defines a \emph{geodesic}
\bes
c_x \= \{z \in \calD: z \perp x\} \, \subset \, \calD\,.
\ees
Under the condition that~$L_\frakM$ is anisotropic (or more generally if $x^\perp
\subset V$ is anisotropic) the geodesic~$c_x$ descends to closed geodesic in~$X_{\frakM}$. In fact,
\be
L_{\frakM,~ x^\perp} \,:=\,  L_\frakM \cap x^\perp \, \subset\,  x^\perp
\ee
is a sublattice in a quadratic space $V_{x^\perp} \,:=\,  V \cap x^\perp$ of
signature~$(1,1)$. Multiplicativity of the determinant implies that
Anisotropic implies that $D$ is not a square and the
positive units provides the required non-trivial stabilizer of~$c_x$.
\par
We want to relate primitive lattice vectors to two other notions giving rise to
closed geodesics in~$X_\frakM$. Let $\calO_D \subset \Q(\sqrt{D})$ be the
order of discriminant~$D$. An \emph{optimal embedding} is a ring
homomorphism~$\phi: \calO_D \to \calO_\frakM$ that cannot be extended to
an order $\calO_D' \subset \Q(\sqrt{D})$ properly containing~$\calO_D$.
Note that the conjugation of~$B$ fixes the image of an optimal embedding
(since otherwise $B$ would have a commutative subring of rank~$\geq 3$)
and thus restricts to the Galois conjugation of~$\calO_D$. In particular
an optimal embedding takes the norm form on~$K$ to the reduced norm on~$B$.
\par
Last, a primitive hyperbolic element~$\gamma \in \Gamma_\frakM$ specifies a unique geodesic in~$\calD$ fixed by~$\gamma$ and conversely, to any geodesic we associate the primitive hyperbolic element~$\gamma \in \Gamma_\frakM$, unique up to inverse, that stabilizes the geodesic. 
\par
\medskip
To establish the correspondence between lattice vectors and geodesics we define
the sublattice
\be
\label{eq:L}
L \,:=\, 2L_\frakM \cup \bigl\{\lambda \in L_\frakM:
\tfrac{1+\lambda}2 \in \calO_\frakM\bigr\} \quad \subset L_\frakM
\ee
of index four. Concretely, $L$ is the lattice with localizations
\begin{equation}
  \label{eq:latt-L}
    L_p \=
\begin{cases}
  L_{\frakM, p} & p \neq 2,\\
  \{\kzxz{a}{2b}{2^{e+1}c}{-a}: a, b,c \in \Z_2\} & p = 2 \nmid D(B) \\
  \Z_2 (2i-2k) \oplus     \Z_2 (2j-2k) \oplus \Z_2 (i+j+k) & p = 2 \mid D(B),
\end{cases}
\end{equation}
where $e = \nu_p(\frakM)$. Note that $Q(x) \equiv 0, 1\bmod{4}$ for all $x \in L$.
\par
The conjugation action of the group~$\Gamma_\frakM$ stabilizes~$L \subset L_\frakM$, and thus acts on primitive elements in~$L$, acts on images in~$\calO_\frakM$ of optimal embeddings and on primitive hyperbolic elements. It also acts on geodesics in~$\calD$ by left translation.
\par
\begin{proposition} \label{prop:VecOptEmbHyp}
There is a $\Gamma_\frakM$-equivariant bijection between
\begin{itemize}
\item[i)] primitive vectors $x \in {L} /\{\pm 1\}$ of length~$D >0$, 
\item[ii)] optimal embeddings $\phi: \calO_D \to \calO_\frakM$ of a real quadratic order  up to Galois conjugation,
\item[iii)] primitive hyperbolic elements in~$\Gamma_\frakM$ up to inversion. 
\end{itemize}
that commutes with taking the associated geodesic in~$\calD$. The bijection
between i) and ii) takes the length of~$x$ into four times the discriminant
of the order.
\end{proposition}
\par
We say that a geodesic is \emph{of discriminant~$D$}, if it corresponds to an
optimal embedding of an order of discriminant~$D$.
\par
\begin{proof}
We first give all the maps. For the map from i) to ii), given $x \in {L}$
primitive, let $D = Q(x)$ and if $x \not\in 2L_\frakM$ let $\phi(\calO_D) = \langle 1,
\tfrac{1+x}2 \rangle$. If $x \in 2L_\frakM$, then let $\phi(\calO_D) =
\langle 1, \tfrac{x}2 \rangle$.
Conversely for ii) to i), given~$\phi$, let $x = \pm \phi(\sqrt{D})$. 
\par
For the map from ii) to iii) fix a generator~$\epsilon_D$ of the totally
positive units in~$\calO_D$. Then $\phi(\epsilon_D) \in \Gamma_\frakM$ since
$\phi$ preserves norms, hyperbolic since of infinite order, and primitive
since we took a generator. Conversely for iii) to ii), take the order
generated by~$\Z$ and the given hyperbolic element~$\gamma \in
\Gamma_\frakM$ and extend it to a optimal embedding. The $\Gamma_\frakM$-equivariance of all four maps is obvious.
\par
It remains the check the properties implicit in the definition of the maps.
The images~$\phi(\calO_D)$ are contained in~$\calO_\frakM$ since $\tfrac{1 + \lambda}2 \in \calO_\frakM$ and they are indeed rings, since $(\tfrac{1 + x}2)^2 =
\tfrac{Q(x)+1}4 + \tfrac{x}2 \in \tfrac{1+x}2 + \Z$ for $x \not\in 2L_\frakM$,
and since $x^2 = Q(x) \in 4\Z$ in the other case $x \in 2L_\frakM$.
The congruence $D = 0 \mod 4$ if $x \in 2L_\frakM$ and $D = 1 \mod 4$ otherwise
implies that the ring $\calO_D$ is indeed an order of discriminant~$D$, as suggested by the notation.
\par
Conversely, since~$\phi$ preserves traces, $x$ has trace zero. If $D = 0 \mod 4$,
then $x/2 \in \phi(\calO_D) \subset \calO$, hence $x \in 2L_\frakM \subset L$.
If $D = 1 \mod 4$, then $\tfrac{1 + x}2 \in \phi(\calO_D) \subset \calO_\frakM$
and again $x \in L$. This also shows that the maps between i) and ii) are
inverses to each other. The independence
of $\pm 1$ resp.\ Galois conjugation is obvious from definition.
\par
Back to the map i) to ii), given~$x \in {L}$ primitive, we check that
that~$\phi(\calO_D)$ cannot be extended to a larger order of $\Q(\sqrt{D})$
in~$\calO$. If $x \not\in 2L_\frakM$, hence $D = 1 \mod 4$, the larger orders are
of the form $\langle 1,\tfrac12 + \tfrac{x}{2p} \rangle$ for some $p|D$,
hence $p$ odd. This implies $x/p \in {L}$, contradicting primitivity.
If~$x \in 2L_\frakM$, hence $D = 0 \mod 4$, we have to consider two kinds of
superorders. Superorders of the form $\langle 1, \tfrac{x}{2p} \rangle$ are impossible
for any~$p$ by the primitivity hypothesis. If $D/4 = 1 \mod 4$, 
there is the superorder $\langle 1,\tfrac12 + \tfrac{x}{4} \rangle$, but
if this is contained in~$\calO_\frakM$, we arrive at the contradiction $x/2 \in {L}$.
\par
Conversely, investigating the map ii) to i), to show that~$x$ is
primitive, suppose first that $x$ is divisible in ${L}$ by an odd prime~$p$.
If $x \in 2L_\frakM$, then also $x/p \in 2L_\frakM$, hence $Q(x/2p) \in \Z$ and hence
$\langle 1,x/2p \rangle \subset \calO_\frakM$ is a subring by the preceding proof,
contradicting the optimality. If $x \not\in 2L_\frakM$, then also $x/p \not\in
2L_\frakM$, hence $4Q(x/p) = 1 \mod 4$,  and hence
$\langle 1,\tfrac12 + \tfrac{x}{2p} \rangle \subset \calO_\frakM$ is a subring,
contradicting again optimality. Concerning $p=2$, observe that elements
in ${L} \setminus 2L_\frakM$ are never $2$-divisible in~${L}$. It remains to
consider $x \in 2L_\frakM$, which implies $D = 0 \mod 4$. If $x/2 \in 2L_\frakM$,
then the same argument as in the odd case contradicts optimality. If $x/2 \in L
\setminus L_\frakM$, then $1/2 + x/4 \in \calO_\frakM$, hence
$\langle 1, \tfrac12 + \frac{x}4 \rangle$ is a
superorder that contradicts optimality.
\par
To show that the correspondence commutes with taking the associated geodesic
it suffices to observe that $L_{x^\perp}$ (and thus the subspaces in~$\calD$
contained in there) is invariant under the conjugation action of all of the
order $\phi(\calO_D)$ that we associated to~$x$, and consequently also under
$\phi(\epsilon_D)$.
To see this, we may conjugate the order so that $x=i$ (e.g.\ \cite[Corollary~22]{Ri}),
so $(L_{x^\perp})_\Q = \langle j,k \rangle_\Q$ and observe that this is stable
under $\langle 1,i \rangle_\Q$.
\end{proof}

\subsection{Heegner points}

A vector $x \in L$ whose length $Q(x) < 0$ defines a point in~$\calD$,
a \emph{Heegner point}. In coordinates, under the correspondence~\eqref{eq:HtoD}
\bes
x \= \left(\begin{matrix} B & C \\ -A & -B \end{matrix} \right) \in V(\R)
\quad \text{the Heegner point is } \quad P_{x} \= \frac{-B + i \sqrt{-D}}{A}
\quad \in \H\,.
\ees
With the same proof as in the case of positive vectors one shows:
\par
\begin{proposition} \label{prop:HeegnerOptEmb}
There is a bijection between
\begin{itemize}
\item[i)] primitive vectors $x \in {L} /\{\pm 1\}$ of length~$D/4<0$, 
\item[ii)] optimal embeddings $\calO_D \to \calO_\frakM$
of an imaginary quadratic order
up to Galois conjugation\,.
\end{itemize}
The bijection between takes the length of~$x$ into four times the discriminant
of the order.
\end{proposition}

\subsection{Two vectors in~$L$: a case discussion}

Given a pair of vectors $x = (x_1,x_2) \in L_\frakM^2$ we let
\bes
T \= Q(x) \coloneqq \frac12 ((x_i, x_j)_{i,j=1,2}) \,=:
\left(\begin{matrix} D_1 & b \\ b & D_2  \end{matrix} \right)
\ees
be the corresponding inner product matrix. We summarize the geometry of geodesics and Heegner points associated with~$x$ according to the sign of the entries of~$T$, see Figure~\ref{cap:fourcases}, which is the lattice vector translation via Proposition~\ref{prop:VecOptEmbHyp} and ~\ref{prop:HeegnerOptEmb} of Figure~\ref{cap:Introfourcases}
\par

\begin{figure}[htb]
	\centering
\begin{tikzpicture} [scale=.8]
\begin{scope}
\path[draw] (-.25,0) -- (4.25,0);	
\draw [name path=circ 1] (0,0) arc (180:0:1) ++(150:2cm) node {\footnotesize{$x_1$}};
\draw [name path=circ 2] (1.4,0) arc (180:0:1.25) ++(100:1.1cm) node {\footnotesize{$x_2$}};
\fill [name intersections={of=circ 1 and circ 2, by={s}}] (s) circle (2pt);   
\draw[<-,shorten <=3pt] (s) -- ++(85:.5cm) node[above] {\footnotesize{$x_3$}};	
\end{scope}	
\begin{scope} [xshift=7cm]	
	\path[draw] (-.3,0) -- (5.4,0);	
	\draw [name path=circ 1] (0,0) arc (180:0:1.1) ++(150:2.15cm) node {\footnotesize{$x_1$}};
	\draw [name path=circ 2] (1.4,0) arc (180:0:1.35) ++(127:1.9cm) node {\footnotesize{$x_3$}};
	\draw [name path=circ 3] (3.4,0) arc (180:0:.85) ++(120:1.15cm) node {\footnotesize{$x_2$}};		
	\path [draw=red,name intersections={of=circ 1 and circ 2, by={s1}}] (s1) circle (0pt);  
	\fill [name intersections={of=circ 2 and circ 3, by={s2}}] (s2) circle (0pt);  	
	\path[name path=s1_arc1] ([shift=(0:8pt)]s1) arc (0:90:8pt);
		\path [name intersections={of=circ 2 and s1_arc1, by={w1}}] (w1); 	
	\path[name path=s1_arc2] ([shift=(-90:8pt)]s1) arc (-90:20:8pt);	
		\path [name intersections={of=circ 1 and s1_arc2, by={w2}}] (w2);  
	\draw (w2) coordinate (A) -- (s1) coordinate (B) -- (w1) coordinate (C)
	 pic () [draw, angle radius=6pt, pic text=.] {angle = A--B--C};	 
	\path[name path=s2_arc1] ([shift=(90:8pt)]s2) arc (90:180:8pt);
		\path [name intersections={of=circ 2 and s2_arc1, by={w3}}] (w3);  	
	\path[name path=s2_arc2] ([shift=(180:8pt)]s2) arc (180:250:8pt);	
		\path [name intersections={of=circ 3 and s2_arc2, by={w4}}] (w4);
	\draw (w3) coordinate (A) -- (s2) coordinate (B) -- (w4) coordinate (C)
	pic () [draw, angle radius=6pt, pic text=.] {angle = A--B--C};	 		
\end{scope}	
\begin{scope} [yshift = -3cm, xshift=0cm]	
	\path[draw] (-.3,0) -- (5.1,0);	
	\draw [name path=circ 1] (0,0) arc (180:0:1.1) ++(165:2.4cm) node {\footnotesize{$x_2$}};
	\draw [name path=circ 2] (1.42,0) arc (180:0:1.6) ++(90:.9cm) node {\footnotesize{$x_3$}};	
	\fill [name intersections={of=circ 1 and circ 2, by={s1}}] (s1) circle (2pt);  
	\draw[<-,shorten <=3pt] (s1) -- ++(0:.5cm) node[right] {\footnotesize{$x{_1}'$}};  	
	\path[name path=s1_arc1] ([shift=(0:6pt)]s1) arc (0:90:6pt);
	\fill [name intersections={of=circ 2 and s1_arc1, by={w1}}] (w1);  	
	\path[name path=s1_arc2] ([shift=(120:6pt)]s1) arc (90:200:6pt);	
	\fill [name intersections={of=circ 1 and s1_arc2, by={w2}}] (w2);
	\draw (w1) coordinate (A) -- (s1) coordinate (B) -- (w2) coordinate (C)
	pic () [draw, angle radius=6pt, pic text=.] {angle = A--B--C};	 
	%
		\path[name path=s3_arc1] ([shift=(20:20pt)]s1) arc (20:60:20pt); 
		\fill [name intersections={of=circ 2 and s3_arc1, by={w3}}] (w3) circle (2pt);  
		\draw[<-,shorten <=3pt] (w3) -- ++(100:.5cm) node[above] {\footnotesize{$x_1$}};
\end{scope}	
\begin{scope} [yshift = -3 cm, xshift=7cm]	
	\path[draw] (-.3,0) -- (5.1,0);	
	\draw [name path=circ 1,color = white] (0,0) arc (180:0:1.1) ++(165:2.4cm) node {};
	\draw [name path=circ 2,color = white] (1.42,0) arc (180:0:1.6) ++(90:.9cm) node {};	
	\fill [name intersections={of=circ 1 and circ 2, by={s1}}] (s1) circle (2pt);  
	\draw[<-,shorten <=3pt] (s1) -- ++(0:-.5cm) node[left] {\footnotesize{$x_1$}};  	
	\path[name path=s1_arc1] ([shift=(0:6pt)]s1) arc (0:90:6pt);
	\fill [name intersections={of=circ 2 and s1_arc1, by={w1}}] (w1);  	
	\path[name path=s1_arc2] ([shift=(120:6pt)]s1) arc (90:200:6pt);	
	\fill [name intersections={of=circ 1 and s1_arc2, by={w2}}] (w2);
	%
		\path[name path=s3_arc1] ([shift=(20:20pt)]s1) arc (20:60:20pt); 
		\fill [name intersections={of=circ 2 and s3_arc1, by={w3}}] (w3) circle (2pt);  
\draw[<-,shorten <=3pt] (w3) -- ++(100:.5cm) node[above] {\footnotesize{$x_2$}};
	\draw [name path=circ 3]  (s1) arc (146:120:1.6) ++(90:.9cm) node {};
\end{scope}	
\end{tikzpicture}
\caption{Geometry of two lattice vectors: (A) two positive vectors, intersecting (top left),
(B) non-intersecting (top right),  (C) a pair of a positive and a negative vector (bottom left),
(D) two negative vectors (bottom right).}
 \label{cap:fourcases}
\end{figure}
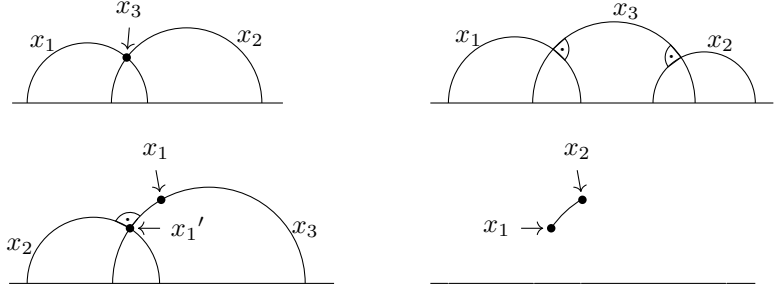

\par
\medskip
\paragraph{\textbf{(A) The positive definite case}}
This implies that $D_1>0$ and $D_2>0$ and that $x_1,x_2 \in L$ define two
geodesics $c_i = c_{x_i}$. The two geodesics are called \emph{$b$-linked}
if $b= (x_1,x_2)/2$. For comparison with \cite{Ri} we also consider two
optimal embeddings $\phi_i: \calO_{D_i} \to \calO_\frakM$ for $i=1,2$.  The optimal embeddings are said to be $b$-linked if $b = \tr(\phi_1(\sqrt{D_1})
\phi_2(\sqrt{D_2}))/2$. See \cite{Bir17,DIT} and \cite[Section~2.3]{RiSL} for
the relation to linking numbers, which serves here only as justification for
the name but will not used in the sequel. The two notions of $b$-linked
agree under the correspondence in Proposition~\ref{prop:VecOptEmbHyp}. In
fact, if $x_i = \phi_i(\sqrt{D_i})$, then 
\bes
b \= (x_1, x_2)/2 \= \tr(\sqrt{D_1} \cdot \sqrt{D_2})/2\,.
\ees
We summarize \cite[Proposition~1.10]{RiSL} and \cite[Proposition~6]{Ri}:
\par
\begin{proposition} \label{prop:geodintersect}
The geodesics stabilized by the hyperbolic elements associated with two
$b$-linked optimal embeddings intersect in~$X_\Gamma$ if and only if $b^2 < D_1D_2$.
In this case the  intersection angle~$\theta$ is given by
\be
\cos(\theta) \= \frac{b}{\sqrt{D_1D_2}}\,.
\ee
\end{proposition}
\par
\medskip
\paragraph{\textbf{(B) The indefinite case with positive diagonal entries}}
Since both $D_1>0$ and $D_2>0$ we still have a pair of $b$-linked geodesics $c_{x_i}$
that now do not intersect according to Proposition~\ref{prop:geodintersect}.
We aim for a geometric interpretation of~$b$ in this case. Let~$x_3$ be
a primitive vector spanning $\langle x_1,x_2 \rangle^\perp$. Since
$\langle x_1,x_2 \rangle$ has signature~$(1,1)$, we deduce that
$D_3 := Q(x_3) >0$.
This vector thus determines a geodesic $c_{x_3}$, whose linking number
with both~$c_{x_1}$ and~$c_{x_2}$ is zero. Consequently, $c_{x_3}$ intersects
these geodesics perpendicularly: $c_{x_3}$ is the \emph{common perpendicular}
to~$c_{x_1}$ and~$c_{x_2}$. Standard hyperbolic geometry (e.g.\
\cite[Theorem~3.2.8]{ratcliffe}) implies:
\par
\begin{proposition} \label{prop:distance2geodesics}
  If two $b$-linked geodesics~$c_{x_i}$ of discriminant~$D_i$ do not
intersect in~$X_\Gamma$, i.e., if $b^2 \geq  D_1D_2$, then their hyperbolic
distance is given by
\be \label{eq:coshdGeo}
\cosh(d_{\mathrm{hyp}}(c_{x_1},c_{x_2}))  \= \frac{b}{\sqrt{D_1D_2}}\,. 
\ee
\end{proposition}
\par
\medskip
\paragraph{\textbf{(C) The indefinite case with one negative diagonal entry}}
Suppose now that $D_1 < 0$ and $D_2 >0$. Now~$x_1$ determines a Heegner point~$P_1$
and~$x_2$ determines a geodesic~$c_{x_2}$ and by abuse of notation we still
call this pair $b$-linked. The geometric meaning is rather as follows, see
\cite[Theorem~3.2.12]{ratcliffe}).
\par
\begin{proposition} \label{prop:distanceHeegnergeodesics}
For a given $b$-linked pair of a Heegner point~$P_1$ and a geodesics~$c_{x_2}$
of discriminants~$D_i$ respectively their hyperbolic distance is given by
\be \label{eq:sinhdGeo}
\sinh(d_{\mathrm{hyp}}(P_1,c_{x_2}))  \= \frac{b}{\sqrt{-D_1D_2}}\,. 
\ee
\end{proposition}
\par
\par
\medskip
\paragraph{\textbf{(D) The indefinite case with two negative diagonal entries}}
Finally suppose that $D_1,D_2 < 0$, so that both~$x_i$ determine Heegner points~$P_i$.
Then 
\be \label{eq:coshHeegner}
\cosh(d_{\mathrm{hyp}}(P_1,P_2))  \= \frac{b}{\sqrt{D_1D_2}}
\ee
by definition of the hyperbolic distance in the Minkowski model (see e.g.\
\cite[Equation~(3.2.1)]{ratcliffe}).

\subsection{The (unsigned) counting result}

We restate the main result of Rickards \cite[Theorem~63]{Ri} on unsigned counting
intersection points. His two standing hypothesis allow to state it with moderate
amount of case distinctions. We will reprove this as a consequence of our main result
in Section~\ref{sec:SW}.
\par
Recall that the level~$\frakM$ of an Eichler order is always coprime to~$D(B)$
since division algebras over local fields have a unique order of given level.
\par
\begin{theorem} \label{thm:rickardscount}
Suppose that $D_1>0$ and $D_2 > 0$ and that $b \in \Z$ with $b \equiv D_1D_2 \mod 2$.
Suppose moreover that $\gcd(D_1,D_2,D_1D_2-b^2) =1$.\footnote{In \cite{Ri}
  these hypothesis
are called \emph{admissible} and \emph{nice} respectively.} Then the (unsigned) 
number of intersection points of a primitive geodesic of discriminant~$D_1$
with a primitive geodesic of discriminant~$D_2$ that are $b$-linked
is equal to $2^s \prod_{p \geq 2} N_p$, where the local factors~$N_p$ are given for any
prime~$p$ as follows. 
\begin{itemize}
\item[0)] If $p \nmid \det(T/4)$, then $N_p= 1$ if also $p \nmid D(B)\frakM$,
but $N_p = 0$ otherwise. Else, suppose without loss of generality
that $p \nmid D_1$ and then:
\item[i)] Suppose $(D_1,p)_p = -1$ and  $p \nmid D(B)$ and $\nu_p(\det(T/4))$
is even. If $\nu_p(D_2) \geq 2$ { or $p | \frakM$}, then $N_p = 0$,
and otherwise $N_p=1$.
\item[ii)] Suppose $(D_1,p)_p = -1$ and  $p \nmid D(B)$. If $\nu_p(\det(T/4))$
is odd, then $N_p = 0$.
\item[iii)] Suppose $(D_1,p)_p = -1$ and  $p \mid D(B)$ and $\nu_p(\det(T/4))$
is odd. If $\nu_p(D_2) \geq 2$ then $N_p = 0$ and otherwise $N_p=1$.
\item[iv)] Suppose $(D_1,p)_p = -1$ and  $p \mid D(B)$. If $\nu_p(\det(T/4))$
is even, then $N_p = 0$.
\item[v)] If $(D_1,p)_p = +1$ and $p \mid D(B)$ then $N_p = 0$.
\item[vi)] Suppose $(D_1,p)_p = +1$ and $p \nmid D(B)$. If $\nu_p(D_2) \leq 1$,
then $N_p = \max\{\nu_p(\det(T/4)) +1 { - \nu_p(\frakM)}, 0\}$. If $\nu_p(D_2) \geq 2$,
then $N_p = 2$ { unless $\nu_p(\frakM) = \nu_p(\det(T/4))$, in which case
$N_p = 1$}.
\end{itemize}
The exponent of the prefactor is $s = \omega(\frakM D(B)) + 1$, i.e., 
$s -1$ is the number of prime factors of $\frakM D(B)$.
\end{theorem}

\section{Theta lifts of Maass forms} \label{sec:Thetalift}

Here we recall some facts on Siegel Maass forms  in genus 2 and their Fourier expansions with a focus on the case of half-integral weight. Then we construct such forms as theta lifts of Maass forms on Shimura curves and describe their Fourier expansions in terms of goedesic Taylor coefficients and twisted cycle integrals.

\subsection{Siegel Maass forms}
\label{sect:SM}
Let $\H_2$ be the Siegel upper half space in genus~$2$. We denote by $\Mp_{2,\R}$ the metaplectic extension of the symplectic group $\Sp_2(\R)$ of genus $2$, realized by the two possible choices of a holomorphic square root of the automorphy factor $\det(c\tau+d)$ for $\tau\in \H_2$ and $\kabcd\in \Sp_2(\R)$. That is, an element of $\Mp_{2,\R}$ is given by a pair $\wt{g} = (g,\phi)$, where $g=\kabcd\in \Sp_2(\R)$ and $\phi:\H_2\to \C$ is a holomorphic function with $\phi^2= \det(c\tau+d)$. The group $\Mp_{2,\R}$ acts on $\H_2$ by fractional linear transformations through the covering map 
\be \label{eq:MpSp}
\prM: \Mp_{2,\R} \to\Sp_2(\R),\quad (g,\phi)\mapsto g. 
\ee
For $k\in \frac{1}{2}\Z$, this group also acts on functions on $\H_2$ via the usual Petersson slash operator in weight~$k$.
\par
Let $\wt\Gamma \subset \Mp_{2,\R}$ be an arithmetic subgroup which maps to a finite index subgroup of $\Gamma_0(4)\subset\Sp_2(\Z)$ under~$\prM$, and let $\xi:\wt\Gamma\to \C^1$ be a character. We call a $C^\infty$-function $F:\H_2\to \C$ a weak modular form of weight~$k$ for $\wt\Gamma$ and $\xi$, if it satisfies the transformation law
\begin{align}
\label{eq:siegeltr}
  (F\mid_k (\gamma, \phi))(\tau)
:= \phi^{-2k} \cdot F(\gamma  \tau)
  \= \xi(\gamma)\cdot F(\tau)
\end{align}
for all $(\gamma, \phi)\in \wt\Gamma$. 
We denote by $A_k(\wt\Gamma,\xi)$ the $\C$ vector space of all weak modular forms of weight $k$ for $\wt\Gamma$. As $\H_2$ is the symmetric space of $\Mp_2(\R)$, we can also view $F$ as a function $\widetilde F$ on $\Mp_2(\R)$. In that case, the slash operator corresponds to the right action on $\widetilde F$ via left multiplication on its argument.
\par
The center $\calC$ of the universal enveloping algebra of the Lie algebra of  $\Sp_2(\R)$ is a free $\R$-algebra of rank $2$. It is generated by the Casimir element~$C_1$ (which has degree~$2$) and an element~$C_2$ of degree~$4$. We use throughout the normalization given in~\eqref{eq:SpCasimir} in Appendix~\ref{app:Fock}.
The action of $\Mp_{2,\R}$ via the Petersson slash operator in weight~$k$ induces an action of~$\calC$ on functions on $\H_2$. We denote by $D_{1,k}$, $D_{2,k}$ the differential operators induced by~$C_1$ and~$C_2$. These operators take the space $A_k(\wt\Gamma,\xi)$  to itself. 
\par
A \emph{Siegel Maass form} of weight~$k$ for the group $\wt\Gamma$ and character $\xi$ is a function~$F$ in $A_k(\wt\Gamma,\xi)$ which is a simultaneous eigenform of $D_{1,k}$ and~$D_{2,k}$ and which has moderate growth at all rational boundary components of $\wt\Gamma$, see e.g.~\cite{Hori}.  
\par
We recall some facts on the Fourier expansion of Siegel Maass forms, loosely following \cite{Moriyama}. Let $N$ be the unipotent radical $N = \{ \kzxz{1}{*}{0}{1}\}$ of the Siegel parabolic subgroup of $\Sp_2(\R)$. We will also view it a subgroup of $\Mp_{2,\R}$ via the homomorphism $n\mapsto (n,1)$. The discrete subgroup 
$\wt\Gamma_N = \wt\Gamma\cap N$ acts  on $\H_2$ by translations by a lattice. Throughout we assume that $\xi$ is trivial on this group. Then the transformation behavior \eqref{eq:siegeltr} under $\wt\Gamma_N$ implies that any Siegel Maass form~$F$ of weight~$k$ for~$\wt\Gamma$ and $\xi$ has a Fourier expansion of the form 
\begin{align}
\label{eq:siegelfou}
F(\tau) \= \sum_{T\in \Sym_2(\Q)} F_T(v)\, \ebf(\tr(Tu)),
\end{align}
that is, an expansion with respect to the characters of the compact abelian group $N/\wt\Gamma_N$. Here we have written $\tau= u+iv$ for the decomposition of $\tau$ into its real part $u$ and imaginary part $v$. The coefficients $F_T(v)$ vanish unless~$T$ is contained in the dual with respect to the trace form of the translation lattice of $\wt\Gamma_N$. The fact that~$F$ is an eigenform of~$D_{1,k}$ and $D_{2,k}$ translates into two partial differential equations satisfied by the coefficients~$F_T(v)$.
\par
For $a\in \GL_2(\R)$ we put  
\be \label{eq:Mptilde}
m(a)\=\zxz{a}{0}{0}{{}^ta^{-1}}\in \Sp_2(\R),\quad \tilde m(a)\= \left( m(a), \sqrt{\det (a^{-1})}\right) \in \Mp_{2,\R},
\ee
where we use the principal branch of the square root.
Fix an invertible $T\in \Sym_2(\Q)$. We consider the abelian subgroup
\[
\SO(T) \=\{ a\in \SL_2(\R)\mid\; aT\, {}^ta = T\}
\]
and use the map $a \mapsto \tilde m(a)$ to regard it as a subgroup $\SO(T)
\subset \Mp_{2,\R}$, in fact inside the Levi factor of the Siegel parabolic of this metaplectic group.
\par
The group  $\wt\Gamma_T = \wt\Gamma \cap {\SO}(T)$ is a discrete subgroup with compact quotient ${\SO}(T)/\wt\Gamma_T$. If $T$ is definite, then ${\SO}(T)$ is compact, and hence $\wt\Gamma_T$ is finite. If $T$ is indefinite, then ${\SO}(T) \cong \R^\times$ and $\wt\Gamma_T$ is either isomorphic to $\Z$ or to $\Z\times (\Z/2\Z)$. The group ${\SO}(T)$ acts  on the space of smooth functions $h:\H_2\to \C$ satisfying $h(\tau+b)= \ebf(\tr (Tb))h(\tau)$ by 
\[
(a.h)(\tau) \= h\mid_k \tilde m\left({}^t a\right) (\tau)= h({}^t a \tau a)
\]
for $a\in \SO(T)$. If $h(\tau)=F_T(v)\ebf(\tr Tu)$ is the $T$-th Fourier coefficient of a Siegel Maass form~$F$ for $\wt\Gamma$ as in \eqref{eq:siegelfou}, then~$h$ is invariant under the action of the subgroup~$\wt\Gamma_T$. Consequently, $F_T$ has a convergent expansion 
\ba
\label{eq:charexp}
F_T(v) & \= \sum_\chi F_{T,\chi}(v)
\ea
with respect to the characters $\chi$ of ${\SO}(T)/\wt\Gamma_T$, where
\[
F_{T,\chi}(v) \= \frac{1}{\vol({\SO}(T)/\wt\Gamma_T)}\int_{{\SO}(T)/\wt\Gamma_T} F_T({}^t av a)\chi(a)^{-1} da,
\]
and $da$ denotes the Haar measure on ${\SO}(T)$. The coefficient $F_{T,\chi}(v)$ is a \emph{Whittaker function}\footnote{Such a function is called a generalized Whittaker function for $\Sp_2(\R)$ in \cite{Niwa} and \cite{Moriyama}. Since no other Whittaker functions appear in the present paper, we will drop the attribute \emph{generalized} in the present paper.}  of index $(T,\chi)$ and weight $k$, that is, it has moderate growth, transforms under ${\SO}(T)/\wt\Gamma_T$ according to~$\chi$, and $F_{T,\chi}(v)\ebf(\tr Tu)$ is and eigenfunction of $D_{i,k}$ with the same eigenvalues as~$F$ for $i=1,2$.  
\par
To describe the expansion \eqref{eq:charexp} in more detail we distinguish cases according to the signature of~$T$.

\subsubsection{The case of definite $T$}

If $T$ is positive definite,  let $B\in \GL_2(\R)$ such that $T={}^t\!B B$.
Then there is an isomorphism
\[
 \SO(1_2)\to \SO(T), \quad a \mapsto {}^t\!B a \, {}^t\!B^{-1},
\] 
and the characters of $\SO(T)/\{\pm 1\}$ can be parameterized as
\[
\chi_{T,n}: \SO(T)\to \C^\times,\quad {}^t\!B k_\theta \, {}^t\!B^{-1} \mapsto e^{2 in\theta}
\]
for $n\in \Z$.
Here we have written $k_\theta= \kzxz{\cos\theta}{\sin\theta}{-\sin\theta}{\cos\theta }$. Consequently, the Fourier coefficient $F_T(v)$ of any Siegel Maass form~$F$ of weight~$k$ for $\wt\Gamma$ as in \eqref{eq:siegelfou} has an expansion of the form 
\begin{align}
\label{eq:siegelfou2}
F_T(v)  \= \sum_{n\in \Z} F_{T,n}(v) ,
\end{align}
where
\begin{align}
\label{eq:siegelfou3}
F_{T,n}(v)  \= \frac{1}{\pi} \int_0^\pi F_T\left( (B^{-1} \,{}^t\!k_\theta  B) \,v \, ({}^t\!B k_\theta \, {}^t\!B^{-1})\right) e^{-2 in\theta}\, d\theta.
\end{align}
For given $v$ there exist an $\alpha \in \R$ and a diagonal matrix $\delta=\kzxz{\delta_1}{0}{0}{\delta_2}$ such that 
\ba
\label{eq:BvB}
Bv\, {}^t\!B \= {}^t k_\alpha \delta k_\alpha.
\ea
Here $\delta_1, \delta_2\in \R_{>0}$ are the eigenvalues of the real symmetric matrix $Bv\, {}^tB$. The expansion \eqref{eq:siegelfou2} then implies that 
\[
F_T(v)  \= \sum_{n\in \Z} F_{T,n}( B^{-1} \delta \,{}^t\!B^{-1} ) e^{2 in \alpha}.
\]
If $T$ is negative definite, the Fourier coefficient $F_T(v)$ can be described in the same way. We omit the details here, since this case will not be relevant for our purposes.

\subsubsection{The case of indefinite $T$}
\label{sec:indef-T}
For indefinite $T$, let $B\in \GL_2(\R)$ such that $T={}^tB\!\kzxz{1}{}{}{-1}B$. Then there are isomorphisms
\begin{equation}
  \label{eq:SO-indef}
  \begin{split}
    \R \cong \SO(\kzxz{1}{}{}{-1})/\{\pm 1\}, \quad b &\mapsto \pm \rho_\hv,~ \rho_\hv := \zxz{\cosh \hv }{\sinh \hv}{\sinh \hv}{\cosh \hv},\\
 \SO(\kzxz{1}{}{}{-1})\to \SO(T), \quad a &\mapsto {}^t\! B a \, {}^t\! B^{-1},  
  \end{split}
\end{equation}
and the unitary characters of $\SO(T)/\{\pm 1\}$, can be parameterized as
\[
  \chi_{T,s}: \SO(T)/\{\pm 1\}\to \C^\times,\quad \pm  {}^t\! B \rho_\hv \, {}^t\! B^{-1} \mapsto
\ebf(\hv s)
\]
for $s\in \R$.
When $F$ is modular of weight $k$ with respect to $\wt \Gamma \subset \Mp_{2,\R}$, the subgroup $\Gamma_T := \prM(\wt\Gamma) \cap \SO(T)/\{\pm 1\}$ is isomorphic to $\Z \cdot\hv_0 \subset \R$ for a unique $b_0 > 0$ via the first isomorphism above. 
By \eqref{eq:charexp}, the Fourier coefficient $F_T(v)$ as in \eqref{eq:siegelfou} has an expansion of the form 
\begin{align}
\label{eq:siegelfou2-indef}
  F_T(v)  \= \sum_{n\in \Z} F_{T,n}(v) ,~
\end{align}
where
\begin{align}
\label{eq:siegelfou3-indef}
  F_{T,n}(v) \= F_{T, \chi_{T, n/\hv_0}}(v)
  \= \frac{1}{\hv_0} \int_0^{\hv_0} F_T\left( (B^{-1} \rho_\hv  B) \,v \, ({}^t B \rho_\hv \, {}^t B^{-1})\right) \ebf(-n\hv /\hv_0)\, d\hv.
\end{align}
For given a symmetric matrix $M$ and $w \in \R^2$, we denote $M[w] := {}^t w M w \in \R$. 
For positive definite $v$,  
we denote
\begin{equation}
  \label{eq:beta}
  \beta := -\frac14 \log \frac{(Bv{}^tB)[\binom{1}{-1}]}{(Bv{}^tB)[\binom{1}{1}]} \in \R.
\end{equation}
Then
\ba
\label{eq:BvB-indf}
Bv\, {}^tB \= \rho_\beta \delta \rho_\beta\,.
\ea
Here $\delta = \kzxz{\delta_1}{}{}{\delta_2}$, where $\delta_1, \delta_2\in \R_{>0}$ are the eigenvalues of the real symmetric matrix $Bv\, {}^tB$. The expansion~\eqref{eq:siegelfou2-indef} then implies that 
\[
F_T(v)  \= \sum_{n\in \Z} F_{T,n}( B^{-1} \delta \,{}^tB^{-1} ) \ebf(n \beta/\hv_0).
\]


\subsection{The Weil representation and Siegel theta functions} \label{sec:WeilrepSiegelTheta}

We begin by recalling some facts on the metaplectic group and the Weil representation over local fields and the adeles. Basic references are \cite[Section~1]{Ku1}, \cite[Chapter~8.5]{KRY-book}, and \cite{Rao}. 

Let $r \in \N$ denote the genus. Almost exclusively we work with $r=2$, rarely  we will compare to the case $r=1$.
Let $\A$ be the ring of adeles of $\Q$ and denote by $\psi$ the standard additive character of $\A/\Q$ such that $\psi_\infty(x)= e^{2\pi i x}$. Let $F$ be a local field $\Q_p$ for a place $p\leq \infty$ or the ring of adeles $\A$. 
Denote by $\Sp_r(F)\subset \GL_{2r}(F)$ the symplectic group of genus $r$, 
and let $P=NM$ be its standard Siegel parabolic subgroup, where 
\begin{align*}
M=\left\{ m(a)=\zxz{a}{}{}{{}^ta^{-1}}\mid\; a\in \Gl_r(F)\right\},\quad 
N=\left\{ n(b)=\zxz{1}{b}{}{1}\mid \; b\in \Sym_r(F)\right\}.
\end{align*}
We write $\Mp_{r,F}$ for the two-fold metaplectic cover of $\Sp_r(F)$, given by the nontrivial extension of topological groups 
\[
1\longrightarrow  \{\pm 1\} \longrightarrow  \Mp_{r,F} \longrightarrow \Sp_r(F)\longrightarrow  1 .
\]
We will frequently use the identification of $\Mp_{r, F}$ with $\Sp_r(F) \times \{\pm 1\}$ via the normalized Rao cocycle  $c_R(g_1, g_2)\in \{\pm 1\}$ given in \cite{Rao}, written with tuples in square brackets, such that 
\[
[g_1, z_1]\cdot [g_2, z_2] \=[g_1g_2, z_1 z_2 c_R(g_1, g_2)].
\]
In particular, for $g_1=n(b_1)m(a_1)$ and $g_2=n(b_2)m(a_2)$ in $P$ the Rao cocylce is given in terms of the Hilbert symbol by 
\[
c_R(g_1, g_2) = (\det(a_1),\det(a_2))_F.
\]
Slightly abusing notation we will often briefly write $g$ for the element $[g,1]\in \Mp_{r, F}$.
Recall that there is a unique splitting homomorphism $\Sp_r(\Q)\to \Mp_{r, \A}$.
\par
Throughout this paper we write $V(F)=V\otimes_\Q F$ for the base change to~$F$ of the quadratic space~$V$ of dimension~$3$ in~\eqref{eq:V}. Recall that there is a Weil representation $\omega=\omega_{V, \psi}$ of $\Mp_{r, F}$ on the space of Schwartz-Bruhat functions $S(V^r(F))$ given by \cite[Page~400]{KuSplit}. It is determined by
\begin{align*}
\omega(n(b)) \phi(x) &\= \psi(\tr(Q(x)b)) \phi(x),
\\
\omega(m(a))\phi(x) &\= \chi_V(\det a) \gamma(\det a, \frac{1}2\psi)^{-3} (\det a, -1)_F |\det a|^{\frac{3}{2}} \phi(xa),
\\
\omega(w) \phi(x) &\=\gamma(V^r(F)) \int_{V^r(F)} \phi(y) \psi(-\tr(x, y)) d_\psi y,
\quad \text{where} \quad w \= \kzxz {0} {-I_r} {I_r} {0} \,,
\end{align*}
and $d_\psi y$ is the self-dual Haar measure on $V^r(F)$ with respect to $\psi$, and $\gamma(V(F)^r) = \gamma(\psi\circ \det Q)^{-r}$.  Here $\gamma(\psi)$ and $\gamma(a, \psi)$ (for $a \in F^\times$)  are the local Weil indices defined in \cite[Appendix]{Rao}, and
$$
\det Q  \= \det (\frac{1}2(e_i, e_j)) \in F^\times/(F^{\times})^2,
$$
for an $F$-basis $e_1, e_2 , e_3$ of $V(F)$. Finally,
\begin{align}
\label{eq:qc}
\chi_V(a) \= (a, - \det V)_F
\end{align}
is the quadratic character associated to $V$. Recall that a non-degenerate quadratic space~$V$ over a non-archimedian local field~$F$ is determined up to isomorphism by its dimension, its quadratic character $\chi_V$, and its Hasse invariant $\epsilon(V)$. We denote the Weil representation by $\omega_p$ for $F=\Q_p$ with $p\leq \infty$ and by $\omega_{\A}$ for $F =\A$, if we want to distinguish those cases.
Furthermore, the action of $\Mp_{r, F}$ commutes with the linear action by the group $H(F)$ on $S(V^r(F)$, which we also use $\omega$ to denote. 

\par
Next we recall basic facts on Siegel theta functions. We begin by introducing explicit coordinates on $V(\R)$ and pick as a base point $z_0\in \calD$ the negative line generated by 
\be \label{eq:e3}
e_3 \= \zxz{0}{-1}{1}{0} .
\ee
The stabilizer of $e_3$ in $H(\R)\cong \GL_2(\R)$ is equal to $\SO_2(\R)\times \R^{\times}$, where $\R^\times$ is embedded in  the center of $H(\R)$. 
We also put 
\be \label{eq:e1e2}
e_1 \= \zxz{1}{0}{0}{-1},\qquad e_2\= \zxz{0}{1}{1}{0}.
\ee
Then $e=(e_1,e_2)$ is an orthogonal basis of the real quadratic space $z_0^ \perp$, and  $e_1,e_2,e_3$ is an orthogonal basis of $V(\R)$ with $(e_1,e_1)=(e_2,e_2)=-(e_3,e_3)=2$.
\par
%
At the archimedian place we consider the Schwartz function 
\be \label{eq:StdSchwarz}
\varphi_\infty(x) \= \exp\left(-2\pi \tr( Q(x_{z_0^\perp}) - Q(x_{z_0}))\right)\in S(V^r(\R)).
\ee
At the finite places we consider the characteristic function 
\[
\varphi_L \= \operatorname{char}(\widehat L^r) \in S(V^r(\A_f))
\]
associated with the completion $\widehat L =L\otimes_\Z \widehat\Z$ of the lattice $L =L_{\mathfrak{M}}\subset V$. For $g\in \Mp_{r,\A}$ and $h\in H(\A)$, we let 
\be \label{eq:thetaWeilrep}
\theta^{(r)}(g,h,\varphi_L\otimes \varphi_\infty)= \sum_{x=(x_1,\dots, x_r)\in V^r(\Q)} (\omega(g,h) \varphi_L\otimes \varphi_\infty)(x)
\ee
be the associated theta function on $\Mp_{r,\A}\times H(\A) $. It is invariant under left translations by $\Mp_{r,\Q}\times H(\Q)$. Since the Gaussian $\varphi_\infty$ transforms in weight $1/2$ under the action of the standard maximal compact subgroup of 
$\Mp_{r, \R}$, 
we obtain a classical theta function as follows.
\par
For a point $\tau=u+iv\in \H_r$ in the Siegel upper half space in genus $r$ we choose $a\in \GL_r^+(\R)$ with $a \,{}^t a=v$. Then, using the above notation, 
the element
\[
g_\tau \= n(u)m(a) \in \Mp_{r,\R}
\]
has the property $g_\tau (i 1_r) = \tau$. 
Similarly, for $z\in \calD$ we take $h_z\in H(\R)^+$ such that $h_z z_0 = z$. The function 
\be \label{eq:defthetaL}
\theta_L^{(r)}(\tau,z) \,:=\, (\det v)^{-1/4} \theta^{(r)}(g_\tau,h_z,\varphi_L\otimes \varphi_\infty)
\ee
is a classical theta function on $\H_r\times \calD$ associated with the even lattice $L$.
As a function in $z$ it is invariant under $\Gamma=\Gamma_{\mathfrak{M}}\subset H(\Q)^+$. In the variable~$\tau$ it transforms as a non-holomorphic Siegel modular form of weight $1/2$, genus $r$, and level determined by the level of $L_{\mathfrak{M}}$.
Using the formulas for the Weil representation, a quick computation shows that 
\begin{align*}
 (\omega(g_\tau,h_z) \varphi_\infty)(x) &\= \ebf(\tr Q(x)u)(\det a)^{3/2}\varphi_\infty(h_z^{-1}xa)\\
&\= (\det v)^{3/4} \ebf\big( \tr (Q(x_{z^\perp})\tau) + \tr (Q(x_{z})\bar\tau)\big).
\end{align*}
Consequently, 
\begin{align*}
\theta_L^{(r)}(\tau,z) &\= (\det v)^{1/2} \sum_{x\in L^r} 
\ebf\big( \tr (Q(x_{z^\perp})\tau) + \tr (Q(x_{z})\bar\tau)\big)\\
&\= (\det v)^{1/2} \sum_{x\in L^r} 
e^{-4\pi \tr( Q(x_{z^\perp})v)}\cdot \ebf\big( \tr (Q(x) \bar \tau )\big).
\end{align*}

\subsection{The theta integral}

Let $f$ be a Maass form of weight $0$ for the group $\Gamma$ with eigenvalue $\lambda$ under the hyperbolic Laplace operator
\begin{align}
\label{eq:lapz}
\Delta= -y^2\left(\frac{\partial^2}{\partial x^2} + \frac{\partial^2}{\partial y^2} \right) .
\end{align}
We consider the \emph{theta integral}
\begin{align}
\label{eq:thetaint}
I^{(r)}(\tau,f) = \int_{\Gamma\bs \calD} f(z) \theta_L^{(r)}(\tau,z)\, d\mu(z),
\end{align}
where $d\mu(z)= \frac{dx\, dy}{y^2}$ denotes the invariant measure on
$\calD$.
For $r=1$ this theta integral was considered by Katok and Sarnak \cite{KS}. For our purposes, we will be almost always interested in the case $r = 2$. In that case, we omit the superscript $(r)$ from the notation.
\par
From now on we assume that $D(B)\neq 1$ so that $\Gamma\bs \calD$ is compact. Then the integral converges absolutely and defines a Siegel Maass form of genus~$2$ and weight~$1/2$ for a congruence subgroup~$\wt\Gamma$ of~$\Mp_{2,\R}$. The theta integral is also an eigenform for the invariant differential operators on $\Gamma\bs \calD$. To compute the eigenvalues we need to fix generators of the algebra of invariant differential operators. For this purpose we consider the Casimir elements $C_1,C_2$ for the symplectic group normalized as in~\eqref{eq:SpCasimir}. Since elements in the center of the universal enveloping algebra preserve $K$-types, these descend to invariant differential operators in weight $1/2$, which we denote abbreviate as ~$D_1:=D_{1,1/2}$ and~$D_2:=D_{2,1/2}$.
\par
\begin{proposition}
\label{prop:eigenvals}
Let $f$ be a Maass form as above with eigenvalue $\lambda=s(1-s)$. Then 
\bes
D_i I(\tau,f) \= \lambda_i I(\tau,f),
\ees
where
\bes
  \lambda_1 \=  4 \lambda - 25, \qquad
\lambda_2 \= 8\lambda^2 + 58 \lambda - \frac{627}4\,.
\ees
\end{proposition}
\par
\begin{proof}
   Thanks to the definition \eqref{eq:defthetaL} we can rewrite the theta integral as an integral of $\theta^{(2)}(g_\tau,h_z,\varphi_L\otimes \varphi_\infty)$ against $f$ on $\Gamma\bs \PSL_2(\R)$. Since the element $1$ in the Fock space corresponds to the standard Gaussian \eqref{eq:StdSchwarz} and since $\theta^{(2)}$ is given as a summation of translates under the Weil representation, we can compute the action of $C_i$ using the infinitesimal Weil representation on the element $1$. We can now use \eqref{eq:laplacian-commute} of the appendix.
\end{proof}
\par
Since $L$ is an even lattice, the theta integral  has a Fourier expansion of the form 
\begin{align}
\label{eq:f0}
I(\tau,f) \= \sum_{T\in \Sym_2(\Z)^\vee} I_T(v,f)\cdot  \ebf (\tr T u ),
\end{align}
where $\Sym_2(\Z)^\vee$ denotes the lattice of half integral symmetric $2\times 2$ matrices, that is, the dual of $\Sym_2(\Z)$ with respect to the trace form.
The Fourier coefficients are given by 
\[
I_T(v,f) \= \int_{u\in  \Sym_2(\R)/\Sym_2(\Z)} I(\tau,f)\cdot \ebf (-\tr T u )\, du.
\]
Inserting the definition of the theta integral, we find 
\begin{align*}
I_T(v,f) &= \int_{\Gamma\bs \calD} f(z) \int_{u\in  \Sym_2(\R)/\Sym_2(\Z)} \theta_L^{(2)}(\tau,z)\cdot  \ebf (-\tr T u )\, du\, d\mu(z)\\
&= (\det v)^{1/2} e^{2\pi  \tr Tv } \int_{\Gamma\bs \calD} f(z) \sum_{\substack{x\in L^2\\ Q(x)=T}}  e^{-4\pi \tr( Q(x_{z^\perp})v)}\, d\mu(z).
\end{align*}
For $x\in L^2$ we define the \emph{orbital integral}    
\ba
\label{eq:defix}
I(x,v,f) &\= 
\int_{ \Gamma_x\bs\calD} f(z)   e^{-4\pi \tr( Q(x_{z^\perp})v)}\, d\mu(z), \\
\ea
where $\Gamma_x$ denotes the stabilizer of $x$ in $\Gamma$. Then, by the usual unfolding argument, we obtain 
\be
\label{eq:f2}
I_T(v,f) \=(\det v)^{1/2} e^{2\pi  \tr Tv }\sum_{\substack{x\in R_T}} I(x,v,f),
\ee
where 
\be \label{eq:RTdef}
R_T \= \{ x\in L^2\mid \; Q(x)\=T\} /\Gamma
\ee
is a set of representatives for the $\Gamma$-orbits of $x\in L^2$ with $Q(x)=T$.

Since $V$ is anisotropic, by reduction theory, for all $T\in \Sym_2(\Z)^\vee$ the set $R_T$ is finite. If $T$ is invertible, then for all $x\in L^2$ with $Q(x)=T$, the space $x^\perp\subset V$ is of dimension one. Hence the stabilizer of $x$ is given by $\Gamma_x=\{\pm 1\}$.
\par
The following result can be proved in the same way as Proposition \ref{prop:eigenvals}.
\par
\begin{proposition}
\label{prop:eigenvals2}
Let $x\in L^2$. Let $f:\calD\to \C$ be a smooth function which is invariant under $\Gamma_x$ and which is an eigenfunction of $\Delta$ with eigenvalue $\lambda=s(1-s)$. Assume that the function $f(z)   e^{-4\pi \tr( Q(x_{z^\perp})v)}$ is of rapid decay on  $\Gamma_x\bs\calD$. Then the orbital integral~\eqref{eq:defix} of $f$ converges and 
\begin{align*}
 (\det v)^{1/2} \ebf(\tr Q(x) \bar \tau)I(x,v,f)  
\end{align*}
is an eigenfunction of $D_1$ and $D_2$ with eigenvalues as in  Proposition~\ref{prop:eigenvals}.
\end{proposition}
\par
We state some basic properties of orbital integrals, which will be used for the computation of the Fourier coefficients of the theta integral. The group $\Gl_2(\R)$ acts on $V(\R)^2$ (viewed as row vectors) by multiplication from the right, that is,
\begin{align}
\label{eq:act1}
x\mapsto x a \= (x_1,x_2) a
\end{align}
for $x=(x_1,x_2)\in V(\R)^2$ and $a\in \GL_2(\R)$. Moreover, this group acts on positive definite $2\times 2$ matrices by 
\[
v\mapsto av\,{}^t a.
\]
The group $H(\R)$ acts on $V(\R)^2$ by conjugation, that is,
\begin{align}
\label{eq:act2}
x\mapsto g \cdot x \= (gx_1 g^{-1}, gx_2g^{-1})
\end{align}
for $g\in H(\R)$. 
Moreover, it acts on smooth functions $f:\calD\to \C$ by the slash action (pull-back)
\[
f\mapsto f\mid g.
\]
For these actions, the following relations are easily verified.
\begin{proposition}
\label{prop:Ipos}
Using the above notation we have
\begin{align}
\label{eq:I1}
I(xa,v,f) &\= I(x,av\,{}^ta,f),\\
\label{eq:I2}
I(h.x,v,f) &\= I(x,v,f\mid h)
\end{align}
for  $a\in \GL_2(\R)$ and  $h\in H(\R)$.
\end{proposition}
\par
\subsubsection{The Siegel-$\Phi$ operator}
Holomorphic Siegel modular forms posses a Fourier expansion around the maximal dimensional cusp of the Siegel spaces and the constant terms of this Fourier expansion is obtained by the classical Siegel-$\Phi$-operator, taking plainly the limit of the Siegel modular form along the block matrices $\kzxz{\tau_1}{ 0}{ 0}{iv_2} \in \H^{r-1} \times \H \subset \H^r$ as $v_2 \to \infty$. It sends Siegel modular forms of genus~$r$ to Siegel modular forms of genus~$r-1$, and in fact Hecke eigenforms to Hecke eigenforms, see e.g.\ \cite[Section~II.5]{Klingen}, \cite[Section~2.3.4 and~4.2.3]{AndZhuBook}.
\par
In the general context of Siegel Maass forms we are not aware of a general definition of a Siegel operator. This is related to the interesting problem of computing an expansion as in Theorem~\ref{thm:evaleintro} and determining the coefficient growth for \emph{any} Siegel Maass form.
\par
Instead we give an ad hoc definition on the subspace consisting of theta integrals. In our setting, we define the \emph{Siegel-$\Phi$ operator} by
\begin{equation}
  \label{eq:Phi}
  (\Phi \theta_L^{(r)})(\tau_1, z) :=
  \lim_{v_2 \to \infty} v_2^{-1/2}
  \theta_L^{(r)}(\kzxz{\tau_1}{ 0}{ 0}{iv_2}, z),
\end{equation}
where $\kzxz{\tau_1}{ 0}{ 0}{iv_2} \in \H^{r-1} \times \H \subset \H^r$. 
It is clear the limit eliminates all $x \in L^r$ with the last
coordinate non-zero, and the existence of the limit is justified by the explicit formula 
\begin{equation}
  \label{eq:Phir}
(\Phi \theta^{(r)} )(\tau_1, z) \=
\lim_{v_2 \to \infty} v_2^{-1/2} \theta_L^{(r)}(\kzxz{\tau_1}{ *}{ *}{iv_2}, z)
\=  \theta_L^{(r-1)}(\tau_1, z)
\end{equation}
for any $r \ge 1$. This also justifies our ad-hoc definition:
\par
\begin{proposition} The Siegel-$\Phi$-operator on theta integrals defined by
\be
\Phi I^{(r)}(\tau_1, f) \,:= \,\lim_{v_2 \to \infty} v_2^{-1/2} I^{(r)} (\kzxz{\tau_1}{ *}{ *}{iv_2}, f)
\ee
converges and satisfies
\be \label{eq:SiegelActOnThetaLift}
\Phi I^{(r)}(\tau_1, f) \= \Phi I^{(r-1)}(\tau_1, f)
\ee
for any Maass form~$f$ of weight~$0$.
\end{proposition}

\subsection{Fourier coefficients for positive definite~$T$}
\label{subsec:posdef}

Assume that $T$ is postive definite in this subsection. 
According to \eqref{eq:siegelfou2} the Fourier coefficients of the theta integral $I(\tau,f)$ have a further expansion in terms of the characters of ${\SO}(T)/\{\pm 1\}$ as 
\begin{align}
\label{eq:Ifou2}
I_T(v,f)  \= \sum_{n\in \Z} I_{T,n}(v,f) ,
\end{align}
where
\begin{align}
\label{eq:Ifou3}
I_{T,n}(v,f)  \= \frac{1}{\pi} \int_0^\pi I_T\left( (B^{-1} \,{}^t  k_\theta  B) \,v \, ({}^t B k_\theta \, {}^t B^{-1}),f\right) e^{-2 in\theta}\, d\theta.
\end{align}
Similarly as in \eqref{eq:f2} these coefficients decompose into contributions coming from the orbital integrals, that is, 
\ba
\label{eq:Ifou4}
I_{T,n}(v,f)&\= (\det v)^{1/2} e^{2\pi  \tr Tv } \sum_{x\in R_T} I_n(x,v,f),
\ea
where 
\ba
\label{eq:defIn}
I_n(x,v,f)&\=  \frac{1}{\pi} \int_0^\pi I\left( x, (B^{-1} \,{}^t  k_\theta  B) \,v \, ({}^t B k_\theta \, {}^t B^{-1}),f\right) e^{-2 in\theta}\, d\theta.
\ea
\par
\begin{remark}
\label{rem:orbitwhittaker}
Let $x\in L^2$ with $Q(x)=T$ and let $f:\calD\to \C$ be  an eigenfunction of $\Delta$ with eigenvalue $\lambda=s(1-s)$. Assume that the function $f(z)   e^{-4\pi \tr( Q(x_{z^\perp})v)}$ is of rapid decay on~$\calD$. Then
Proposition \ref{prop:eigenvals2} implies that $(\det v)^{1/2} e^{2\pi  \tr Tv }I_n(x,v,f)$ is a Whittaker function of index $(T,n)$ in the sense of Section \ref{sect:SM}.
\end{remark}
 \par
\begin{remark}
\label{rem:coeffJ}
For a matrix $M=\kabcd$ we put $M^*=\kzxz{a}{-b}{-c}{d}$.
It is easily seen that 
the coefficients $I_{T,n}(v,f)$ possess the symmetry
\[
I_{T,n}(v,f) \= I_{T^*,-n}(v^*,f).
\]
\end{remark}
\par
The group  $\Orth_2(\R)$ acts on the set of orthonormal bases of $z_0^\perp$ via the inclusion into $\GL_2(\R)$ and \eqref{eq:act1}. It also acts on this set via the inclusion into $H(\R)$ and \eqref{eq:act2}. The following lemma gives a compatibility for these actions.
We call an orthonormal basis $x=(x_1,x_2)$ of $z_0^\perp$ {\em positively oriented} if there is a $k\in \SO_2(\R)$ such that $x= 2^{-1/2}e k$. Note that either~$x$ or $(x_1,-x_2)$ is positively oriented.
\par
\begin{lemma}
\label{lem:eact}
If $x=(x_1,x_2)$ is an oriented orthonormal basis of  $z_0^\perp$ and $k\in \SO_2(\R)$, then 
\[
k \cdot x \= x k^2.
\]
\end{lemma}
\par
\begin{proof}
Since $\SO_2(\R)$ acts transitively on the set of positively oriented orthonormal bases of $z_0^\perp$, it suffices to check that the actions agree on the orthogonal basis~$e$, that is,
\begin{align*}
(ke_1 k^{-1},ke_2k^{-1})&\=(e_1,e_2)k^2.
\end{align*}
This is verified by a straightforward computation.
\end{proof}
\par
\begin{proposition}
\label{prop:ISO}
For any smooth and bounded function $f$ on~$\calD$ and for $k\in \SO_2(\R)$ the orbital integral satisfies 
\be \label{eq:movek}
I(e, k^2 v \,{}^tk^2,f) \= I(e,v,f\mid k). 
\ee
\end{proposition}
\par
\begin{proof} 
We compute
\bas
I(e, k^2 v \,{}^tk^2,f) = I(e  k^2,v,f) = I(k\cdot e , v,f) = I( e , v,f\mid k), 
\eas
using Proposition~\ref{prop:Ipos} together with Lemma~\ref{lem:eact}.
\end{proof}
\par
Note that the reflection $\kzxz{1}{0}{0}{-1}\in \Orth_2(\R)$ satisfies
\[
\zxz{1}{0}{0}{-1} \cdot e = e\zxz{1}{0}{0}{-1}.
\]
Arguing as in the proof of Proposition \ref{prop:ISO}, we obtain the identities of orbital integrals  
\begin{align}
\label{eq:movesigma}
I(e,  v^*,f) &\= I(e,v,f\mid \sigma),\\
\label{eq:movesigma2}
I_n(e,  v^*,f) &\= I_{-n}(e,v,f\mid \sigma).
\end{align}


\begin{lemma} \label{le:B}
Let $x,x'\in V(\R)^2$ such that $Q(x)=Q(x')\in \Sym_2(\R)$ is invertible.
Then there exists a $g_1\in H(\R)$ satisfying $x=g_1 x'$. It is unique up to multiplication by an element of the center of $H(\R)$.
\end{lemma}
\par
\begin{proof}
We denote by $\langle x\rangle_\R=\R x_1+\R x_2$ the subspace of $V(\R)$ generated by $x=(x_1,x_2)\in V(\R)^2$. By assumption, $x=(x_1,x_2)$ and $x'=(x_1',x_2')$ are pairs of linearly independent vectors. Let $h$ be the linear map defined by 
\[
h: \langle x'\rangle_\R\to \langle x\rangle_\R,\quad h(x_1')\=x_1,\quad  h(x_2')\=x_2.
\] 
Since $Q(x')=Q(x)$, this is an isometry of quadratic spaces. Since $Q(x)$ is invertible, the quadratic spaces are non-degenerate.
\par
According to Witt's Theorem there exists an element $h_1\in \Orth(V)(\R)$ extending $h$, that is,
$h_1\mid_{\langle x'\rangle_\R} = h$. This element is unique up to multiplication from the left with the reflection~$\sigma$ that fixes $\langle x\rangle_\R$ and acts by multiplication by $-1$ on the orthogonal complement. In particular, replacing $h_1$ by $\sigma\circ h_1$ if necessary, we may assume without loss of generality that $h_1\in \SO(V)(\R)$. Then $h_1$ is the unique 
element of $\SO(V)(\R)$ satisfying $h_1\mid_{\langle x'\rangle_\R} = h$.
\par
Now any preimage $g_1\in H(\R)$ of $h_1$ under the natural surjective map $H(\R)\to \SO(V)(\R)$ has the required properties.
\end{proof}
\par
Note that in Lemma \ref{le:B} the reduced norm of $g_1$ may be negative. For example, this is the case if $x=(e_1,e_2)$ and $x'=(e_1,-e_2)$, where $g_1$ is a scalar multiple of $\kzxz{1}{0}{0}{-1}$.
\par
Since $T$ is positive definite, there exists a $B\in \Gl_2(\R)$ such that $T = {}^t BB$. If $\tilde B\in \GL_2(\R)$ is a second matrix with this property, then $\tilde B = k B$ with $k\in \Orth_2(\R)$. 
The pair  
\begin{align*}
x'=(x_1',x_2')\=(e_1,e_2)B \in V(\R)^2
\end{align*}
satisfies $Q(x')=T$. If $ x\in V(\R)^2$ also satisfies  $Q(x)=T$, then according to Lemma \ref{le:B} there exists a $g_x\in H(\R)$ such that
\[
x\=g_x (e_1,e_2) B.
\] 


\subsubsection{Geodesic Taylor expansion of a Maass form}  \label{sect:gt}

We will express the coefficients of the Fourier expansion of the theta integral in terms of 
geodesic Taylor coefficients of the Maass form $f$ at Heegner points. Here we recall some facts on geodesic Taylor expansions, following \cite[Section 1]{Fay}, see also e.g.~\cite[Chapter 1.3]{Iw}.

Consider the group $\SL_2(\R)\subset H(\R)$ and its Cartan decomposition~$\SL_2(\R) =KAK$, where
\be
A\= \{ a_t\mid\, t\in \R\}
\quad \text{with} \quad  a_t\= \zxz{e^{t/2}}{0}{0}{e^{-t/2}} ,
\ee
and $K = \SO_2(\R)$. The map 
\begin{align}
\label{eq:geocoord}
[0,\pi)\times (0,\infty)\to \H,\quad (\theta,t)\mapsto k_\theta a_t i
\end{align}
is injective and has as its image the upper half plane $\H$ with the point $i$ 
removed. Here  $k_\theta = \kzxz{\cos\theta}{\sin\theta}{-\sin\theta}{\cos\theta }\in K$ as before. Hence any $z\in \H$ can be written in the form 
$z=k_\theta a_t i$ 
with $t\in \R_{\geq 0}$ and $\theta\in [0,\pi)$. The pair $(t,\theta)$ defines the standard geodesic polar coordinates of $z$ centered at $i$, where $t$ is the geodesic distance from $z$ to $i$, and $2\theta$ the angle determined by the geodesic ray $i\R_{\geq 1}$ and the geodesic ray from~$i$ to~$z$. According to \cite[Chapter~1.3]{Iw}, the invariant measure $d\mu(z)=\frac{dx\,dy}{y^2}$ becomes in these coordinates 
\[
d\mu(z) \= (2\sinh t) d\theta\, dt.
\] 
Moreover, the invariant Laplacian \eqref{eq:lapz} is expressed as 
\begin{align}
\label{eq:polarlap}
\Delta \= -\frac{\partial^2}{\partial t^2} - \frac{\cosh t}{\sinh t}\frac{\partial }{\partial t} -
\frac{1}{(2 \sinh t)^2}\frac{\partial^2}{\partial \theta^2}.
\end{align}
By pull-back via \eqref{eq:geocoord}, we may view any function~$f$ on~$\H$  
as a function of $t\in \R_{\geq 0}$ and  $\theta\in \R$, which is $\pi\Z$-periodic in~$\theta$. Hence it has a 
geodesic Taylor expansion at the point $z_0=i$ of the form
\be \label{eq:hypFourier}
f(k_\theta a_t z_0)  \= \sum_{n\in \Z} f_{n,t}(z_0) e^{2i n\theta},
\ee
where the coefficients are given by 
\begin{align}
\label{eq:tcoeff}
\qquad f_{n,t}(z_0) \= \frac{1}{\pi}\int_0^\pi  f(k_\theta a_t z_0) e^{-2 i n \theta} \, d\theta .
\end{align}
\par
The shape of the Fourier coefficients has for instance been determined by Fay. Recall from Appendix~\ref{app:special} the definition of Legendre's $P$-function $P^{-\mu}_{\nu}(u)$, which we will use on the interval $u \in (1,\infty)$ and for $s \in \C$. 
\par
\begin{proposition}
\label{prop:Fay0}
Suppose that $f$ is an eigenfunction of the hyperbolic Laplacian $\Delta$ with eigenvalue $\lambda= s(1-s)$. Then the standard geodesic Taylor expansion of $f$ at $z_0=i$ has the form
\ba \label{eq:applyFay0}
f(z) \= \sum_{n\in \Z} f_n(z_0)  P^{-|n|}_{-s}(\cosh(t))\cdot e^{2in\theta},
\ea
where the $f_n(z_0)$ are unique complex coefficients and $ P^{-\mu}_{\nu}(u)$ denotes Legendre's $P$-function.
\end{proposition}
\par
\begin{proof}
Using the Laplacian in polar coordintes \eqref{eq:polarlap}, the expansion \eqref{eq:hypFourier}, and the fact that $f$ is an eigenfunction, we obtain for the coefficients  $f_{n,t}(z_0)$ the differential equation~\eqref{eq:LegendreDE} with $u=\cosh(t)$ and parameters~$\nu=-s$ (or equivalently~$s-1$), $\mu=|n|$. This implies that $f_{n,t}(z_0)$ is a linear combination of the Legendre $P$-function and the Legendre $Q$-function. Since the $Q$-function is not regular at~$t=0$, the coefficient of $Q$-function vanishes. Hence 
\[
f_{n,t}(z_0) \=  f_n(z_0)  P^{-|n|}_{-s}(\cosh(t))
\]
for a unique coefficient $f_n(z_0)$.
\end{proof}
\par
More generally, if $z_1\in \H$ is any point, we choose a translation matrix $g\in \SL_2(\R)$ such that $z_1= gi$. We obtain geodesic polar coordinates centered at $z_1$ by writing $z= g k_\theta a_t i$. Then $t$ denotes again the hyperbolic distance from $z$ to $z_1$, and $2\theta$ the angle determined by the geodesic ray $g(i\R_{\geq 1})$ and the geodesic ray from $z_1$ to $z$. Note that the angle actually depends on the choice  of $g$, not only on $z_1$. The geodesic Taylor expansion of $f$ at the point $z_1$ (with respect to $g$) is defined as the geodesic Taylor expansion of $f\mid g$ at $z_0=i$. The corresponding Taylor coefficients are denoted by 
$f_{n,g}(z_1)= (f\mid g)_n(z_0) $, or simply by $f_{n}(z_1)$ if the choice of $g$ is clear from the context. Hence, we have the expansion 
\begin{align}
\label{eq:applyFay1}
f(z) \=\sum_{n\in \Z} f_{n,g}(z_1)  P^{-|n|}_{-s}(\cosh(t))\cdot e^{2in\theta},
\end{align}
where 
\[
f_{n,g}(z_1)  P^{-|n|}_{-s}(\cosh(t))= \frac{1}{\pi}\int_0^\pi  f(g k_\theta a_t i) e^{-2 i n \theta} \, d\theta ,
\]
and $t$ denotes the hyperbolic distance from $z$ to $z_1$ and $2\theta$ the angle as described above.
\par
\begin{remark}
i) Let the notation be as above. For  $\alpha\in \R$, we have $z_1=gi=gk_\alpha i$. The corresponding geodesic Taylor coefficients of $f$ satisfy
\[
f_{n,gk_\alpha}(z_1)\= e^{2in\alpha} f_{n,g}(z_1).
\]
ii) The geodesic Taylor coefficients can also be computed using differential operators, see e.g.~\cite[Theorem 1.2]{Fay}.
\end{remark}

\subsubsection{Computation of the Whittaker coefficients}

We are now ready to compute the Whittaker coefficients $I_{T,n}(v,f)$ in \eqref{eq:Ifou2}. According to \eqref{eq:Ifou4} it suffices to determine the analogous coefficients $I_{n}(x,v,f)$ of the orbital integrals for $x\in R_T$.
\par
Using the coordinates of Section \ref{sect:gt}, we consider the function 
\[
h(k_\theta a_t z_0) = P^{-|2n|}_{-s}(\cosh(t))e^{-4i n\theta}
\]
on $\calD$. It is an eigenfunction of $\Delta$ with eigenvalue $s(1-s)$. 
In view of Remark~\ref{rem:orbitwhittaker}, we may use the orbital integral $I_n(e,v,h)$ to construct a Whittaker function on~$\H_2$. 
\par
\begin{theorem}
\label{thm:w1}
Let $n\in \Z$. The function 
\begin{align}
\label{eq:defwh1}
\calW^{+,+}_{n}(v,s) \,:=\, \det(v)^{1/2} e^{2\pi \tr (v)} \cdot I_n(e,v,P^{-|2n|}_{-s}(\cosh(t))e^{-4i n\theta})
\end{align}
is a Whittaker function of index $(1,n)$ and weight $1/2$. More precisely, it has the following properties:
\begin{itemize}
\item[(i)] it has moderate growth,
\item[(ii)] it satisfies 
\[
D_i \left( \calW^{+,+}_{n}(v,s) \ebf(\tr u) \right) \= \lambda_i\cdot   \calW^{+,+}_{n}(v,s) \ebf(\tr u)
\]
with eigenvalues $\lambda_i$ as in Proposition \ref{prop:eigenvals},
\item[(iii)] it satisfies $\calW^{+,+}_{n}({}^t k_\theta v k_\theta ,s) = e^{2in\theta}\calW^{+,+}_{n}(v  ,s)$ for $k_\theta\in \SO_2(\R)$,
\item[(iv)] and $\calW^{+,+}_{n}( v^*  ,s) = \calW^{+,+}_{-n}(v  ,s)$.
\end{itemize}
Furthermore, it has the integral representations
\ba
\calW^{+,+}_{n}(v,s) &\= 2 (\det v)^{1/2} e^{2\pi  \tr( v) }
\int\limits_{t>0} \int\limits_0^{\pi}\!  \sinh(t) P^{-|2n|}_{-s}(\cosh{t}) \\
&\qquad \times \exp(-4i n\alpha)\, \exp(-4\pi \tr( Q( (a_t^{-1} k_\alpha^{-1} e)_{z_0^\perp})v))\, d\alpha\, dt\,.
\ea
\end{theorem}
\par
The argument of the second exponential factor can be computed and will be abbreviated by
\ba
\tQv &\,:=\, \tQv_{t,\alpha}(v) := \tr( Q( (a_t^{-1} k_\alpha^{-1} e)_{z_0^\perp})v))\\
&\= v_1^2((c^2-s^2)^2+4c^2s^2 \cosh(t)^2) + v_2^2((c^2-s^2)^2c^2s^2 + 4\cosh(t)^2) \\
&\= 4\cosh(t)^2(v_1^2c^2s^2+v_2^2) + (c^2-s^2)^2(v_1^2 + v_2^2c^2s^2)
\ea
where $c = \cos(\alpha)$ and $s = \sin(\alpha)$. 
\par
\begin{proof}[Proof of Theorem~\ref{thm:w1}]
The properties (i)--(iii) follow immediately from Proposition \ref{prop:eigenvals2} and Remark \ref{rem:orbitwhittaker}. Property (iv) is a consequence of \eqref{eq:movesigma2}. It remains to prove the integral representation. In fact, we will show slightly more generally, that if  $f:\calD\to \C$ is any smooth eigenfunction of $\Delta$ with eigenvalue $\lambda= s(1-s)$, then
\ba
\label{eg:w11}
I_n(e,v,f)
&=2 f_{-2n}(z_0) \int\limits_{t>0} \int\limits_0^{\pi}  P^{-|2n|}_{-s}(\cosh{t}) e^{-4i n\alpha}\, e^{-4\pi \tQv} 
\sinh(t)\, d\alpha\, dt.
\ea 
The assertion then follows by specializing $f$ to the function 
$P^{-|2n|}_{-s}(\cosh(t))e^{-4i n\theta}$. 
\par
To show \eqref{eg:w11}, we deduce by \eqref{eq:defIn} and Proposition \ref{prop:ISO} that
\begin{align*}
I_n(e,v,f)&\= \frac{1}{2\pi} \int_0^{2\pi} I\left( e, {}^t  k_\theta v  k_\theta ,f\right) e^{-2 in\theta}\, d\theta\\
&\=\frac{1}{2\pi} \int_0^{2\pi} I\left( e, v , f\mid k_{-\theta/2}\right) e^{-2 in\theta}\, d\theta \\
&\=\frac{1}{\pi} \int_0^\pi I\left( e, v , f\mid k_{\theta}\right) e^{4 in\theta}\, d\theta .
\end{align*}
We use the geodesic polar coordinates \eqref{eq:geocoord} around $z_0$  
to express $I(e,  v , f)$ as follows:
\begin{align*}
I(e,v,f) &\= 
\int_{\H} f(z)   e^{-4\pi \tr( Q(e_{z^\perp})v)}\, d\mu(z), \\
&\= 2\int_{t>0} \int_{\alpha=0}^\pi  f(k_\alpha a_t z_0)
e^{-4\pi \tQv}\, 
\sinh(t)\, d\alpha\, dt.
\end{align*}
Putting this into the integral for $I_n(e,v,f)$ we get
\begin{align*}
I_n(e,v,f)
&\= \frac{2}{\pi}\int_{\theta=0}^{\pi}\int_{t>0} \int_{\alpha=0}^\pi   f(k_\theta k_\alpha a_t z_0)   e^{-4\pi \tQv}
\sinh(t)\, d\alpha\, dt \,e^{4 in\theta}  d\theta\\
&\=\frac{2}{\pi}\int_{t>0} \int_{\alpha=0}^\pi \int_{\theta=0}^{\pi}  f(k_{\theta+\alpha}  a_t z_0 )  e^{4 in\theta}\, d\theta\, e^{-4\pi \tQv}
\sinh(t)\, d\alpha\, dt .
\end{align*}
By means of Proposition \ref{prop:Fay0} we find that 
\begin{align*}
\frac{1}{\pi}\int_{0}^{\pi}  f(k_{\theta+\alpha} a_t z_0)  e^{4 in\theta}\, d\theta 
\= f_{-2n}(z_0)\cdot P^{-|2n|}_{-s}(\cosh{t}) e^{-4i n\alpha}.
\end{align*}
Consequently, 
\begin{align*}
I_n(e,v,f)
&\=f_{-2n}(z_0)\cdot 2 \int\limits_{t>0} \int\limits_0^{\pi}  P^{-|2n|}_{-s}(\cosh{t}) e^{-4i n\alpha}\, e^{-4\pi \tQv}
\sinh(t)\, d\alpha\, dt.
\end{align*}
This concludes the proof of the theorem.
\end{proof}
\par
\begin{remark}
\label{rem:w2}
i) Let $B\in \GL_2(\R)$ such that $T ={}^t BB$ as before. 
It follows that the function 
\ba
\label{eq:defwt}
\calW^{+,+}_{B,n}(v,s) \,:=\, (\det B)^{-1}\,\calW^{+,+}_{n}(Bv\, {}^tB,s)
\ea
is a Whittaker function of index $(T,n)$ and weight $1/2$ with the same eigenvalues as in Theorem~\ref{thm:w1} (ii).

ii) If $k_\theta\in \SO_2(\R)$, then 
\begin{align*} 
\calW^{+,+}_{k_\theta B,n}(v,s) \= e^{-2in\theta}\,\calW^{+,+}_{B,n}(v,s).
\end{align*}

iii) Because of Theorem~\ref{thm:w1} (iv), we have 
\begin{align*}
\calW^{+,+}_{\kzxz{1}{0}{0}{-1}B,n}(v,s) \= \calW^{+,+}_{B,-n}(v,s).
\end{align*}

iv) The functional equation $P^n_{-s}(z)= P^n_{s-1}(z)$ of the Legendre function implies the functional equation 
\[
\calW^{+,+}_{B,n}(v,s) \= \calW^{+,+}_{B,n}(v,1-s)
\]
for the Whittaker function.
\end{remark}


\begin{theorem}
\label{thm:evale}
Let $B\in \GL_2(\R)$ such that $T ={}^t BB$.
Assume that $f:\calD\to \C$ is a smooth eigenfunction of $\Delta$ with eigenvalue $\lambda= s(1-s)$. Then the $n$-th coefficient of $I_T(v,f)$ is given by
\begin{align}
\label{eq:genwc}
I_{T,n}(v,f) &\=  \calW^{+,+}_{B,n}(v,s)\cdot 
\sum_{x \in R_T} f_{-2n,g_x}(z_x).
\end{align}
Here, for every $x\in R_T$ we have fixed $g_x\in H(\R)$ such that $x=g_x (e_1,e_2) B$, 
and $z_x=g_x z_0 =x^\perp\in \calD$ denotes the corresponding Heegner point. Moreover, $f_{-2n,g_x}(z_x)$ denotes a geodesic Taylor coefficient at $z_x$ as in \eqref{eq:applyFay1}, and  
$\calW^{+,+}_{B,n}(v,s)$ the Whittaker function \eqref{eq:defwt}.
\par
In particular, the $0$-th coefficient is given by 
\begin{align*}
I_{T,0}(v,f) &\=  \calW^{+,+}_{B,0}(v,s)\cdot \sum_{x \in R_T} f(z_x).
\end{align*}
\end{theorem}
\par
\begin{proof}[Proof of Theorem~\ref{thm:evale}]
 According to \eqref{eq:Ifou4}, we need to compute the summands of
\begin{align}
\label{eq:expIT2}
I_{T,n}(v,f)&= (\det v)^{1/2} e^{2\pi  \tr Tv } \sum_{x\in R_T} I_n(x,v,f),
\end{align}
where 
\begin{align*}
I_n(x,v,f)&=  \frac{1}{\pi} \int_0^\pi I\left( x, (B^{-1} \,{}^t  k_\theta  B)
 \,v \, ({}^t B k_\theta \, {}^t B^{-1}),f\right) e^{-2 in\theta}\, d\theta.
\end{align*}
From the identity $x=g_x (e_1,e_2) B$ we deduce by Proposition \ref{prop:Ipos} that 
\begin{align*}
 I\left( x, (B^{-1} \,{}^t  k_\theta  B) \,v \, ({}^t B k_\theta \, {}^t B^{-1
}),f\right) 
& =  I\left( g_x e B, (B^{-1} \,{}^t  k_\theta  B) \,v \, ({}^t B k_\theta \, 
{}^t B^{-1}),f\right) \\
& =  I\left( g_x e , {}^t  k_\theta  B \,v \, {}^t B k_\theta ,f\right
)\\
&= I\left( e , {}^t  k_\theta  B \,v \, {}^t B k_\theta  ,f\mid g_x
\right). 
\end{align*}
This implies 
\begin{align}
\label{eq:reduction}
I_n(x,v,f)&\=  I_n(e,Bv\, {}^t B, f\mid g_x),  
\end{align}
an identity which reduces our computation to the case $x=e$, $T=1$, and $B=1$. 
We infer by Theorem \ref{thm:w1} and \eqref{eg:w11} that 
\begin{align*}
I_n(x,v,f)&\=  (f\mid g_x)_{-2n}(z_0)\cdot  I_n (e,Bv\, {}^t B, P^{-|2n|}_{-s}(\cosh(t))e^{-4i n\theta}).
\end{align*}
Inserting this into \eqref{eq:expIT2} and using \eqref{eq:defwh1}, we obtain 
\begin{align*}
I_{T,n}(v,f)&\=  \sum_{x\in R_T} (f\mid g_x)_{-2n}(z_0) \cdot (\det B)^{-1}\calW_{n}( Bv\, {}^t B,s).
\end{align*}
Now the first assertion follows from  \eqref{eq:defwt} and the fact that 
$(f\mid g_x)_{-2n}(z_0)=f_{-2n,g_x}(z_x)$. 
For the second we use in addition that $f_{0,g_x}(z_x)$ is just the value of $f$ at $z_x=g_x z_0$. 
\end{proof}
\par
\begin{remark}
i)
If we replace in Theorem \ref{thm:evale} the matrix $B$ by $ B' = k_\theta B$ with $k_\theta\in \SO_2(\R)$, the Whittaker function changes by 
\[
\calW^{+,+}_{k_\theta B,n}(v,s) \= e^{-2in\theta}\cdot \calW^{+,+}_{B,n}(v,s).
\]
Then we also have to replace the $g_x\in H(\R)$ by $ g_x'= g_x k_{-\theta/2}$. It is easily seen that 
\[
 (f\mid g_x')_{-2n}(z_0) \=   e^{2 in \theta}\cdot (f\mid g_x)_{-2n}(z_0) , 
\]
and hence the right hand side of \eqref{eq:genwc} remains unchanged.

ii) The definition of the Whittaker function and Theorem \ref{thm:w1} (iv) imply that 
\[
\calW_{B^*,-n}(v^*,s)= \calW_{B,n}(v,s).
\]
\end{remark}

\subsection{Fourier coefficients for indefinite~$T$}
\label{subsec:indef}

\subsubsection{Geodesic cartesian coordinate expansion of a Maass form}

In the so-called Flensted-Jensen coordinates, see Remark~\ref{rem:introIT}~i) for the geometric interpretation, we write $g = a_\hv \rho_{t} k_\theta \in \SL_2(\R)$ with $\rho_t$ defined in \eqref{eq:SO-indef}.
This gives a decomposition $\SL_2(\R)= ARK$ with $R = \{\rho_t: t \in \R\}$.
In these coordinates, the Laplacian $\Delta$ and the Haar measure are
\begin{equation}
  \label{eq:Lap-HAK}
  \Delta \=-\frac14\left(
   \partial_t^2 +2 \tanh(2t) \partial_t +  \frac{4\partial_\hv^2}{(\cosh(2t))^2}
 \right),~
 d\mu(z)
\= \frac{dxdy}{y^2} \= 2 (\cosh2t) dt d\hv.
\end{equation}
Let $f$ be a Maass cusp form on $\Gamma \backslash \H$ with eigenvalue $\lambda = s(1-s) = \frac14 + r^2, r \ge 0$ under~$\Delta$. The subgroup $A \cap \Gamma$ is infinite cyclic. Choose a generator $a_{\hv_0}$ with $\hv_0 > 0$. For $n \in \Z$, we are interested in understanding the function
\begin{equation}
  \label{eq:Cnf}
  C_{n, f}(g) \,:=\,
\frac1{\hv_0}
  \int_{\R/\Z \cdot \hv_0} f((a_{\hv} g )\cdot i)\ebf(-n\hv/\hv_0) d\hv
\end{equation}
on $\SL_2(\R)$.
Since $a_{\hv + \hv'} = a_{\hv} \cdot a_{\hv'}$, this function can be decomposed as 
$$
C_{n, f}(a_{\hv'} g) \= \ebf(n\hv'/\hv_0) C_{n, f}(g)
$$
for all $\hv' \in \R$. Also, $C_{n, f}$ is right $\SO_2(\R)$-invariant. So we define the following function
\begin{equation}
  \label{eq:cnf}
  c_{n, f}(t) \,:=\,
  \ebf(-n\hv/\hv_0)   C_{n, f}(a_\hv \rho_{t} k_\theta)
  \=
  \frac1{\hv_0}
  \int_{\R/\Z \cdot \hv_0} f((a_{\hv} \rho_t )\cdot i)\ebf(-n\hv/\hv_0) d\hv.
\end{equation}
Notice that
\begin{equation}
  \label{eq:cnf-inv}
  \begin{split}
    c_{n, f}(-t)
    &\=  C_{n, f}( \rho_{-t} k_\theta)
      \=  C_{n, f}(\kzxz{1}{}{}{-1}\rho_{{t}} k_{-\theta}\kzxz{1}{}{}{-1})
      \= c_{n, \sigma(f)}(t),
  \end{split}
\end{equation}
where $\sigma(f)(z) = f(-\overline z)$. For an even Maass form $f$, i.e.\ if $\sigma(f) = f$, the function $c_{n, f}(t)$ is even in~$t$ for all $n \in \Z$. Using $(\Delta - \lambda)f = 0$, we see that $c_{n, f}$ satisfies the differential equation
\begin{equation}
  \label{eq:diff-op-1}
  \left(   \partial_t^2
    + 2\tanh(2t)\partial_t
    + \frac{-4(2\pi n/\hv_0)^2}{(\cosh(2t))^2}
    + 4\lambda\right)
  c_{n, f}(t) \= 0.
\end{equation}
We make a substitution $u = \tanh(2t)
$, which implies $e^{4t} = \frac{1+u}{1-u}$, and set
$$
\tc_{n, f}(u) \,:=\,
\cosh(2t)^{1/2} c_{n, f}(t).
$$
This turns the equation above into the Legendre differential equation \eqref{eq:LegendreDE} with parameters $\mu = \pm ir = \pm (s - \frac12)$ and $\nu = -\frac12 \pm 2n \pi i /\hv_0$. In particular, we define
\begin{equation}
  \label{eq:Pc}
  \begin{split}
    \Pb_{n, s}^\pm(x)
    &\,:=\,
      \frac1{C^\pm} \lp (      P^{\mu}_{\nu}(x) + P^{-\mu}_{\nu}(x))
      \pm     (      P^{\mu}_{\nu}(-x) + P^{-\mu}_{\nu}(-x) )\rp,
    \\
    C^+
    &\=         2(      P^\mu_{\nu}(0) + P^{-\mu}_{\nu}(0)),~   \quad  C^-
    \=         2(      (P^{\mu}_{\nu})'(0) + (P^{-\mu}_{\nu})'(0))
  \end{split}
\end{equation}
for $-1 < x < 1$. Then $P^\mu_\nu(u)$ and $Q^\mu_\nu(u)$ are linear combinations of
$\Pb^\pm_{n, s}(u)$. It is clear that the function $\Pb^\pm_{n, s}$ satisfies the following symmetries
  \begin{equation}
    \label{eq:P-symm}
    \Pb^\pm_{-n, s}(x) \= \Pb^\pm_{n, s}(x),\quad 
    \Pb^\pm_{n, s}(-x) \= \pm \Pb^\pm_{n, s}(x).    
  \end{equation}
Putting these together gives us the following result.
\begin{proposition}
  \label{prop:cnf} 
  For a Maass cusp form $f$ with eigenvalue $s(1-s)$, let  $c_{n, f}(t)$ be the function defined in \eqref{eq:cnf}. Then 
\begin{equation}
  \label{eq:cnf-1}
  c_{n, f}(t) \=
  \cosh(2t)^{-1/2}
  \lp  \gamma_{f, n} \Pb^+_{n, s} (\tanh(2t))
  +   \gamma'_{f, n} \Pb^-_{n, s} (\tanh(2t))\rp,
\end{equation}
where the constants $\gamma_{f, n}, \gamma^{\prime}_{f, n} \in \C$ are given by
\begin{equation}
  \label{eq:Cnf-0}
  \begin{split}
      \gamma_{f, n}
  &\=
\frac{   1}{ \hv_0}
    \int_{\R/\Z \cdot \hv_0} f(a_{\hv}\cdot i)\ebf(-n\hv/\hv_0) d\hv,\\
          \gamma'_{f, n}
  &\=
\frac{   -i}{ \hv_0}
    \int_{\R/\Z \cdot \hv_0}
(\partial_z f)
    (a_{\hv}\cdot i)\ebf(-n\hv/\hv_0)
    ie^\hv    d\hv.
  \end{split}
\end{equation}
\end{proposition}
\par
\begin{remark}
  \label{rmk:cycle-int}
  Since $f$ has weight 0, the derivative $(\partial_z f)(z)$ is modular of weight 2. If we set $z = a_\hv \cdot i = e^\hv i$, then $ie^\hv d\hv = dz$ and the integral expression of $\gamma^-_{f, n}$ is the cycle integral of the differential form $(\partial f)(z)dz$ against the exponential function.
  Using these coefficients, we have the following expansion
  \begin{equation}
    \label{eq:f-FJ-exp}
    f \lp e^\hv \frac{i\cosh t  + \sinh t }{i \sinh t  + \cosh t} \rp
\= \sum_{n \in \Z}
  \lp  \gamma_{f, n} \Pb^+_{n, s} (\tanh(2t))
  +   \gamma'_{f, n} \Pb^-_{n, s} (\tanh(2t))\rp
\frac{\ebf(n\hv/\hv_0) }{\cosh(2t)^{1/2}} .
  \end{equation}
\end{remark}
\begin{proof}[Proof of Proposition~\ref{prop:cnf}]
Let $f^\pm := \frac12 ( f \pm \sigma(f))$. 
  From \eqref{eq:diff-op-1} and the subsequent discussion, we obtain~\eqref{eq:cnf-1} with some constants $\gamma_{f, n}, \gamma^{\prime}_{f, n}$ independent of~$t$. Since $c_{n, f^\pm}(-t) =  \pm c_{n, f^\pm}(t)$ by \eqref{eq:cnf-inv},  decomposing $c_{n, f}(t) = c_{n, f^+}(t) + c_{n, f^-}(t)$ and using the parity of $\Pb^\pm_{n, r}$ from~\eqref{eq:P-symm} gives us
  $$
  c_{n, f^+}(t) \=
    \cosh(2t)^{-1/2}
    \gamma_{f, n} \Pb^+_{n, s} (\tanh(2t)), \quad
    c_{n, f^-}(t) \=
    \cosh(2t)^{-1/2}
  \gamma'_{f, n} \Pb^-_{n, s} (\tanh(2t)).
  $$
To determine $\gamma_{f, n}$, we simply put $t = 0$ and use $f^+(a_\hv \cdot i ) = f(a_\hv \cdot i)$. To determine $\gamma_{f, n}^-$, we look at the first term of the Taylor expansion at $t = 0$.  This  finishes the proof.
\end{proof}

\subsubsection{Computation of the Whittaker coefficients}

We will compute the Whittaker coefficients $F_{T, n}(v)$ in \eqref{eq:siegelfou3-indef} for 
$F = I(\tau,f)$.
By \eqref{eq:Ifou4}, this is the sum over $x\in R_T$ of the corresponding coefficients
$(\det v)^{1/2} e^{2\pi \tr Tv} I_{n}(x,v,f)$, where
\begin{equation}
  \label{eq:Inx-indef}
  I_n(x,v,f) \=
 \frac{1}{\hv_0} \int_0^{\hv_0} I\left(x, (B^{-1} \rho_\hv  B) \,v \, ({}^t B \rho_\hv \, {}^t B^{-1}), f\right) \ebf(-n\hv /\hv_0)\, d\hv
\end{equation}
is the indefinite analogue of \eqref{eq:defIn} in the notation of Section~\ref{sec:indef-T}. Here we have chosen
\begin{equation}
  \label{eq:gxeB}
  x \= g_x \cdot e\cdot B \= g_x \cdot (e_2, e_3)\cdot B
\end{equation}
with $g_x \in H(\R)$. By Proposition~\ref{prop:Ipos}, we have
$$
I_n(x,v,f) \=
 \frac{1}{\hv_0} \int_0^{\hv_0} I\left(e \rho_\hv B, v, f \mid g_x^{}\right) \ebf(-n\hv /\hv_0)\, d\hv.
$$
In the indefinite case, the identity
  $$
  e \cdot \rho_\hv =
(\cosh \hv e_2 + \sinh\hv e_3, \sinh \hv e_2 + \cosh\hv e_3)
=
\lp
\zxz{}{e^{-\hv}}{e^{\hv}}{},
\zxz{}{-e^{-\hv}}{e^{\hv}}{}
\rp
=a_{-\hv} \cdot e.
  $$
gives the following analogue of Lemma~\ref{lem:eact}:
\par
\begin{lemma}
  \label{lem:eact-indef}
  For $e = (e_2, e_3)$ and $\rho_\hv \in R$, we have
  $$
e \cdot \rho_\hv \= a_{-\hv} \cdot e.
  $$
\end{lemma}
\par
With this lemma, we can simplify $I_n(x, v, f)$ to
\begin{equation}
  \label{eq:Inx-indef1}
  I_n(x,v,f) \=
 \frac{1}{\hv_0} \int_0^{\hv_0} I\left(e B, v, f_x\mid a_{-\hv}\right) \ebf(-n\hv /\hv_0)\, d\hv
\end{equation}
with $f_x := f\mid g_x$.

Inserting the definition of $I$ and the cycle integral $C_{n, f}$ from \eqref{eq:Cnf} gives us the following analogue of Theorem~\ref{thm:w1} and Theorem~\ref{thm:evale} in the indefinite case.
 \begin{theorem}
   \label{thm:W-indef}
Let $f$ be a Maass cusp form of weight 0 and eigenvalue $\lambda$. For each $x \in R_T$, let $g_x, B$ be as in \eqref{eq:gxeB}, $f_x$ be as above, and  $c_{f_x, n}$ the cycle integral defined in \eqref{eq:Cnf-0} with $n \in \Z$. Then
   \begin{equation}
     \label{eq:Inx-indef-0}
     I_n(x, v, f) \= \gamma_{f_x, n}\cdot \calW^{+,-}_{B, n}(v, s)
   \end{equation}
where $\calW^{+,-}_{B, n}(v, s) := (\det B)^{-1}\calW^{+,-}_{n}(B v \, {}^tB, s)$ with\
$$ \calW^{+,-}_{n}(v, s) := \int_{\R^2} e^{-4\pi \tr(Q(e_{(a_c \rho_t \cdot i)^\perp})v)}
\ebf(-nc/\hv_0) \Pb^+_{n, s} (\tanh(2t)) 2 (\cosh(2t))^{1/2} dt dc.
$$
being a Whittaker function of index $(1, n)$ satisfying the following properties:
\begin{itemize}
\item[(i)] it has moderate growth,
\item[(ii)] it satisfies 
\[
D_i \left( \calW^{+,-}_{n}(v,s) \ebf(\tr u) \right)= \lambda_i\cdot  \calW^{+,-}_{n}(v,s) \ebf(\tr u)
\]
with eigenvalues $\lambda_i$ as in Proposition \ref{prop:eigenvals},
\item[(iii)] it satisfies $\calW^{+,-}_{n}({}^t \rho_\hv v \rho_\hv ,s) \= 
  \ebf(-n\hv/\hv_0)  \calW^{+,-}_{n}(v  ,s)$ for all $\hv \in \R$,
\item[(iv)]
 $\calW^{+,-}_{B, n}(v, s) \= \pm \calW^{+,-}_{B, n}(v, 1-s)$ and
  $\calW^{+,-}_{B^*, n}(v^*, s) \= \calW^{+,-}_{B, n}(v, 1-s)$.
\end{itemize}
In particular, the $n$-th coefficient of $I_T(v,f)$ is given by
\begin{align}
\label{eq:genwcIndef}
I_{T,n}(v,f) &\=  \calW^{+,-}_{B,n}(v,s)\cdot \sum_{x \in R_T} \gamma_{f_x, n}.
\end{align}
 \end{theorem}
\par
 \begin{proof}
Using Proposition~\ref{prop:Ipos}, we can suppose $x = e$. Next, inserting the definitions \eqref{eq:defix} and \eqref{eq:Cnf} into \eqref{eq:Inx-indef1} we obtain
   \begin{align*}
     I_n(e,v,f)
     &\=
       \frac{1}{\hv_0} \int_0^{\hv_0}
\int_{\calD} f(a_{-\hv } \cdot z) \exp({-4\pi \tr( Q(e_{z^\perp})v)})\, d\mu(z)\ebf(-n\hv /\hv_0)\, d\hv\\
     &\= 
       \frac{2}{\hv_0} \int_0^{\hv_0}
       \int_{\R^2} f(a_{-\hv +c} \rho_t \cdot i) \exp({-4\pi \tr( Q(e_{(a_c\rho_t\cdot i)^\perp})v)})
       \cosh(2t)dt dc \, \ebf(-n\hv /\hv_0)\, d\hv\\
     &\= 2  \int_{\R^2}
c_{n, f}(t) \exp({-4\pi \tr( Q(e_{(a_c\rho_t\cdot i)^\perp})v)})
       \ebf(-nc /\hv_0)   \cosh(2t)^{}dt dc \, .
   \end{align*}
   Note that
   \begin{align*}
        \tr( Q(e_{(a_c\rho_t\cdot i)^\perp})v) &\=
   (   \cosh(c)^2 + \sinh(c)^2\sinh(2t)^2)v_1\\
&   + 2\sinh(c)\cosh(c)\cosh(2t)^2 v_2
   + (   \sinh(c)^2 + \cosh(c)^2\sinh(2t)^2)v_3
   \end{align*}
   is even in $t$. 
   Substituting in equation \eqref{eq:cnf-1} yields \eqref{eq:Inx-indef-0}.
\par
  The properties (i) and (ii) are proved in the same way as in Theorem \ref{thm:w1}. The last two properties follows from
   $$
   \tr( Q(e_{(a_c\rho_t\cdot i)^\perp})v) \=
\tr( Q(e_{(a_c\rho_{-t}\cdot i)^\perp})v^*) 
   $$
and \eqref{eq:P-symm}. Equation \eqref{eq:genwcIndef} follows from summing \eqref{eq:Inx-indef-0} over $x \in R_T$.
 \end{proof}
\par
\begin{remark}
   \label{rmk:cfx0}
   For $x \in R_T$  and $f_x$ as above, the coefficient $\gamma_{f_x, 0}$ is just the cycle integral
   \begin{equation}
     \label{eq:cfx0}
f(C_x) \,:=\,  \gamma_{f_x, 0} \=  \int_{C_x} f d c,
   \end{equation}
   where $C_x\subset \Gamma\backslash\calD$ is the Heegner cycle determined by the negative line in $V(\R)$ generated by $x$, and $dc$ is the invariant measure on $C_x$ such that $C_x$ has length 1.
 \end{remark}

\subsection{Singular matrices~$T$.}
\label{sec:rankone}

Here we consider the Fourier coefficients $I_T(v,f)$ for singular matrices $T$.
We begin with the case $T=0$. Since $V$ is anisotropic, it follows from \eqref{eq:f2} that 
\begin{align}
I_0(v,f) &\=(\det v)^{1/2} \int_{\Gamma\bs \calD} f(z)\, d\mu(z)\\
\nonumber
&\=\begin{cases} (\det v)^{1/2} \vol(\Gamma\bs \calD)f(z_0), & \text{if $f$ is constant,}\\
0,&\text{otherwise.}
\end{cases}
\end{align}

We now consider the case when $T$ has rank $1$. Then there exists an $\eps\in \{\pm 1\}$ and a $B\in \GL_2(\R)$ such that 
\[
T\={}^t\!B \zxz{\eps}{0}{0}{0}B.
\]
Recall from \eqref{eq:f2} that 
\[
I_T(v,f) \= (\det v)^{1/2} e^{2\pi  \tr Tv }\sum_{\substack{x\in R_T}} I(x,v,f),
\]
where 
\[
I( x, v ,f) \= 
\int_{ \Gamma_x\bs\calD} f(z)   e^{-4\pi \tr( Q(x_{z^\perp})v)}\, d\mu(z). 
\]

\subsubsection{ $T$  {negative} semidefinite.}
The stabilizer $\Gamma_x$ is finite and contains $\pm 1$. It is useful to consider the following variant of the orbital integral:
\[
\tilde I( x, v ,f) \= 
\int_{ \calD} f(z)   e^{-4\pi \tr( Q(x_{z^\perp})v)}\, d\mu(z).
\]
Then $I( x, v ,f)=\frac{2}{|\Gamma_x|} \tilde I(x,v,f)$. 
We write $v=\kzxz{v_1}{v_{12}}{v_{12}}{v_2}$ and analogously for $u$.
The following result is an analogue of Theorem~\ref{thm:w1}.
\par
\begin{theorem}
\label{thm:w-10}
 The function
\begin{align}
\label{eq:defw-10}
\calW^{-,0}(v,s) \,:=\, \det(v)^{1/2} e^{-2\pi v_1} \cdot \tilde I((e_3,0),v,P^{0}_{-s}(\cosh t ))
\end{align}
is a Whittaker function of index $\kzxz{-1}{0}{0}{0}$ and weight $1/2$. More precisely, it has the following properties:
\begin{itemize}
\item[(i)] it has moderate growth,
\item[(ii)] it satisfies 
\[
D_i \left( \calW^{-,0}(v,s) \ebf( u_1) \right) \= \lambda_i\cdot \calW^{-,0}(v,s) \ebf(\tr u_1)
\]
with eigenvalues $\lambda_i$ as in Proposition \ref{prop:eigenvals},
\item[(iii)] it is equal to 
\[
\calW^{-,0} \,:=\, \det(v)^{1/2} (4 v_1)^{-3/4} \pi^{1/4} W_{-1/4,1/4-s/2}(4\pi v_1) ,
\]
where $W_{\lambda,\mu}(z)$ denotes the classical Whittaker function, see \eqref{eq:Whit}.
\end{itemize}
\end{theorem}
\par
\begin{proof}
The properties (i)-(ii) follow immediately from Proposition \ref{prop:eigenvals2} and Remark \ref{rem:orbitwhittaker}.  It remains to prove (iii). In fact, we will show slightly more generally, that if  $f:\calD\to \C$ is any smooth eigenfunction of $\Delta$ with eigenvalue $\lambda= s(1-s)$, then
\ba
\label{eg:w-10}
\tilde I((e_3,0),v,f)
&\=   f(z_0)  (4 v_1)^{-3/4}\pi^{1/4}e^{2\pi v_1} W_{-1/4,1/4-s/2}(4\pi v_1)  .
\ea 
The assertion then follows by specializing $f$ to the function  $P^{-0}_{-s}(\cosh(t))$ in geodesic polar coordinates $(\theta,t)$ around $z_0$. To show \eqref{eg:w-10}, we compute
\begin{align*}
\tilde I((e_3,0),v,f)&\= \int_{\H} f(z)   e^{-4\pi \tr( Q((e_3,0)_{z^\perp})v)}\, d\mu(z)\\
&\=\int_{\H} f(z)   e^{-4\pi  Q((e_3)_{z^\perp})v_1 }\, d\mu(z)\\
&\=2\int_{t>0} \int_{\alpha=0}^\pi  f(k_\alpha a_t z_0)   e^{-4\pi Q( (a_t^{-1} k_\alpha^{-1} e_3)_{z_0^\perp})v_1}\, 
\sinh(t)\, d\alpha\, dt,
\end{align*}
where we have used geodesic polar coordinates around $z_0$ as in the proof of Theorem \ref{thm:w1} in the latter equality.
It is easily seen that 
\begin{align*}
k_\alpha^{-1} e_3 &\= e_3,\\
a_t^{-1}e_3 &\= \sinh(t) e_2 + \cosh(t) e_3,
\end{align*}
and therefore
\[
(a_t^{-1} k_\alpha^{-1} e_3)_{z_0^\perp}) \= \sinh(t) e_2.
\]
Inserting this, we obtain by means of \eqref{eq:tcoeff} and Proposition~\ref{prop:Fay0} that
\begin{align*}
\tilde I((e_3,0),v,f)&\=2 \int_{t>0} \int_{\alpha=0}^\pi  f(k_\alpha a_t z_0)  e^{-4\pi \sinh^2(t)v_1}\, 
\sinh(t)\, d\alpha\, dt\\
&\=2\pi  f(z_0) \int_{t>0} P_{-s}^0(\cosh t)  e^{-4\pi \sinh^2(t)v_1}\, 
\sinh(t)\, dt\\
&\=\pi  f(z_0)   \int_{w>0} P_{-s}^0( \sqrt{1+w}) \frac{e^{-4\pi v_1 w}}{ \sqrt{1+w}}\, dw.
\end{align*}
Inserting~\eqref{eq:IntPgivesW} we obtain \eqref{eg:w-10}, concluding the proof of the theorem.  
\end{proof}
\par
\begin{remark}
\label{rem:w22}
Let $B\in \GL_2(\R)$ such that $T ={}^t B\kzxz{-1}{0}{0}{0}B$ as before. 
It follows that the function 
\ba
\label{eq:defw-negsemi}
\calW_{T}^{-,0}(v,s) := (\det B)^{-1}\,\calW^{-,0}(Bv\, {}^tB,s)
\ea
is a Whittaker function of index $T$ and weight $1/2$ with the same eigenvalues as in Theorem~\ref{thm:w-10} (iii).
This function is actually independent of the choice of $B$, justifying the notation.
\end{remark}
\par
\begin{theorem}
\label{thm:negsemidef}
Let $B\in \GL_2(\R)$ such that $T={}^t\!B \kzxz{-1}{0}{0}{0}B$. Let $f$ be a Maass form of weight $0$ for $\Gamma$ with eigenvalue $\lambda= s(1-s)$. 
Then 
\begin{align}
\label{eq:negsemidef}
I_{T}(v,f) &\=  \calW^{-,0}_{T}(v,s)\cdot 
\sum_{x \in R_T} \frac{2}{|\Gamma_x|} f(z_x),
\end{align}
where $z_x\in \calD$ denotes the Heegner point determined by the negative line in $V(\R)$ generated by~$x$.
\end{theorem}
\par
\begin{proof}
In terms of orbital integrals we have 
\begin{align*}
I_T(v,f)&\=(\det v)^{1/2} e^{2\pi  \tr Tv }\sum_{\substack{x\in R_T}} I(x,v,f)\\
&\= (\det v)^{1/2} e^{2\pi  \tr Tv }\sum_{\substack{x\in R_T}} \frac{2}{|\Gamma_x|}\tilde I(x,v,f).
\end{align*}
We choose for $x\in R_T$ an element $g_x\in H(\R)$ such that 
\[
x\=g_x(e_3,0)B.
\]
According to Proposition~\ref{prop:Ipos}, we have 
\begin{align*}
\tilde I(x,v,f) &\= \tilde I( g_x(e_3,0)B,v,f) \= \tilde I( (e_3,0),Bv\,{}^t B,f\mid g_x).
\end{align*}
This reduces the computation of the orbital integrals to the case when $Q(x)=\kzxz{-1}{0}{0}{0}$.
By means of Theorem~\ref{thm:w-10}, we obtain
\begin{align*}
I_T(v,f) & \= \sum_{\substack{x\in R_T}} \frac{2}{|\Gamma_x|} (f\mid g_x)(z_0) \cdot \det(B)^{-1}\calW^{-,0}(Bv\,{}^tB,s).
\end{align*}
Here we have also used the fact that the uppper left entry of $Bv\,{}^tB$ is equal to $-\tr Tv$.
Since $(f\mid g_x)(z_0)=f(z_x)$, this implies the assertion.
\end{proof}
\subsubsection{ $T$ {positive} semidefinite.}
For $x = g_x (e_1, 0) B$ with $g_x \in H(\R)$ and $B \in \GL_2(\R)$, we can write the stabilizer as $\Gamma_x = \{g_x a_\hv g_x^{-1}: \hv \in \Z \cdot \hv_0\}$ with $\hv_0 \in \R_{> 0}$ uniquely determined by~$x$. Applying Proposition~\ref{prop:Ipos} gives us
\begin{align*}
  I( x, v ,f)
  &\=   I( (e_1, 0), v',f_x)
     \= \int_\R \int_{ \R/\Z\cdot \hv_0} f_x(a_\hv \rho_t\cdot i)   d\hv~
     e^{-4\pi\cosh(2t)^2 v'_1}   2 \cosh(2t) dt\\
  &\= 2\hv_0 \int_\R
    c_{0, f_x}(t) e^{-4\pi\cosh(2t)^2 v'_1}  \cosh(2t) dt\\
   &\= 2\hv_0 \gamma_{f_x, 0} \int_\R
     \Pb^+_{0, s} (\tanh(2t)) e^{-4\pi\cosh(2t)^2 v'_1}
 (\cosh 2t)^{1/2} dt.
\end{align*}
where $v' := B v {}^tB = \kzxz{v'_1}{*}{*}{*}$.  Here we have used
$$ \tr( Q((e_1, 0)_{(a_\hv \rho_t \cdot i)^\perp})v) \=
Q((e_1)_{(\rho_t \cdot i)^\perp})v_1 \= \cosh(2t)^2 v_1,$$
Proposition~\ref{prop:cnf}, and the oddness of $\Pb^-_{n, s}$. The quantity $\gamma_{f_x, 0}$ is defined in \eqref{eq:Cnf-0} and rewritten as cycle integral in Remark~\ref{rmk:cfx0}. Using the integral formula
$$
\int_\R \Pb^+_{0, s}\lp \frac{u}{\sqrt{u^2 + v'}} \rp
e^{-4\pi u^2} \frac{4 du}{(1 + u^2/v')^{1/4}}
\=
\frac{e^{2\pi v'}}{(4\pi v')^{1/4}} W_{1/4, 1/4 - s/2}(4\pi v')
$$
(which can be proven by checking that both sides satisfy the same differential equation and by computing the asymptotics $v' \to \infty$ by using dominated convergence to determine the proportionality constant) we obtain
\begin{equation}
  \label{eq:KS-pos}
  I(x, v, f) \=
  \hv_0 \gamma_{f_x, 0}  
   \frac{e^{-2\pi v_1}}{(4\pi)^{1/4} v_1^{3/4}} W_{1/4, 1/4-s/2}(4\pi v_1).
 \end{equation}
We define as in \eqref{eq:defw-negsemi} the functions
 \begin{equation}
   \label{eq:defw-possemi}
   \begin{split}
     \calW^{+,0}(v,s)
     &:= \det(v)^{1/2}     (v_1)^{-3/4} (4\pi)^{-1/4} W_{1/4,1/4-s/2}(4\pi v_1) ,\\
     \calW^{+,0}_{T}(v,s)
     &:= (\det B)^{-1}\,\calW^{+,0}(Bv\, {}^tB,s),
   \end{split}
 \end{equation}
 for $T ={}^t B\kzxz{1}{0}{0}{0}B$.
 These are Whittaker functions of positive semi-definite index $T$ and weight $\frac12$ satisfying the conditions in Theorem \ref{thm:w-10}.
 Furthermore they are independent of the choice of $B$. 
 Summing Equation \eqref{eq:KS-pos} over $x \in R_T$ gives us the following analogue of Theorem \ref{thm:negsemidef}. 
 \begin{theorem}
   \label{thm:possemidef}
   Let $f$ be a Maass form of weight $0$ for $\Gamma$ with eigenvalue $\lambda= s(1-s)$. 
Then 
\begin{align}
\label{eq:possemidef}
I_{T}(v,f) &\=  \calW^{+,0}_{T}(v,s)\cdot 
             \sum_{x \in R_T} \hv_0 f(C_x)
\end{align}
where $f(C_x)$ is the cycle integral from Remark~\ref{rmk:cfx0}. 
\end{theorem}

\section{The action of the Hecke algebra} \label{sec:Hecke}

The goal of this section is to show in Proposition~\ref{prop:HeckeIandf} that the action of the Hecke operators on the theta integrals can be expressed in terms of the usual Hecke action on the Maass forms. As usual, we mainly focus on $r=2$, but for comparison we also record the case $r=1$ and check compatibility under the Siegel operator. The method is an application of the strategy of Rallis-Zhuravlev. As in these sources we restrict \emph{throughout this section to primes $p \nmid 2D(B)\cdot \frakM$}.
\par
As guide to notation, straight letters $\wh{T}_{p,i}, T_{p,i}$ denote Hecke operators expressed in global Hecke rings, whereas calligraphic letters $\wh{\calT}_{p,i}, \calT_{p,i}$ denote Hecke operators expressed in local Hecke rings. Decorations by a hat denote Hecke operators on the metaplectic side, whereas undecorated $T$'s denotes Hecke operators on the orthogonal side, making use of the fact that we do not need here Hecke operators for the symplectic group.

\subsection{Structure results for Hecke algebras}

For a subgroup~$\Gamma \subset G$ together with a semigroup $S \subset G$ such that $(\Gamma, S)$ is a Hecke pair, we denote by $D(\Gamma, S) = D_\C(\Gamma, S)$ the associated Hecke algebra \cite[Chapter~3,~\textsection 1]{AndZhuBook}. We are interested in the structure of some of its subalgebras when~$G$ is metaplectic covering group $\Mp_{r, \R}$ and orthogonal group $\SO(V)$ for a quadratic space $V$ with odd dimension, such as the one in~\eqref{eq:V}. 

\subsubsection{The metaplectic Hecke algebra}

We recall the structure of various Hecke algebras related to the metaplectic covering group $\Mp_{r, \R}$, following mainly \cite{AndZhuBook} and \cite{ZhuMult}.
There is a global version, which acts modular forms, and a local version, which acts on Schwartz functions. They are compatible, and we start with the former.
\par
For $N \in 4\N$ and any prime $p \nmid N$, define
\begin{equation}
  \label{eq:S0}
  S_0 := \left\{
    \zxz{A}{B}{C}{D} \in \Sp_r(\Z[p^{-1}]):  C \equiv 0 \bmod N\right\}
\end{equation}
and let $\Gamma_0 = \Gamma^r_0(N) := S_0 \cap \Sp_r(\Z) \subset \Sp_r(\Z)$ be the congruence subgroup of level~$N$. Both $(\Gamma_0, S_0)$ and  their preimages $(\widehat\Gamma_0, \widehat S_0)$ in $\Mp_{r, \R}$ under the map $\p$ in \eqref{eq:MpSp} form a Hecke pair. In the Hecke ring $D(\widehat\Gamma_0, \widehat S_0)$ we consider the subring $\wh{E}_p = \widehat E_p(N)$ generated by
\begin{equation}
  \label{eq:Ti}
  \widehat T_{p,i}
= \widehat T^{(r)}_{p,i}
  := \widehat\Gamma_0 \widehat K_i \widehat\Gamma_0,~
  \widehat K_i := (K_i, p^{(r-i)/2}),\quad 
  K_i := \frac1p \mathrm{diag}(p E_i, E_{r-i}, pE_i, p^2E_{r-i})
\end{equation}
for $0 \le i \le r$ \cite[\textsection~8.3]{ZhuMult}, where $E_m := (1, \dots, 1) \in \Z^m$.
The elements in this ring act on Siegel modular forms (or automorphic forms) in the usual way, which we denote using the slash operator. 

\par
If we replace $\Z \subset \Z[p^{-1}]$ in the definition of $\Gamma_0 \subset S_0$ by $\Z_p \subset\Q_p$, the corresponding pair in $\Sp_r(\Q_p)$, denoted by $(\Gamma_p, S_p)$, is again a Hecke pair. We denote their preimages in $\Mp_{r, \Q_p}$ by $\widehat\Gamma_p \subset \widehat S_p$, which also form a Hecke pair. 
The subalgebra $\widehat \calE_p \subset D(\widehat \Gamma_p , \widehat S_p)$ generated by
\begin{equation}
  \label{eq:Tpi}
  \widehat \calT_{p, i} \= \widehat \calT^{(r)}_{p, i}
  \,:=\, \widehat\Gamma_p \widehat K_{p,i} \widehat\Gamma_p, \qquad
  \widehat K_{p,i} \,:= \, [K_i, 1] \in \Mp_{r, \Q_p}
\end{equation}
is a local Hecke algebra.

\begin{lemma}
  In the notations above with $p \nmid N \in 4\N$, the map
  \begin{equation}
  \label{eq:epsilon}
  \epsilon: \widehat E_p(N) \mapsto \widehat \calE_p, \qquad 
  \widehat T_{p,i} \mapsto
  \widehat \calT_{p, i}
\end{equation}
  is an isomorphism.
  Furthermore, we have
  $$
  \theta^{(r)}(g, h, \varphi) \mid \widehat T
  =
  \theta^{(r)}(g, h, \omega_p(\epsilon(\widehat T))\varphi)
  $$
  for all $\widehat T \in \widehat E_p(N)$. 
\end{lemma}

\begin{proof}
  For this isomorphism see the equation above~(8.12) in \cite{ZhuMult}, but note that the factor $\chi_Q(p^{n-i})$ loc.\ cit.\ is not present since the slash operator in (8.2) of \cite{ZhuMult} differs from ours in \eqref{eq:siegeltr} by an extra factor $\chi_Q(M)$. 
The second claim follows from Equations (8.6), (8.7) and (8.12) in \cite{ZhuMult}.
  \end{proof}
\par
We denote by $C_r := \C[x_1^{\pm 1},\ldots,x_r^{\pm 1}]_{W_r}$ the invariant ring of
the action of the Weyl group $W_r = W(\Sp_r)$. Recall that this group is the semidirect product of the symmetric group (with its usual permutation action on the $x_i$) with a group isomorphic to~$(\Z/2\Z)^r$ consisting of an involution for each~$i$ (sending $x_i \mapsto x_i^{-1}$ and fixing the other variables).
This is identified with the local Hecke algebra $\widehat \calE_p$ via the following result.
\par
\begin{proposition}[Theorem 8.1 in \cite{ZhuMult}]
\label{prop:Omega}
The  Satake map 
\be
\wh{\calS}_r : \widehat \calE_p  \to C_r
\ee
sending a coset representative in $\widehat\Gamma_p (M,t) \in\Mp_{r, \Q_p}$ with  
\bes
M = \left(\begin{matrix} A & B \\ 0 & D \end{matrix} \right) \quad
\text{and} \quad 
D = \left(\begin{matrix} d_1 & * & * \\ 0 & \ddots & * \\
0 & 0 & d_r\end{matrix} \right)
\quad \text{to} \quad
t^{-1} \prod_{j=1}^r (x_j \cdot p^{-j} )^{\nu_p(d_j)}\
\ees
is an isomorphism 
  \end{proposition}
\par
Using the decomposition of $\widehat \calT_{p, i}$ in (8.10) of \cite{ZhuMult}, we obtain the following consequence of Proposition~\ref{prop:Omega}. 
\begin{proposition}
  \label{prop:symp-Hecke}
  In the notations above, we have 
  \begin{equation}
    \label{eq:symp-Hecke1}
    \begin{split}
      \wh{\calS}_1(\widehat{\calT}^{(1)}_{p, 0})
      &= p(x_1 + x_1^{-1})   ,\\      
      \wh{\calS}_2(\widehat{\calT}^{(2)}_{p, 0})
      &= p^3(x_1x_2 + x_1^{-1}x_2 + x_1x_2^{-1} + x_1^{-1}x_2^{-1}) + p^2(p-1) + p(p-1)\varepsilon_p^{-m},\\
      \wh{\calS}_2(\widehat{\calT}^{(2)}_{p, 1})
      &= p^2(x_1 + x_1^{-1} + x_2 + x_2^{-1}),
    \end{split}
\end{equation}
where $m = \dim(V)$ and $\varepsilon_p := 1$, resp.\ $i$, if $p \equiv 1 \bmod 4$, resp.\ $-1 \bmod 4$.
\end{proposition}
\begin{proof}
  We first express $\widehat{\calT}^{(r)}_{p,i}$ using coset representatives $\widehat M_{a, b}(B_0, S, V)$ as in (8.8)-(8.11) of \cite{ZhuMult}.
  For $r = 1, i = 0$, there are three possibilities for $(a, b)$. When $a = b = 0$, the indices $B_0, S, V$ are all trivial, and the image of this coset under $\wh{\calS}_r$ is $px_1^{-1}$. When $a = 0, b= 1$, the indices $B_0, V$ are trivial and $S \in \Z/p^2\Z$. Their images under $\wh{\calS}_r$ gives $px_1$.
  When $a = 1, b = 0$, the indices $S, V$ are trivial and $B_0 \in (\Z/p\Z)^\times$. Its image under $\wh{\calS}_r$ is $(\varepsilon_p \frac{\det(Q) B_0}p)$, and the sum is 0. This proves the first equation in \eqref{eq:symp-Hecke1}. 
\par
  For $r = 2, i = 0$, there are 6 cases for $(a, b)$. When $a = b =0$, $B_0, S, V$ are trivial and we get $p^3(x_1x_2)^{-1}$. Its symmetrization comes from the contribution of $a = 0, b= 2$, where $B_0, V$ are trivial and $S \in \mathrm{Sym}_2(\Z/p^2\Z)$.
  When $a = 0, b= 1$, $B_0$ is trivial, $S \in \Z/p^2\Z$ and
  $$
  V \in
  \left\{\zxz{1}{j}{}{1}: 0 \le j \le p^2-1 \right\}
\cup 
\left\{w\zxz{1}{pj}{}{1} : 0 \le j \le p-1 \right\}
\subset \SL_2(\Z)
  $$
  are representatives of $\Gamma^0(p^2)\backslash \SL_2(\Z)$. 
  The contribution is $p^3(x_1x_2^{-1} + x_1^{-1}x_2) + p^2(p-1)$. When $a = 1$ and $b = 0$ or $1$, we have $B_0 \in (\Z/p\Z)^\times$ and the contribution is 0 as in the case $r = 1, i = 0, a = 1, b = 0$ above. The last case with $a = 2, b = 0$, we have $S, V$ trivial and $B_0 \in \mathrm{Sym}_2(\Z/p\Z)$ non-singular. The contribution is $\varepsilon_p \sum_{B_0} ( \frac{\det B_0}p)$, which can be evaluated as $p(p-1) \varepsilon_p$.
\par
  The case $r = 2, i = 1$ is similar and we omit the proof. 
\end{proof}

\subsubsection{The orthogonal Hecke algebra}

Let $(V_p, Q_p)$ be a $\Q_p$-vector space of odd dimension~$m$ and
Witt index $\nu = (m-1)/2$ such that its discriminant has a representative in $\Z_p^\times$. 
Then there is a basis $f_1,\ldots,f_m$ such that the quadratic form has the shape
\be
Q_p = \left(\begin{matrix} 0 & E_\nu & 0 \\ E_\nu & 0 & 0 \\
0 & 0 & \epsilon \end{matrix} \right)
\ee
with respect to this basis for some $\epsilon \in \Z_p^\times$.
Let $L_p \subset V_p$ be a $\Z_p$-lattice generated by this basis.
Then the pair $U_p := O(L_p) \subset O_p := O(V_p)$ is a Hecke pair and we now describe the structure of the local Hecke algebra $D(U_p, O_p)$ following \cite[Section~8.7]{ZhuMult}.
\par
Every operator $\calT \in D(U_p, O_p)$ can be written as sum of $U_p$-left cosets of the form (see \cite[(8.23)]{ZhuMult} for details)
$$
\calT = \sum_{M} U_p M,~ M =
\left(
\begin{matrix}
  A & * & *\\
  0 & D & 0 \\
  0 & * & 1
\end{matrix}
\right),~
D =
\left(
\begin{matrix}
  d_1 &  & *\\
   & \ddots &  \\
  0 &  & d_\nu
\end{matrix}
\right).
$$
The Satake map $\calS: D(U_p, O_p) \to C_\nu$ defined by
\begin{equation}
  \label{eq:omega-iso}
  \calS(U_p M) := \prod_{i = 1}^\nu |d_i|_p^{t_i + i - 1/2},~ x_i = p^{-t_i}
\end{equation}
is an isomorphism \cite[(8.24)]{ZhuMult}.
\par
For $V$ as in \eqref{eq:V} with lattice $L = L_\frakM$ and $p \nmid 2D(B)\frakM$, we can take $V_p = V(\Q_p)$ and an isometry $L_p \cong M_2(\Z_p)^0$ which identifies $\operatorname{SO}(V_p) \cong \mathrm{PGL}_2(\Q_p)$ and $\mathrm{SO}(L_p) \cong K_p := \mathrm{PGL}_2(\Z_p)$.
With respect to the basis
$f_1 = \kzxz0100, f_2 = \kzxz0010, f_3 = \kzxz100{-1}$, 
the Hecke operator $\calT_p := K_p \kzxz{p}{}{}{1} K_p
= K_p  \kzxz{p}{}{}{1}
\cup_{j \in \Z/p\Z} K_p \kzxz{1}{j}{}{p}$ can then be written as
\begin{equation}
  \label{eq:OTp}
  \calT_p \=
U_p
\left(
\begin{matrix}
  p & 0 & 0\\
  0 & p^{-1} & 0 \\
  0 & 0 & 1
\end{matrix}
\right)
+
\sum_{j \in \Z/p\Z}
U_p
\left(
\begin{matrix}
  p^{-1} & -pj^2 & -2j\\
  0 & p & 0 \\
  0 & pj & 1
\end{matrix}
\right),
\end{equation}
and consequently by \eqref{eq:omega-iso}
\be  \label{eq:OTp2}
\calS(\calT_p) \= \sqrt{p}(x_1 + x_1^{-1}).
\ee

\subsection{The Rallis-Zhuravlev map}

The Rallis map is a surjection between symplectic and orthogonal Hecke-Algebras that intertwines the action on the Weil representation. It was computed by Rallis \cite{Rallis} for $m$ even and extended by Zhuravlev to $m$ odd in \cite{ZhuLocal} (see also \cite[Section~8.9]{ZhuMult}). We state here the case where $r \geq \nu$. Define the map
\be
\wt{\mathrm{RZ}}_{r,\nu}: C_r \to C_\nu
\ee
between Weyl group invariants of Laurent polynomials in~$r$ resp.~$\nu$ variables by
\bes
x_i \mapsto x_i, \quad (i=1,\ldots,\nu), \qquad x_i \mapsto p^{i-\frac{m}2}, \quad
(i=\nu+1,\ldots,r)
\ees
and let ${\mathrm{RZ}}_{r,\nu}$ be the induced map such that the diagram
\begin{center}
  \begin{tikzcd}
  \widehat \calE_p \arrow{r}{{\mathrm{RZ}}_{r,\nu}} \ar{d}{\wh{\calS}_r}
  &
D(U_p, O_p)
\arrow{d}{\calS} \\
C_r
\arrow{r}{\wt{\mathrm{RZ}}_{r,\nu}} &
C_\nu
	\end{tikzcd}
\end{center}
commutes. We recall the conclusion from \cite[Theorem~1.1 and Section~8.4]{ZhuLocal} or \cite[Theorem~8.2]{ZhuMult}.
\par
\begin{proposition} \label{prop:HeckeIandf}
  Let $\theta^{(r)}(g, h, \varphi)$ be the theta function defined in \eqref{eq:thetaWeilrep} and $\omega = \otimes_p \omega_p$ the Weil representation from Section~\ref{sec:WeilrepSiegelTheta}. Suppose that $p\nmid 2D(B)\cdot \frakM$. Then
  \begin{equation}
    \label{eq:Hecke-theta-comm}
    \theta^{(r)}(g, h, \omega_p(\calT)\varphi)
    \=
    \theta^{(r)}(g, h, \omega_p(\mathrm{RZ}_{r, 1}(\calT))\varphi)    
  \end{equation}
  for all $\calT \in \widehat \calE_p$.
  In particular, we have
  \begin{equation}
    \label{eq:Hecke-comm}
    \begin{split}
 p^{-3/2}   \,   \widehat T^{(1)}_{p,0} I^{(1)}(\tau, f)
      &\=
    I^{(1)}(\tau, T_p f),\\
p^{-1}    \,  \widehat T^{(2)}_{p,0} I^{(2)}(\tau, f)
      &\=
         I^{(2)}(\tau, (p^2(p+1)T_p + (p-1)(p+\varepsilon_p)) f),\\
p^{-3/2}   \,   \widehat T^{(2)}_{p,1} I^{(2)}(\tau, f)
      &\=
        I^{(2)}(\tau, (pT_p + p+1) f),
    \end{split}
  \end{equation}
  where $T_p$ is the usual Hecke operator on modular forms normalized as in \cite[Chapter~III, (5.15)]{Koblitz}
\end{proposition}
\begin{remark}
  \label{rmk:Tp} 
  On the orthogonal side, the Hecke operator $\calT_p$ on automorphic forms is $p^{1-k}$ times the Hecke operator $T_p$ on the associated  modular forms (see (2.13) in \cite{LZ25}).
  \par
  On the symplectic side for $r = 1$, the Hecke operator $\widehat T_0$ is $p^{3/2}T_{p^2}$, where $T_{p^2}$ is the usual Hecke operator on half integral weight modular forms as in \cite[Chapter IV,~(3.3)]{Koblitz}.
\end{remark}
\begin{proof}
  Equation~\eqref{eq:Hecke-theta-comm} follows from the definition of $\mathrm{RZ}_{r, \nu}$. By Proposition~\ref{prop:symp-Hecke} and Equation~\eqref{eq:OTp2} we have $ \mathrm{RZ}_{1, 1}(\wh{T}_{p, 0})     = \sqrt{p} \calT_p$ and
  \begin{equation}
    \label{eq:Hecke-comm1}
    \begin{split}
      \mathrm{RZ}_{2, 1}(\wh{T}_{p, 0})
      = p^2(p+1) \calT_p + p^2(p-1) + 2p \varepsilon_p,\quad
      \mathrm{RZ}_{2, 1}(\wh{T}_{p, 1})
      = p^{3/2}(\calT_p + (p+1)).
    \end{split}
  \end{equation}
By a change of variable in the theta integral, we can move the action of $\calT_p$ on Schwartz function in the theta kernel to an action of $\calT_p$ on $f$. 
  Then the isomorphism \eqref{eq:epsilon} and the first observation in Remark \ref{rmk:Tp} finishes the proof. 
\end{proof}
\par
Lastly, we give the commutation relation between Hecke operators and the Siegel-$\Phi$ operator. In the context of holomorphic Siegel modular forms this is also known as the Zharkovskaia commutation relation \cite[Chapter~4, Theorem~2.12]{AndZhuBook}, which is compatible with Proposition~\ref{prop:HeckeIandf} (see Remark~\ref{rmk:HeckePhi}). 
\par
\begin{proposition}
  \label{prop:HeckePhi}
  Let $\Phi$ be the Siegel-$\Phi$ operator 
  and $F(\tau) :=  I^{(2)}(\tau, f)$ a Siegel-Maass cusp form of weight $\frac12$ as in Proposition \ref{prop:HeckeIandf}.
  Then for $p \nmid 2D(B)\frakM$, the following commutation relation between the Siegel~$\Phi$-operator (as defined in \eqref{eq:SiegelActOnThetaLift}) and the metaplectic Hecke operators hold:
  \begin{equation}
    \label{eq:HeckePhi}
    \begin{split}
\Phi\lp  \widehat T^{(2)}_{p,0} F\rp
      &\=
p^{3/2}(p+1) \widehat T^{(1)}_{p,0} \Phi(F)
        + p(p-1)(p+\varepsilon_p) \Phi(F),\\
\Phi\lp  \widehat T^{(2)}_{p,1} F\rp
      &\=
p \widehat T^{(1)}_{p,0} \Phi(F)
        + p^{3/2}(p+1) \Phi(F).
    \end{split}    
  \end{equation}
\end{proposition}

\begin{remark}
  \label{rmk:HeckePhi}
  Using the action of the Siegel-$\Phi$-Operator on theta functions in~\eqref{eq:Phir} and the first equation in \eqref{eq:Hecke-comm}, we see that the relations in \eqref{eq:HeckePhi} are compatible with the last two equations in \eqref{eq:Hecke-comm} as
  \begin{align*}
\Phi\lp p^{-1}    \,  \widehat T^{(2)}_{p,0} I^{(2)}(\tau, f)\rp
      &\=
        p^{-1}    \Phi\lp  \widehat T^{(2)}_{p,0} F\rp
        =
p^{1/2}(p+1) \widehat T^{(1)}_{p,0} \Phi(F)
        + (p-1)(p+\varepsilon_p) \Phi(F),\\
        &=
          p^{1/2}(p+1) \widehat T^{(1)}_{p,0}
 I^{(1)}(\tau, f)
        + (p-1)(p+\varepsilon_p)  I^{(1)}(\tau, f),\\
        &=
          I^{(1)}(\tau,           (p^{2}(p+1) T_p          + (p-1)(p+\varepsilon_p))  f)\\
      &= \Phi\lp
                  I^{(2)}(\tau,           (p^{2}(p+1) T_p          + (p-1)(p+\varepsilon_p))  f)\rp.
  \end{align*}
\end{remark}
\par
\begin{proof}
  Denote $F_1 := \Phi(F)$. From definition (see Section~8 in \cite{ZhuMult}), we
can decompose the action of the genus one Hecke operator into the three terms
 have
  \bas
    \label{eq:T1p}
&    \lp  \widehat T^{(1)}_{p,0} F_1\rp(\tau_1)
      \= C_0 + C_1 + C_2\,, \qquad \text{where}\\
&      C_0
      :=p^{1/2} F_1(p^2 \tau_1),\quad
        C_1 := \varepsilon_p^{-1} \sum_{B \in (\Z/p\Z)^\times} \lp \frac{B}p \rp F_1 \lp \tau_1 + \frac{B}p \rp,\quad
        C_2 := p^{-1/2} F_1(\tau_1/p^2). 
        \eas
Using the definition of $\Phi$, we can directly check  (see e.g.\ the proof of Theorem~2.12 in \cite{AndZhuBook} on page~246) that
  \begin{equation}
    \label{eq:PhiExp}
    \Phi\lp  F(m(A^*) n(B) \cdot \tau)\rp
    \= |a| \cdot F_1(m(A_1^*) n(B_1) \cdot \tau_1),
  \end{equation}
  when $A = \kzxz{A_1}{*}{0}{a}, B = \kzxz{B_1}***, \tau = \kzxz{\tau_1}***$, where $A^* := {}^tA^{-1}$.

  As in the proof of Proposition~\ref{prop:symp-Hecke}, the coset decomposition of $  \widehat T^{(2)}_{p,0}$ has six cases according to the parameters $a, b$. 
  One can use their descriptions in \cite[Section~8.3]{ZhuMult} and~\eqref{eq:PhiExp} to compute the corresponding contribution to $\Phi\lp  \widehat T^{(2)}_{p,0} F\rp$. They are given by
$$ \begin{array}{|c|c|c|c|c|c|c|}
\hline &&&&&& \\ [-\halfbls]
(a,b)
& (0,0 & (1,0)& (0,1) & (2,0) & (1,1) & (0,2)\\
[-\halfbls] &&&&&&\\
\hline &&&&&& \\ [-\halfbls]
& p^{3/2} C_0 &  p^{3/2}C_1& p^{5/2}C_0 + p^{3/2} C_2 + p^2(p-1)F_1
& p(p-1)\varepsilon_p F_1 &  p^{5/2}C_1
 & p^{5/2} C_2\\
[-\halfbls] &&&&&&\\
\hline
\end{array}
$$
Adding them together produces the first identity in \eqref{eq:HeckePhi}. The second is verified similarly.
\end{proof}
\par

\section{A computation using the Siegel-Weil formula}
\label{sec:SW}


In this section we consider the theta integral $I(\tau,f)$ for the constant function $f=1$. Our general results of Section~\ref{sec:Thetalift} imply in this case that the theta integral is a Siegel Maass form of weight $1/2$ and that  the cardinalities of the sets $R_T$ occur as (non-archimedean parts of) its Fourier coefficients.
We give a different proof of this fact by means of the Siegel-Weil formula. It identifies the theta integral as a specific Siegel Eisenstein series. We compute its Fourier coefficients using explicit formulas of Kudla and Yang for local Whittaker functions. This leads to a new proof and a generalization of the result of Rickards on the number of intersection points of $b$-linked geodesics, stated as Theorem~\ref{thm:rickardscount} in Section~\ref{sec:setup}.

\subsection{Siegel Eisenstein series}
\label{sect:siegeleis}

We begin by recalling some facts about Siegel Eisenstein series following \cite{KRY-book} and \cite{BY}, specialized to genus $r=2$. Recall the notation set up for the Weil representation in Section~\ref{sec:WeilrepSiegelTheta}. In particular, $F = \Q_p$ for a place $p \leq \infty$ or the ring~$\A$ of adeles of~$\Q$.
For a character~$\chi$ of~$F^\times$ and $s\in \C$, let $I(s, \chi)$ be the degenerate principal series representation of~$\Mp_{2,F}$ consisting of smooth functions $\Phi(s)$ on~$\Mp_{2,F}$ satisfying
\[
\Phi([n(b) m(a),z]g, s)
 \= 
 z\,\chi(\det a) \gamma(\det a, \frac{1}2 \psi)^{-1} |\det a|^{s+\rho_2} \Phi(g, s),
 \]
where $\rho_2 \= \frac{3}{2}$ and $m(a), n(b) \in \Sp_2(F)$ are as in the beginning of Section~\ref{sec:WeilrepSiegelTheta}. Such a section~$\Phi$ is called standard, if its restriction to $\wt{K}_{F}$ is independent of~$s$. Here, for $p<\infty$, we denote by $\wt{K}_{F}$ the preimage of $\Sp(\Z_p)$ under the  double covering map $\prM: \Mp_{2,F} \to \Sp_2(F)$, and for $p=\infty$ the preimage of the maximal compact subgroup $\Uni(2)$, and the product of these if $F =\A$.
\par
If $T \in \Sym_2(F)$ is a symmetric matrix, the Whittaker function of $\Phi$ with respect to $T$ is the function 
\begin{equation} \label{eq:Whittaker}
W_{T}(g, s, \Phi) \= \int_{\Sym_2(F)} \Phi(w n(b) g, s) \psi(-\tr(T b) )\,d_\psi b,
\end{equation}
on $\Mp_{2,F}$, where $d_\psi b$ is the self-dual Haar measure on $\Sym_2(F)$ with respect to the pairing $(b_1,b_2)\mapsto \psi(\tr (b_1b_2))$. It has the transformation behavior
\begin{align}
\label{eq:i1}
W_T(n(b) m(a)g,s,\Phi)&\=\psi(\tr(Tb))\chi(a)^{-1}
|a|^{\rho_2-s} \gamma(a, \frac{1}2\psi) \cdot W_{{}^{t}aTa}(g,s,\Phi).
 \end{align}
Here we have shortened $\gamma(a, \frac{1}2\psi) = \gamma(\det a,\frac{1}2\psi)$  and similarly for $\chi(a)$. When $F=\R$ and $\det a >0$, then $\gamma(a, \frac{1}2\psi) =1$. If $F=\Q_p$ we usually write $\Phi_p$ for sections of $I(s,\chi)$ for $p \leq \infty$, and similarly write $W_{T,p}(g,s,\Phi_p)$ for the corresponding Whittaker functions.
\par
We now define Eisenstein series on $\Mp_{2,\A}$. Recall that there is a unique splitting homomorphism $\Sp_2(\Q)\to \Mp_{2,\A}$ to the covering $\prM: \Mp_{2,\A} \to \Sp_2(\A)$. We denote its image by~$G_\Q$ and define $P_\Q = \wt{P}_{\A} \cap G_\Q$. Let $\Phi=\prod_{p\leq \infty} \Phi_p\in I(s,\chi)$ be a standard factorizable section of the induced representation of $\Mp_{2,\A}$. The corresponding Eisenstein series is defined by
\begin{align}
\label{eq:defeis}
E(g,s,\Phi) \= \sum_{\gamma\in P_\Q\bs G_\Q}\Phi(\gamma g,s),
\end{align}
where $g\in \Mp_{2,\A}$. It converges normally for $\Re(s)>\rho_2$ and has a meromorphic continuation in~$s$ to the whole complex plane. Moreover, it satisfies a functional equation relating values at~$s$ and~$-s$. The Eisenstein series has a Fourier expansion of the form
\begin{align}
\label{eq:foureis}
E(g, s, \Phi) \= \sum_{T\in \Sym_n(\Q)} E_T(g,s,\Phi).
\end{align}
When  $T \in \Sym_n(\Q)$ is non-singular, the $T$-th Fourier coefficient factorizes,
\begin{align}
\label{eq:foureis2}
E_T(g, s,  \Phi) \= \prod_{p \le \infty} W_{T, p}(g_p, s, \Phi_p),
\end{align}
into  a product of local Whittaker functions as in \eqref{eq:Whittaker}.

\subsection{The Siegel-Weil formula} \label{sec:SWformula}

We use the Siegel-Weil formula to compute the theta integral $I(\tau,1)$ defined in \eqref{eq:thetaint} for the constant function $1$. Since $V$ is anisotropic over $\Q$ the theta integral converges. In this case its Fourier coefficients~\eqref{eq:f2} can be written as 
\begin{align}
\label{eq:IT1}
I_T(v,1) &\=(\det v)^{1/2} e^{2\pi  \tr Tv }\sum_{\substack{x\in R_T}} I(x,v,1)\\
\nonumber
&\=(\det v)^{1/2} \sum_{\substack{x\in R_T}}\int_{\calD}    e^{-2\pi \tr( (Q(x_{z^\perp}) - Q(x_z))v)}\, d\mu(z) \\
\nonumber
&\=(\det v)^{1/2} \sum_{\substack{x\in R_T}}\frac2{\pi}
  \int_{Z(\R)\bs H(\R)^+}
\varphi_\infty(v, h^{-1} x) dh,
\end{align}
where $\varphi_\infty(v, x) :=  e^{-2\pi \tr( (Q(x_{z_0^\perp}) - Q(x_{z_0}))v)}$ as in~\eqref{eq:StdSchwarz}, and $2dh = d\mu(z) d\theta$ is normalized as in \cite{Ku1}.
Here the factor $\pi^{}$ is the volume of $\mathrm{PSO}_2(\R) \subset \mathrm{PSL}_2(\R) \cong Z(\R)\backslash H(\R)^+$ since $\calD \cong \mathrm{PSL}_2(\R)/\mathrm{PSO}_2(\R)$.
\par
We now construct the standard factorizable section used as input for the Eisenstein series.  Recall that there is an $\Mp_{2,F}$-equivariant map
\begin{align}
\label{eq:deflambda}
\lambda:  S(V^2(F)) \rightarrow I(0, \chi_V),      
\quad \lambda(\phi)(g) \= \big(\omega(g) \phi\big)(0),
\end{align}
from the space of Schwartz functions to the induced representation associated with the quadratic character $\chi_V$ of $V$.
We will also denote by $\lambda(\phi) $ the associated standard section in $I(s, \chi_V)$ with $\lambda(\phi) (g, 0) =\lambda(\phi)(g)$. 
\par
Let $\wh{L} =L\otimes_\Z \wh\Z$, and write  $\varphi_L= \operatorname{char}(\wh{L}^2)\in S(V^2(\A_f))$ for the characteristic function of $\wh{L}^2$.
Applying the intertwining operator $\lambda$, we obtain a corresponding section $\lambda(\varphi_L)$ of the induced representation of $\Mp_{2,\A_f}$. At the archimedian place the Gaussian $\varphi_\infty\in S(V^2(\R))$ transforms in weight $1/2$ under $K_\infty$. Its image under the intertwining operator $\lambda$ is the standard section $\Phi_{\infty}^{1/2}$ of weight $1/2$ in the induced representation of $\Mp_{2,\R}$ (normalized such that $\Phi_\infty^{1/2}(1)=1$). The product 
\[
\lambda(\varphi_L)\otimes \Phi_\infty ^{1/2}
\]
defines a factorizable standard section in the induced representation of $\Mp_{2,\A}$.
We will often view the corresponding Siegel Eisenstein series 
\[
E(g, s, \lambda(\varphi_L) \otimes \Phi_\infty^{1/2})
\]
on $\Mp_{2,\A}$ as a function on the Siegel upper half space $\H_2$. 
To this end, for $\tau=u+iv\in \H_2$ with real part $u$ and imaginary part $v$, we choose $a\in \GL_2^+(\R)$ with $a\,{}^ta=v$ and let $g_\tau=n(u)m(a)\in \Sp_{2, \R}$.
Then the Eisenstein series
\begin{align}
\label{eq:eissw}
E(\tau, s,  \lambda(\varphi_L)\otimes \Phi_\infty^{1/2}) \= (\det v)^{-\frac{1}{4}}\cdot E(g_\tau, s,  \lambda(\varphi_L)\otimes  \Phi_\infty^{1/2})
\end{align}
is a (non-holomorphic) Siegel modular form of weight~$1/2$ and genus~$2$ for a congruence subgroup of~$\Mp_{2,\Q}$. The following result is now a consequence of the Siegel-Weil formula, see e.g.~\cite{KR}, \cite{We}.
\par
\begin{theorem}
  \label{thm:sw}
  Suppose $\Gamma = H(\Q)^+ \cap K \subset H(\Q)^+$ is an arithmetic subgroup such that the image of the open compact subgroup $K \subset H(\hat\Q)$ under the spinor norm is $\hat\Z^\times$. 
Let $\vol(X_\Gamma)$ be the volume of $X_\Gamma$ with respect to the measure $d\mu(z)$. 
The above Eisenstein series is holomorphic in $s$  at $s=0$. The theta integral $I(\tau,1)$ is given by
\[
\frac{2}{\vol(X_\Gamma)}I(\tau,1) =  E(\tau, 0, \lambda(\varphi_L) \otimes \Phi_\infty^{1/2}).
\]
In particular, for the Fourier coefficients in \eqref{eq:f0} we have
\begin{equation}
  \label{eq:IT0}
\frac{2}{\vol(X_\Gamma)}I_T(v,1) \= (\det v)^{-\frac{1}{4}} W_{T, \infty}(m(a), 0, \Phi_\infty^{1/2})\cdot \prod_{p<\infty} W_{T, p}(1, 0, \lambda(\varphi_{L_p})).
\end{equation}
\end{theorem}
\par
Using the local archimedean Siegel-Weil formula in \cite{Ku1}, we can express the size of the set~$R_T$  in~\eqref{eq:RTdef} in terms of values of local Whittaker functions at non-archimedean places. 
\par
Let $K=K_{\frakM}\subset  H(\A_f)$ by the compact open subgroup defined in \eqref{eq:KL} and put $\Gamma_{\frakM}= H(\Q)^+\cap K_{\frakM}$. Recall from \eqref{eq:decomp} that the corresponding Shimura curve $X_\frakM=X_{K_{\frakM}}$ is connected.
\par
\begin{theorem}
\label{thm:posdeg2}
Let $T$ be invertible, and let $R_T$ be the set of representatives defined in~\eqref{eq:RTdef}. 
Then 
\begin{equation} \label{eq:RT}
  |R_T| \= 2^{3/2}\pi
  \vol(X_{\frakM}) \prod_{p< \infty} W_{T, p}(1, 0, \lambda(\varphi_{L_p})).
\end{equation}
\end{theorem}
\par
\begin{proof}
We employ the local Siegel-Weil formula, which expresses orbital integrals associated with Schwartz-Bruhat functions in terms of local Whittaker functions. 
According to \cite[Corollary~5.3.5]{KRY-book} we have 
\begin{align}
\label{eq:p01}
\prod_{p< \infty} W_{T, p}(1, 0, \lambda(\varphi_{L_p})) \= C_\infty(V,\alpha,\beta,\psi)^{-1} \cdot \prod_{p< \infty} O_{T,p}(\varphi_{L,p})\,.
\end{align}
Here the local orbital integral on the right hand side is defined by 
\[
O_{T,p}(\varphi_{L,p}) \= \int_{\SO(V)(\Q_p)}\varphi_{L_p}(h^{-1} x_0) \,d_p h
\]
if there exists an $x_0\in V^2(\Q)$ with $Q(x_0)= T$, and by $0$ if there exists no such $x_0$. Here we have used that $\varphi_{L,p}$ is even, and therefore the $\SO(V)(\Q_p)$-orbital integral is equal to the 
$\Orth(V)(\Q_p)$-orbital integral. Finally the measures are normalized as in \cite[page~124]{KRY-book}, and the constant $C_\infty(V,\alpha,\beta,\psi)$ is defined in Proposition~5.3.3 of \cite{KRY-book}. It is independent of~$T$ and of the choice of the Schwartz function in the argument.
\par
Let $\bar K\subset \SO(V)(\A_f)$ be the compact open subgroup which is given by the image of $K=K_{\mathfrak{M}}\subset H(\A_f)$ under the projection $H(\A_f)\to \SO(V)(\A_f)$ and put 
\[
\bar\Gamma \,:=\, \SO(V)(\Q)\cap ( \SO(\R)\times\bar K) \= \calO_{\mathfrak{M}}^\times/\{\pm 1\}.
\]
Since $\calO_{\mathfrak{M}}^\times$ contains an element of reduced norm $-1$, the group $\Gamma_{\mathfrak{M}}/\{\pm 1\}=\calO_{\mathfrak{M}}^1/\{\pm 1\}$ is a subgroup of $\bar \Gamma$ of index $2$. Since $\SO(V)(\A_f) = \SO(\Q) \bar K$, 
\cite[Proposition~5.3.6]{KRY-book} implies that 
\begin{align*}
\prod_{p< \infty} O_{T,p}(\varphi_{L,p})&\= \vol(\bar K)\cdot \sum_{\substack{x\in V^2(\Q)/\bar\Gamma\\ Q(x)= T}}
\varphi_L(x)\\
&\= \vol(\bar K)\cdot \frac{1}{2}\sum_{\substack{x\in L^2/(\Gamma_{\frakM}/\{\pm 1\})\\ Q(x)= T}} 1\\
&\= \frac{1}{2}\vol(\bar K)\cdot \sum_{\substack{x\in L^2/\Gamma_{\frakM}\\ Q(x)= T}} 1\\
&\=\frac{1}{2}\vol(\bar K) \cdot | R_T|.
\end{align*}
Here we have also used that the stabilizer $\SO(V)(\Q)_{x}$ of any $x\in V^2(\Q)$ with $Q(x)=T$ is trivial and that $\{\pm 1\}\subset \Gamma_{\mathfrak{M}}$ acts trivially on $V(\Q)$.
Inserting the latter formula into \eqref{eq:p01}, we obtain
\begin{align}
\label{eq:p1}
\prod_{p< \infty} W_{T, p}(1, 0, \lambda(\varphi_{L_p})) \= \frac{\vol(\bar K)}{ 2C_\infty(V,\alpha,\beta,\psi)}  \cdot | R_T|.
\end{align}
According to equations (5.3.46) and (5.2.11) of \cite{KRY-book}, the constant on the right hand side is given by 
\begin{align*}
C_\infty(V,\alpha,\beta,\psi) \= 2^{-1} C_\infty(V) \=  2^{1/2}\gamma_\infty(V)^2 D(B)^2.
\end{align*}
As our choice of~$w$ agrees with the one in \cite{Ku1}, the constant $\gamma_\infty(V)^2$ here is~1, and differs from the one in \cite{KRY-book} by a factor of $-1$ (see the paragraph above Theorem~5.2.7 on page~116 of \cite{KRY-book}). 
Moreover, by Lemma 5.3.9~(iii) of \cite{KRY-book} and \eqref{eq:volumeformula}, we have 
\[
\vol(\bar K) \= \pi^{-1} D(B)^2 \vol(X_{\frakM},d\mu)^{-1}. 
\]
In fact, this is first seen if $\frakM=1$, and then a simple index argument shows that it also holds for general~$\frakM$. Consequently, we get
\[
\frac{C_\infty(V,\alpha,\beta,\psi)}{ \vol(\bar K)} \=  
2^{1/2}\pi \vol(X_{\frakM},d\mu).
\]
Putting this into \eqref{eq:p1} concludes the proof of the theorem.
\end{proof}
\par
\begin{remark}
The formula of Theorem~\ref{thm:posdeg2} can also be deduced from the global Siegel-Weil formula, together with a comparison of asymptotics as $v$ tends to infinity (in a suitable direction). Since this argument aligns nicely with the one given in \cite{KS} we briefly explain it here in the case when $T$ is positive definite.
\par
Combining \eqref{eq:f2} and \eqref{eq:IT0}, we find
\begin{align*}
\frac{2}{\vol(X_\mathfrak{M})}I_T(v,1) &\=\frac{2}{\vol(X_\mathfrak{M})}(\det v)^{1/2} e^{2\pi  \tr Tv }\sum_{\substack{x\in R_T}} I(x,v,1)\\
&
\= (\det v)^{-\frac{1}{4}} W_{T, \infty}(m(a), 0, \Phi_\infty^{1/2})\cdot \prod_{p<\infty} W_{T, p}(1, 0, \lambda(\varphi_{L_p})).
\end{align*}
We write $T={}^tBB$ with $B\in \GL_2^+(\R)$ as before and specialize to 
$v=r B^{-1}\,{}^t\!B^{-1}$ with positive real $r$. In particular, we consider the archimedian Whittaker function at the argument $m(\sqrt{r}B^{-1})$.  Applying Proposition~\ref{prop:Ipos}, we find for $x\in R_T$ that 
\[
I(x, v,1) \= I(x,r B^{-1}\,{}^t\!B^{-1},1) \= I(e,r1_2,1).
\]
On the other hand, by \cite[Lemma~4.2]{BY}, we have 
\[
W_{T,\infty}(m(\sqrt{r}B^{-1}), 0, \Phi_\infty^{1/2}) \= |B|^{-3/2} W_{1,\infty}(m(\sqrt{r}1_2), 0, \Phi_\infty^{1/2}).
\]
Inserting these identities, we obtain 
\begin{align}
\label{eq:asystart}
&\phantom{\=} \frac{2}{\vol(X_\mathfrak{M})}  |R_T| e^{4\pi  r }I(e,r 1_2,1)\\
\nonumber
& \= r^{-3/2} \prod_{p<\infty} W_{T, p}(1, 0, \lambda(\varphi_{L_p}))\cdot W_{1, \infty}(m(\sqrt{r}1_2), 0, \Phi_\infty^{1/2}).
\end{align}
From the definition of $I(e,r 1_2,1)$ one can derive the asymptotic behavior
\begin{align}
\label{eq:Iasy}
I(e,r 1_2,1) \= \frac{1}{4r}e^{-8\pi r}(1+O(r^{-1}))
\end{align}
as $r\to \infty$. Moreover, using similar arguments as in \cite[Section~4]{BY}, it can be shown that 
\begin{align}
\label{eq:Winftyasy}
 W_{1, \infty}(m(\sqrt{r}1_2), 0, \Phi_\infty^{1/2}) \= \sqrt{2} \pi \sqrt{r}e^{-4\pi r}(1+O(r^{-1}))
\end{align}
as $r\to \infty$. Plugging \eqref{eq:Iasy} and \eqref{eq:Winftyasy} into \eqref{eq:asystart} and taking the limit $r\to \infty$, we derive
\[
\frac{1}{2\vol(X_\mathfrak{M})}  |R_T| \=  \sqrt{2} \pi \prod_{p<\infty} W_{T, p}(1, 0, \lambda(\varphi_{L_p})),
\]
which yields the desired formula for $|R_T|$.
\end{remark}

\subsection{Non-archimedian Whittaker functions} 

Let  $p$ be a finite prime and $T\in \Sym_2(\Z_p)$ non-singular. For $p$ odd, or for $p=2$ with $\det(T) \in 2\Z_2$ and $T/2 \not\in\Sym_2(\Z_2)$, the matrix $T$ is $\GL_2(\Z_p)$-equivalent to an elementary divisor matrix of the form
\begin{equation}
  \label{eq:elemdiv}
 T \sim \diag(\eps_1 p^{a_1},\eps_2 p^{a_2})
\end{equation}
with $0\leq a_1\leq a_2$ and $\eps_1,\eps_2\in \Z_p^\times$. The Hasse invariant of $T$ is  
\begin{equation}
  \label{eq:hasseT}
\epsilon_{p}(T) \,:=\, (\eps_1 p^{a_1},\eps_2 p^{a_2})_p.
\end{equation}
For the computation of non-archimedian Whittaker functions, the following result of Kudla (see \cite[Proposition~1.3]{Ku1}) will be useful. 
\par
\begin{proposition}
\label{prop:Ku1.3}
Let $p$ be a finite prime and let $\chi$ be a quadratic character of $\Q_p^\times$. For every $T\in \Sym_2(\Q_p)$ with $\det(T)\neq 0$, there exists an (up to isometry) unique quadratic space $W$ over $\Q_p$ of dimension $3$ and quadratic character $\chi_W=\chi$ such that~$W$ represents~$T$. Its Hasse invariant is given by
\[
\epsilon_p(W) \= \chi(\det T) \epsilon_p(T),
\]
where $\epsilon_p(T)$ is the Hasse invariant of $T$.
\end{proposition}
\par
As before, let $V$ be the quadratic space defined in \eqref{eq:V} and let $V_p=V\otimes_\Q\Q_p$ be its base change to $\Q_p$. Using the basis $e_1,e_2,e_3$ of~\eqref{eq:e1e2}, it is easily seen that the quadratic character $\chi_{V_p}$ is the trivial character and that  
\begin{align*}
\epsilon_{p}(V_p) \= \begin{cases}
+1,& \text{if $p\nmid D(B)$,}\\
-1,& \text{if $p\mid D(B)$.}
\end{cases}
\end{align*}
\par
\begin{corollary}
Let $p$ be a finite prime and let  $T\in \Sym_2(\Q_p)$ with $\det(T)\neq 0$.
The space $V_p$ represents $T$ if and only if 
\begin{align*}
\epsilon_{p}(T) \= \begin{cases}
+1,& \text{if $p\nmid D(B)$,}\\
-1,& \text{if $p\mid D(B)$.}
\end{cases}
\end{align*}
\end{corollary}
\par
The first result on local Whittaker functions is the following.
\par
\begin{proposition}
Let $p$ be a finite prime and let  $T\in \Sym_2(\Q_p)$ with $\det(T)\neq 0$.
If $T$ is not represented by $V_p$, then $W_{T, p}(1, 0, \lambda(\varphi))=0$ for all Schwartz function $\varphi$ on $V_p$. 
\end{proposition}
\par
\begin{proof}
The assertion is special case of \cite[Proposition~1.4]{Ku1}.
\end{proof}
We now normalize $W_{T, p}$ by the factor $C_p(V) = \gamma_p(V)^2 |D(B)|_p^2|2|_p^{3/2}$ and $(1-p^{-2})$ to define
\begin{equation}
  \label{eq:tW}
  \tilde W_{T, p}(\varphi_{}) \,:=\, \frac{W_{T, p}(1, 0, \lambda(\varphi_{}))}{C_p(V)(1-p^{-2})}.
\end{equation}
It is easy to see that $\prod_{p < \infty} C_p(V) = C_\infty(V)^{-1}$.
When $\varphi = \cha(L_p)$ for a lattice $L_p \subset V_p$,  Lemma~5.7.1 in \cite{KRY-book} gives us 
\begin{equation} \label{eq:tW1}
  \tilde W_{T, p}(\varphi_{})
  \= |2|_p^{-1} \frac{|D(B)|_p^{-2}}{|L_p^\vee/L_p|} \frac{\alpha_p(T, L_p)}{1-p^{-2}},
\end{equation}
where $\alpha_p(T, L_p)$ is the local density defined in Equation~(1.3) in \cite{Ya}. 
\par
From now on, we will take $L \subset L_\frakM$ the lattice in \eqref{eq:L}, which has index 4 in $L_\frakM = V \cap \calO_\frakM$ for the Eichler order $\calO_\frakM$ of level $\frakM$, and denote $L_p := L \otimes \Z_p$. 
As in \cite{Ku1} and \cite{KRY-book} we employ results of Kitaoka and Yang to compute the non-archimedian local Whittaker functions more explicitly.
\par
\begin{theorem}
\label{thm:locwhitt1}
Let $p$ be a prime not dividing $2D(B)\frakM$, and let $T\in \Sym_2(\Q_p)$ be invertible.

(i) If $ T\notin \Sym_2(\Z_p)^\vee$, then  $W_{T, p}(1, s, \lambda(\varphi_{L_p}))=0$.

(ii) Assume that  $T\in \Sym_2(\Z_p)^\vee$ and denote its elementary divisors as in \eqref{eq:elemdiv}. If $T$ is represented by $V_p$ (which means $\epsilon_p(T)=1$), then 
\begin{align*}
\tilde W_{T, p}(\varphi_{L_p}) &\=
\begin{cases} 
2\sum_{j=0}^{\frac{a_1-1}{2}}p^j &\text{if $a_1$ is odd,}\\[1ex]
2\sum_{j=0}^{\frac{a_1}{2}-1}p^j+p^{\frac{a_1}{2}}\sum_{j=0}^{a_2-a_1} (\eps_1,p)_p^j &\text{if $a_1$ is even.}
\end{cases}
\end{align*}

(iii) If  $T\in \Sym_2(\Z_p)^\vee$ is not represented by $V_p$ (which means $\epsilon_p(T)=-1$), then  
\[
W_{T, p}(1, 0, \lambda(\varphi_{L_p}))\=0.
\]
\end{theorem}
\par
\begin{proof}
The assertion is a special case of \cite[Proposition~8.1]{Ku1} (see also \cite[Proposition~8.6]{Ya}) and \cite[Proposition~A.6]{Ku1}, where we have to put $\det(S)=\det(S_0)=-1$, $X=1$, and $r=0$.
\end{proof}
\par
Now we turn to the case of primes dividing the discriminant of $B$.
\par
\begin{theorem}
\label{thm:locwhitt2}
Let $p \mid D(B)$ be an odd prime and let $T\in \Sym_2(\Q_p)$ be invertible.

(i) If $ T\notin \Sym_2(\Z_p)^\vee$, then  $W_{T, p}(1, s, \lambda(\varphi_{L_p}))=0$.

(ii) Assume that  $T\in \Sym_2(\Z_p)^\vee$. If $T$ is represented by $V_p$ (which means $\epsilon_p(T)=-1$), then 
\begin{align*}
\tilde W_{T, p}(\varphi_{L_p}) &\=  \frac{2p}{1-p^{-1}}.
\end{align*}

(iii) If  $T$ is not represented by $V_p$ (which means $\epsilon_p(T)=1$), then  
$\tilde W_{T, p}(\varphi_{L_p})=0$.
\end{theorem}
\par
\begin{proof}
  Since $p$ is odd, we can use \eqref{eq:latt-size} and \eqref{eq:latt-L} to obtain $|2|_p^{-1} \frac{|D|_p^{-2}}{|L_p^\vee/L_p|} = 1$. Then applying Theorem~5.7.2 in \cite{KRY-book} to $\alpha_p(T, L_p)$ finishes the proof. 
\end{proof}
\par
Now for $p \mid 2\frakM$, we define the following quantity
\begin{equation}
  \label{eq:betap}
  \beta_p(T) :=
  \begin{cases}
    \frac{\tilde W_{T, 2}}{2^{3 (1- (D \bmod 2))}} & p = 2 \nmid \frakM,\\
  \frac{1+p^{-1}}2 p^e \tilde W_{T, p}(\varphi_{L_p}) & \nu_p(M) = e \ge 1.
  \end{cases}
\end{equation}
This has the advantage of simplifying the expression for $|R_T|$ in the following corollary.
\par
\begin{corollary} \label{cor:sumfinal}
For integral, non-singular~$T$ the number of
equivalence classes of lattice pairs with intersection matrix~$T$ is given by
\begin{equation}
  \label{eq:RT1}
  |R_T| \=
2^{\omega(\frakM) + 1 - (D \bmod 2)}   \prod_{2 \neq p \mid D} (1 - \epsilon_p(T))
   \prod_{p \mid 2\frakM} \beta_p(T)
\prod_{p \mid  \det(T),~ p \nmid 2D\frakM} \tilde W_{T, p}(\varphi_{L_p}).
\end{equation}
In particular, $|R_T|$ is zero unless $\epsilon_p(T) = -1$ for all odd prime $p \mid D$. In that case, if the entries of $T$ has no common odd factor, then
 $2\nmid \nu_p(\det(T))$ and $(\varepsilon_1, p)_p = -1$ for all odd $p \mid D$ and we have
\begin{equation}
  \label{eq:RT2}
  |R_T| \= 2^{\omega(D\frakM) } 
  \prod_{p \mid 2\frakM} \beta_p(T)
  \prod_{\substack{p \mid \det(T)\\ p \nmid 2D\frakM}} \sum_{j = 0}^{\nu_p(\det(T))} (\varepsilon_1, p)_p^j.
\end{equation}
\end{corollary}
\par
\begin{proof}
  Starting with \eqref{eq:RT}, we substitute in the volume from \eqref{eq:volumeformula} and the normalized $\tilde W_{T, p}$ in \eqref{eq:tW} to obtain
  \begin{equation}
    \label{eq:RT3}
    \begin{split}
          |R_T|
    &\=
\frac{\pi}3 \,D(B) \prod_{p | D(B)} \Big(1 - \frac1p  \Big)
\cdot \frakM \prod_{p | \frakM} \Big(1 + \frac1p  \Big)
\frac{      2^{1/2} \pi  \gamma_\infty(V)^2}{C_\infty(V) \zeta(2)}
      \prod_{p< \infty} \tilde W_{T, p}( \varphi_{L_p})\\
    &\=  D(B)^{-1}
       \prod_{p | D(B)} \Big(1 - \frac1p  \Big)
\cdot \frakM \prod_{p | \frakM} \Big(1 + \frac1p  \Big)
      \prod_{p< \infty} \tilde W_{T, p}( \varphi_{L_p}).
    \end{split}
  \end{equation}
  If $p \nmid 2D(B)\frakM \det(T)$, then $T \in \Sym_2(\Z_p)$ is $\GL_2(\Z_p)$-equivalent to $\diag(\varepsilon_1, \varepsilon_2) \in \GL_2(\Z_p)$ and Theorem \ref{thm:locwhitt1} gives us $\tilde W_{T, p}(\varphi_{L_p}) = 1$. 
  When $2 \neq p \mid D(B)$, Theorem \ref{thm:locwhitt2}  gives us $\tilde W_{T, p}(\varphi_{L_p}) = (1-\epsilon_p(T)) \frac{p}{1-p^{-1}}$.
  Substituting these into \eqref{eq:RT3} gives us \eqref{eq:RT1}.
\par
Let $p$ be an odd prime. 
The condition that the entries of~$T$ have no common odd prime factor implies that $T$ is  $\GL_2(\Z_p)$-equivalent to $\diag(\varepsilon_1, \varepsilon_2p^a)$ for $\varepsilon_i \in \Z_p^\times$ and $t = \nu_p(\det(T))$.
If $p \mid D$, then $1-\epsilon_p(T) = 1-(\varepsilon_1, p)_p^t$ is non-zero if and only if $2 \nmid t$ and $(\varepsilon_1, p)_p = -1$. 
If $p \nmid D$, then Theorem~\ref{thm:locwhitt1} gives us $\tilde W_{T, p}(\varphi_{L_p}) =
\sum_{j = 0}^{a} (\varepsilon_1, p)_p^j$.
Putting these into~\eqref{eq:RT1} yields~\eqref{eq:RT2}. 
\end{proof}
Now we can use the formula for $\alpha_p$ in \cite{Ya, Ya2} to give an explicit formula for~$\beta_p$. 
\begin{theorem}
  \label{thm:locwhitt3}
  For an odd prime $p \mid \frakM$ with $e := \nu_p(\frakM) \ge 1$ and
  $T = \kzxz{\varepsilon_1p^{a_1}}{}{}{\varepsilon_2p^{a_2}} \in \Sym_2(\Q_p)$ non-singular with $a_2 \ge a_1$, let $\beta_p(T)$ be defined as in \eqref{eq:betap}. 
\par
If $T \not\in\Sym_2(\Z_p)$ or  $\epsilon_p(T) = -1$, then $\beta_p(T) = 0$.
Otherwise, we have
 \begin{equation}
   \label{eq:betap1}
   \beta_p(T) \=
   \begin{cases}
\frac{1+v_0}4 p^{a_{1}/2} (p^{ e_0} + p^{e-1-e_0})(a_2-a_1) + \frac{p+1}2 \sum_{j=0}^{a_1 -e} p^{j_0+e-1}     & 2 \mid a_1, a_1 \ge e-1,\\
\frac{1+v_0}2 p^{a_1}\max(a_2-e+1, 0)     & 2 \mid a_1, a_1 \le  e-1,\\
\sum_{j=1}^{a_1-e+1}p^{j_0 +e-1}     & 2 \nmid a_1,
   \end{cases}
 \end{equation}
where $v_0 := (\varepsilon_1, p)_p$, and $a_0 := \lfloor a/2 \rfloor$ for any $a \in \R$. 
\end{theorem}

\begin{proof}
  Combining equations \eqref{eq:latt-size}, \eqref{eq:tW1}, \eqref{eq:betap} and using that $p \nmid D(B)$, that $p \mid \frakM$, and  $p\neq 2$, we obtain
  $$
  \beta_p(T) \=
  \frac{p^{-e}}{2(1-p^{-1})}  {\alpha_p(T, L_p)}\,.
  $$
  From the formula for $\alpha_p$ in \cite[Theorem~7.1]{Ya}, we know that $\beta_p$ is a polynomial in $p^i$ for $i \le a_1+1$. Furthermore, each coefficient is a polynomial in $a_i, (\varepsilon_i, p)_p$ and $(-1, p)_p$.  We have programmed this formula. After substituting various $p$'s, $a_i$'s and $\varepsilon_i$'s, we obtained \eqref{eq:betap1}.
\end{proof}

\begin{theorem}
  \label{thm:locwhitt4}
    Let $p =2$ and $T = \kzxz{\varepsilon_1}{}{}{2^t\varepsilon_2}$ with $\varepsilon_i \in \Z_2^\times$ and $t \ge 1$. 
If $\epsilon_2(T) = \epsilon_2(V_2)$, then
    \begin{equation}
      \label{eq:beta2}
      \beta_2(T) \=
      \begin{cases}
         \max(t-e-1, 0)& \text{if }\varepsilon_1 \equiv 1 \bmod8,\\
        1& \text{if }\varepsilon_1 \equiv 5 \bmod8 \text{ and }2 \nmid \frakM,
      \end{cases}
    \end{equation}
    where $e = \nu_2(\frakM) \ge 0$. Otherwise, $\beta_2(T) = 0$.
\end{theorem}

\begin{remark}
  \label{rmk:p=2}
  Note that $D_1 \equiv 1, 5 \bmod8$ as $D_1 \equiv 1\bmod4$. 
  Also we have  $  \varepsilon_1 \equiv D_1 + 2^t \bmod8$.
  So when $t \ge 2$, we have $\epsilon_2(T) = -1$ if and only if $2\nmid t$ and $\varepsilon_1 \equiv D_1 \equiv 5 \bmod8$. 
\end{remark}
\begin{proof}
  Using equations \eqref{eq:latt-size}, \eqref{eq:latt-L}, \eqref{eq:tW1} and \eqref{eq:betap}
  gives us
  $$
  \beta_2(T) =
\alpha_2(T, L_2) \cdot
  \begin{cases}
    2^{-2-3(1- (D \bmod 2))}3^{-1} & 2 \nmid \frakM,\\
    2^{-4-e} & 2\mid \frakM.
  \end{cases}
  $$
  Now as in the proof of Theorem \ref{thm:locwhitt3}, we
  specialize Theorem 4.4 in \cite{Ya2} to get the result.
\end{proof}
\par
\begin{proof}[Proof of Theorem~\ref{thm:rickardscount}]
Thanks to the correspondence in Proposition~\ref{prop:VecOptEmbHyp}
our task is to translate between pairs of primitive lattice vectors in the
lattice~${L}$ and arbitrary pairs of vectors in the lattice~$L$ as counted 
by~$R_T$. By hypothesis of the theorem, the entries of~$T$ have no common odd factor.
The statement about $|R_T|$ being zero in Corollary~\ref{cor:sumfinal} implies
{\bf Cases~iv) and~v)}, and  the $N_p = 0$ part of {Case~0)} when $p \mid D(B)$.
For $p \mid \frakM$ in Case~0), we get $\beta_p(T) = 0 = N_p$ from Theorem \ref{thm:locwhitt3}. 
For the $N_p = 1$ part, this follows from Theorems \ref{thm:locwhitt1} and  \ref{thm:locwhitt4} for odd and even $p$ respectively. 
This finishes {\bf Case~0)}. 
\par
We may assume now that the prime~$p$ under consideration divides~$\det(T/4)$. We may moreover assume that $p \nmid D_1$, i.e., that $R_T$ only counts pairs $x=(x_1,x_2)$ with $x_1$ primitive. Together with the condition $p \mid \det(T)$, we can locally conjugate $T$ to $\kzxz{\varepsilon_1}{}{}{p^t\varepsilon_2}$ with $\varepsilon_i \in \Z_p^\times$, $t \ge 1$ for odd $p$ and $t \ge 3$ for $p = 2$. In this case, we have
\begin{equation}
  \label{eq:invs}
  t \= \nu_p(\det(T)), \quad \varepsilon_1 \,\equiv\, D_1 \bmod p^{r_p}
\end{equation}
for every prime $p$, where $r_p = 1$ for $p \neq 2$ and $r_2 = 3$. This in particular implies that $(\varepsilon_1, p)_p = (D_1, p)_p$ for every prime $p$. Finally, since $D_1 \equiv 1\bmod4$, we have $(D_1, 2)_2 = (\varepsilon_1, 2)_2 = \pm 1$ if and only if $D_1 \equiv \varepsilon_1 \equiv 7 \pm 2 \bmod8$. 
\par
Before proceeding with the proof, notice that if $\nu_p(D_2) \geq 2$, then 
$p|b$, since $p | \det(T/4)$.  Consider the matrix
\bes T' \= \left(\begin{matrix} D_1 & b/p \\ b/p & D_2/p^2 \end{matrix} \right)\,.
\ees
Since $\nu_p(\det T') = t-2$, any $(x_1',x_2')$ contributing to $R_{T'}$ gives an
element $(x_1', px_2')$, contributing to $R_T$.
Suppose $e = \nu_p(\frakM) < \nu_p(\det(T/4)) = t$ for odd $p$.
Then $\frac{\tilde W_{T', p}(\varphi_{L_p})}{\tilde W_{T, p}(\varphi_{L_p})}$ of the pairs in $R_T$ are imprimitive ones coming from $R_{T'}$.
\par
For {\bf Case i)}, suppose $\nu_p(D_2) \le 1$.
If $p \nmid 2D(B)\frakM$, then $\tilde W_{T, p}(\varphi_{L_p})  = 1$ by Theorem~\ref{thm:locwhitt1}, which gives $N_p = 1$.
If $p = 2 \nmid D(B)\frakM$, then $(D_1, 2)_2 = -1$ gives $\varepsilon_2 \equiv 5 \bmod8$ and Theorem~\ref{thm:locwhitt4} yields $N_2 = \beta_2 = 1$.
If $p \mid \frakM$, then $N_p = \beta_p(T) = 0$ by Theorems~\ref{thm:locwhitt3} and \ref{thm:locwhitt4}.
Suppose $\nu_p(D_2) \ge 2$, then the above argument about primitive vectors show that $N_p = 0$. 
\par
For {\bf Case ii)},
 $\epsilon_p(T) = (\varepsilon_1, p)_p = (D_1, p) = -1$ is different from $\epsilon_p(V_p) = 1$ as $p \nmid D$. Therefore, $\tilde W_{T, p}$, hence $\beta_p(T)$, is 0 for $p \nmid D$ by Theorems~\ref{thm:locwhitt1}, \ref{thm:locwhitt3} and \ref{thm:locwhitt4}.
\par
In  {\bf Case iii)}, 
those odd primes $p|D$ are not explicitly listed in~\eqref{eq:RT3},
thus gives a local factor 1 for $(\varepsilon_1,p)_p = - 1$.
Similarly for $2 = p \mid D$, the condition $(\epsilon_1, D_1)_2 = -1$ implies $\varepsilon_1 \equiv D_1 \equiv 5\bmod8$, and Theorem~\ref{thm:locwhitt4} gives $N_p = \beta_p(T) = 1$. Reasoning by primitivity completes the situation for $\nu_p(D_2) \ge 2$. 
\par
For {\bf Case~vi)}, the pairs with $\nu_p(D_2) \leq 1$ must be primitive, so the claim is clear in that subcase for $p \nmid 2\frakM$ by Theorem \ref{thm:locwhitt1}, for $2 \neq p \mid \frakM$ by Theorem \ref{thm:locwhitt3}, and for $p = 2$ by Theorem \ref{thm:locwhitt4}. When $\nu_p(D_2) \ge 2$ and $e < t$, the ratio
\bes
\frac{\tilde W_{T', p}(\varphi_{L_p})}{\tilde W_{T, p}(\varphi_{L_p})} = \frac{t-1-e}{t+1-e}
\ees
by Theorems \ref{thm:locwhitt1}, \ref{thm:locwhitt3}, and \ref{thm:locwhitt4} for $p \nmid 2 \frakM$, for $2 \neq p \mid \frakM$ and for $p = 2$ respectively. This proves that $N_p = 2$.  If $e = t$, the set $R_{T'}$ is empty and all vectors in $R_T$ are primitive at $p$. The quantity $N_p$ is 1 by the three results above. If $e \ge t+1$, then both $R_T$ and $R_{T'}$ are empty.
This completes the proof.
\end{proof}

\begin{appendix}

\section{Casimir elements acting on polynomial Fock space} \label{app:Fock}

We briefly recall the infinitesimal Weil representation acting on the polynomial Fock space~$\mathcal{S}$ from the appendix \cite{FM06}, specialized to the case $\mathcal{S} \subset S(V(\R)^2)$ of two copies of a quadratic space~$V$ of signature $(p,q)=(2,1)$. Our goal is to determine the image of Casimir elements in both  the orthogoal and symplectic group under the Weil representation and in particular compute their action on the element $1$ that corresponds to the standard Gaussian. 
\par 
Following \cite{FM06} we let $W$ be a two-dimensional standard symplectic space and use their computation of the Weil representation~$\omega$ of $\fraksp(V \otimes W)$ on
$\mathcal{S} = \C[z_{11}, z_{12},z_{21}, z_{22},z_{31}, z_{32}]$
restriced to the image of the inclusion $\frako(V) \times \fraksp(W)$. 
\par
Lemma~A.01 in loc.\ cit.\ gives the Weil representation in terms of $X_{rs} = v_r \wedge v_s$ where $\{v_i\}_{i=1,2,3}$ is an orthonormal basis of~$V$, considered
in the orthogonal group via the isomorphism
\bes
\rho: \bigwedge^2 V \to \frako(V), \qquad \rho(v_r \wedge v_s)(w) \= (w,v_s)v_r - (w,v_r)v_w\,.
\ees
The standard generators of $\Lie(H(\R)) \cong \mathfrak{sl}_2$ acting via conjugation on~$V$ are mapped to
\bes
H =  \left(\begin{smallmatrix} 1&0\\0&-1 \end{smallmatrix}\right) \mapsto -X_{23}, \quad 
E =  \left(\begin{smallmatrix} 0&1\\0&0 \end{smallmatrix}\right) \mapsto (X_{12}-X_{13})/2, \quad
F =  \left(\begin{smallmatrix} 0&0\\1&0 \end{smallmatrix}\right) \mapsto (X_{12}+X_{23})/2 
\ees
in this basis of $\frako(V)$, if we use $v_i = e_i/\sqrt{2}$ with~$e_i$ as in~\ref{eq:e1e2} and~\eqref{eq:e3}. Since the Casimir element that gives our (positive definite) Laplacian is
\be
C_{\frako} \=  - \Bigl(\frac12 H^2 + EF + FE\Bigr)
\ee
we can combine the information and compute:
\par
\begin{lemma} \label{le:CasimirO}
The Casimir element for $\frako(V)$ acts on the standard Gaussian as
\be
C_{\frako}(1) \= Y + 2, \qquad Y \= Y_+ \, Y_- / 8 \pi
  \ee
  where
\be
Y_\pm \= z_{11}z_{31} + z_{12}z_{32} \,\pm\, i (z_{21}z_{31} + z_{22}z_{32})\,.
\ee
\end{lemma}
In fact, $C^n_{\frako}(1)$ is always a polynomial in~$Y$, which follows from the
recursion
\be
C_\frako(Y^n) \= Y^{n+1} + (4n^2 + 5n+2)Y^n - 2n^3(2n+1) Y^{n-1}
\qquad (n \geq 0)\,.
\ee
\par
\medskip
On the other hand \cite{Maurischat} gives a convenient expression of Casimir elements of the symplectic group. A basis of $\fraksp(W)$ is given by
\ba
E_{\pm jk} &\= \left(\begin{matrix} X_{jk} & \pm iX_{jk}\\ \pm iX_{jk} &-X_{jk}
\end{matrix}\right), \qquad \text{where} \quad X_{jk} = (e_{jk} + e_{kj})/2 \\
B_{\pm jk} &\= \left(\begin{matrix} A_{jk} & -iX_{jk}\\ iX_{jk} & A_{jk}
\end{matrix}\right), \qquad \text{where} \quad A_{jk} = (e_{jk} - e_{kj})/2,
\ea
where $e_{jk}$ is the $2\times 2$ matrix whose unique non-zero entry is a one at the $jk$-the entry. Casimir elements $C_i$ in degree~$2i$ of the symplectic group are now given by the formal traces of products (in the enveloping algebra)
\ba \label{eq:SpCasimir}
  C_1 &\= \tr(E_+E_-) + \tr(E_-E_+) + \tr(BB) + \tr(B^*B^*),\\
  C_2 &\= \tr(E_+E_-E_+E_-) + \tr(E_-E_+E_-E_+) + \tr(BBBB) + \tr(B^*B^*B^*B^*)\\
  &      \quad + \sum_{\zeta \in Z_4} \tr(\zeta(E_+E_-BB))  + \tr(\zeta(E_-E_+B^*B^*)) - \tr(\zeta(E_+B^*E_-B)),
\ea
where we grouped the generator matrices in $2\times 2$-blocks as $E_\pm = (E_{\pm jk})_{j,k},~ B = (B_{jk})_{j, k}$, and $B^* = (B_{kj})_{j, k}$. The above generators of $\fraksp(W)$ are related to the generators $w_j' \circ w_k''$,
$w_j' \circ w_k''$, and $w_j' \circ w_k''$ used in \cite{FM06} by
\be
-2i E_{jk} \= w_j'' \circ w_k'', \quad 2i E_{jk} \= w_j' \circ w_k', \quad
2i B_{jk} = w_k' \circ w_j''\,.
\ee
Lemma~A.02 in loc.\ cit.\ gives the Weil representation of $\fraksp(W)$. Combining this we find with the help of a computer machine that
\begin{align*}
  C_1(1) &\= -4Y - 17, \qquad C_2(1) \= 8Y^2 + 90Y - \frac{35}4. 
\end{align*}
Putting this together with Lemma~\ref{le:CasimirO} we obtain:
\begin{proposition} The action of the Casimir elements for $\frako(V)$ and for
  $\fraksp$ on the element one are related by
\begin{equation}
  \label{eq:laplacian-commute}
  C_1(1) \= 4 C_\frako(1) - 25, \qquad
  C_2(1) \= 8 C_\frako(1)^2 + 58C_\frako(1) - \frac{627}4.
\end{equation}
\end{proposition}
\par

\section{Some special functions} \label{app:special}

This appendex summarizes properties of special functions, solutions to some standard differential equations, and in particular normalizations from various sources, most frequently relying on \cite{DLMF} and \cite{GR}.
\par
\medskip
\paragraph{\textbf{Hypergeometric functions}}
The hypergeometric function \cite[Section~15]{DLMF}
\be \label{eq:HGFct}
F(a,b,c;z) \= \sum_{i \geq 0} \frac{(a)_i(b)_i}{(c)_i} \, z^i\,,
\ee
where $(x)_i =\Gamma(x+i)/\Gamma(x)$ denotes the Pochhammer symbol, is the unique (up to scale)
regular solution near $z=0$ of the hypergeometric differential equation
\be \label{eq:HGDE}
z(1-u)\frac{d^2{w}}{dz^2} + (c-(a+b+1)z) \frac{d{w}}{dz}
- abw \= 0\,.
\ee
\par
\medskip
\paragraph{\textbf{Legendre functions}}
Legendre's $P$-function $P^{-\mu}_{\nu}(u)$ and $Q$-function are a basis of solutions of the second order linear differential equation \cite[Section~14.1]{DLMF}
\be \label{eq:LegendreDE}
\frac{d^2{\varphi}}{du^2} + \frac{2u}{u^2-1} \frac{d{\varphi}}{du}
- \Bigl(\frac{\nu(\nu+1)}{u^2-1} + \frac{\mu^2}{(u^2-1)^2} \Bigr) \varphi \= 0\,.
\ee
More precisely, $P^{-\mu}_{\nu}(u)$ is the solution which is regular for $\mu \geq 0$ at $u=1$. It is given on the interval $u \in (1,\infty)$ and $\nu \in \C$ as
\be \label{eq:LegP}
P^{-\mu}_{\nu}(u) \= \frac1{\Gamma(1+\mu)} \Bigl(\frac{u-1}{u+1}\Bigr)^{\mu/2}F(1+\nu, -\nu, 1+\mu, (1-u)/2)\,.
\ee
In fact by \cite[Eq.~14.8.7]{DLMF}
\bes
\lim_{u \to 1^+} P^{-\mu}_{\nu}(u) \sim \frac{1}{\Gamma(1+\mu)} \Bigl(\frac{u-1}{2}\Bigr)^{\mu/2}
\ees
for $\mu \neq -1,-2,-3,\dots$. The sign change in the upper index is given 
by \cite[Eq.~14.9.13]{DLMF}
\bes
 P^{-\mu}_{\nu}(u) \= \frac{\Gamma(1+\nu-\mu)}{\Gamma(1+\nu+\mu)}  P^{\mu}_{\nu}(u)
\ees
whenever $\nu$ is such that the $\Gamma$-function has no poles. We do not need the normalization of the $Q$-function $Q^\mu_\nu(u)$, but see \cite[Section~14.3]{DLMF}.
\par
\medskip
\paragraph{\textbf{Whittaker functions}}
The Whittaker differential equation is the second order differential equation
\cite[(13.14.1)]{DLMF} 	
\be \label{eq:WhitDE}
\frac{d^2{W}}{dy^2} + \Bigl(-\frac14 +\frac{\lambda}{y} + \frac{1/4-\mu^2}{y^2}  \Bigr) W \= 0.
\ee
There is a unique solution of this differential equation, the Whittaker function $W_{\lambda,\mu}(y)$, whose asymptotic behaviour is (see \cite[(13.14.21)]{DLMF}) 	
\be \label{eq:Whit}
W_{\lambda,\mu}(y) \,\sim\, e^{-y/2}y^\lambda \qquad (y \to \infty), 
\ee
see e.g.\ \cite[Chapter~9.220]{GR} or \cite[Section~13.14]{DLMF} for series representations and integral representations.
\par
There are various integral relations between Legendre and Whittaker functions,
for example we need
\be \label{eq:IntPgivesW}
\int_{w>0} P_{\nu}^0( \sqrt{1+w}) \frac{e^{-\beta w}}{ \sqrt{1+w}}\, dw \= \beta^{-3/4} e^{\beta/2} W_{-1/4,\nu/2+1/4}(\beta)\,.
\ee 
This integral is for instance evaluated in \cite[Chapter 7.146]{GR} (however citing incorrectly from \cite[Formula 180(8)]{ET1} with index $+1/4$). The above formula leads to the correct asymptotic behavior.)

\end{appendix}
\printbibliography

\end{document}